\PassOptionsToPackage{usenames, dvipsnames}{color}
\documentclass[a4paper,12pt,oneside,dvipsnames,reqno]{amsart} % reqno if desired

\usepackage{amsmath,amsfonts,amsthm,amssymb, mathrsfs}
\usepackage{mathdots}

\usepackage{xcolor}
\usepackage{graphicx}
\usepackage{cite}

\usepackage{bbm}
\usepackage{tikz-cd}

\usepackage{longtable}

\usepackage%[backref=page]
{hyperref}
\hypersetup{colorlinks=true, citecolor=PineGreen, linkcolor=RoyalBlue, linktoc=page,urlcolor=RoyalBlue}

\def\O{\operatorname{O}}

\usepackage[usenames,dvipsnames]{color}

\usepackage{geometry}
\theoremstyle{plain}
\newtheorem{theorem}{Theorem}[section] %[section] % to omit decimals
\newtheorem{lemmy}[theorem]{Lemma}
\newtheorem{prop}[theorem]{Proposition}
\newtheorem{cor}[theorem]{Corollary}

\theoremstyle{remark}
\newtheorem{rem}[theorem]{Remark}

\theoremstyle{definition}
\newtheorem{defn}{Definition}[section]
\newcommand{\A}{\mathbb{A}}

\newcommand{\R}{\mathbb{R}}

\newcommand{\C}{\mathbb{C}}

\newcommand{\Z}{\mathbb{Z}}

\newcommand{\Q}{\mathbb{Q}}

\newcommand{\GL}{\operatorname{GL}}

\newcommand{\SL}{\operatorname{SL}}

\newcommand{\Mat}{\operatorname{Mat}}

\newcommand{\U}{\operatorname{U}}

\newcommand{\N}{\mathbb{N}}

\renewcommand{\Re}{\operatorname{Re}}

\newcommand{\Hom}{\operatorname{Hom}}

\newcommand{\diag}{\operatorname{diag}}

\newcommand{\reg}{\operatorname{reg}}

\title{Local analysis of the Kuznetsov formula and the density conjecture}

\author{Edgar Assing}
\author{Valentin Blomer}
\author{Paul D. Nelson}

\address{Mathematisches Institut, Endenicher Allee 60, 53115 Bonn}
\email{blomer@math.uni-bonn.de}
\email{assing@math.uni-bonn.de}
\address{Aarhus University, Ny Munkegade 118
  DK-8000 Aarhus, Denmark}
\email{paul.nelson@math.au.dk}

\begin{document}

\begin{abstract}
  We prove Sarnak's spherical density conjecture for the principal congruence subgroup of ${\rm SL}(n, \mathbb{Z})$ of arbitrary level.  Applications include a complete version of Sarnak's optimal lifting conjecture for principal congruence subgroups of ${\rm SL}(n, \mathbb{Z})$, as well as a transfer of the density theorem to certain co-compact situations.  The main ingredients are new lower bounds for Whittaker functions and strong estimates for the cardinality of ramified Kloosterman sets. 
\end{abstract}


\thanks{The first two authors are supported in part by Germany's Excellence Strategy grant EXC-2047/1 - 390685813 and also   by SFB-TRR 358/1 2023 - 491392403. The second author is additionally supported in part by ERC Advanced Grant 101054336.  The third author is supported by a research grant (VIL54509) from VILLUM FONDEN.  Some work towards this paper was carried out while the third author was at the Institute for Advanced Study, supported by the National Science Foundation under Grant No. DMS-1926686}

\setcounter{tocdepth}{2} 	

\maketitle	

\section{Introduction}\label{sec:cnfm30cc81}

Sarnak's density hypothesis is a robust quantitative statement that can often replace the generalized Ramanujan conjecture in applications.  While proving the full Ramanujan conjecture remains out of the reach, there has been much recent progress towards the density hypothesis.  In this paper, we obtain a definitive result for principal congruence subgroups of ${\rm SL}_n( \mathbb{Z})$, extending the squarefree level case treated in \cite{AB}, as well as certain co-compact analogues.
% by removing the squarefreeness assumption required in \cite{AB}.
Experience has shown that the properties of congruence subgroups and their associated automorphic forms display quite different features depending on whether the level is squarefree (``horizontal direction'') or a high $p$-power (``vertical direction'') and the analysis often requires a different set of tools. The present case is no exception, and in treating the horizontal and vertical direction simultaneously, we will face new representation-theoretic and combinatorial challenges.

% In this paper, we obtain a definitive result for principal congruence subgroups of ${\rm SL}_n( \mathbb{Z})$.
% % continue to make progress in this direction. 
% Most importantly, we remove the squarefreeness assumption required in \cite{AB}.  This opens the door for new applications and allows us to transfer the density theorem to certain co-compact quotients using the Jacquet--Langlands correspondence.

% \pn{I wonder if the last three sentences above could be abridged to simply:

% In this paper, we obtain a definitive result for principal congruence subgroups of ${\rm SL}_n( \mathbb{Z})$ by removing the squarefreeness assumption required in \cite{AB}.

% On the one hand, the mention of cocompact quotients feels a bit random, and maybe also a bit confusing: aren't there also some depth zero situations that can be transferred?  On the other hand, it's nice to say something other than ``yeah, we're removing some super-technical hypothesis''.
% }

\subsection{Statement of the main results}\label{sec:cnfm30cesk}
In what follows, we work with a level $q$ (a natural number) and a place $v$ of $\mathbb{Q}$ not dividing $q$, and consider automorphic forms $\varpi$ contributing to the spectral decomposition of $L^2(\Gamma(q) \backslash {\rm SL}_n(\mathbb{R})/{\rm SO}_n(\mathbb{R})),$ where $\Gamma(q) := \ker[\mathrm{SL}_n(\Z)\to \mathrm{SL}_n(\Z/q\Z)]$ is the principal congruence subgroup.  We assume that $\varpi$ is an eigenfunction of the invariant differential operators and, when $v$ is finite, of the Hecke operators defined at $v$.  In either case, the eigenvalues of $\varpi$ at $v$ are described by a multiset of Langlands parameters $\{\mu_{\varpi,v}(j) \mid 1 \leq j \leq n\}$, which we normalize to be purely imaginary when $\varpi$ is tempered.  We set $\sigma_{\varpi,v} := \max_j |\Re \mu_{\varpi,v}(j)|$.  In the following, the terms ``cusp form'' and ``automorphic form'' will include the property of being an eigenfunction at $v$ and $\infty$.  Given a finite family $\mathcal{F}$ of such $\varpi$ and $\sigma \geq 0$, we define
\begin{equation*}
  N_v(\sigma, \mathcal{F}) := \#\{ \varpi \in \mathcal{F} \mid  \sigma_{\varpi,v} \geq \sigma\},
\end{equation*}
which quantifies the failure of temperedness at $v$ in $\mathcal{F}$.

The Ramanujan conjecture predicts that if $\mathcal{F}$ consists of cusp forms, then $N_v(\sigma, \mathcal{F}) = 0$ when $\sigma > 0$.  This prediction does not generalize: the constant function $\varpi = 1$ satisfies $\sigma_{\varpi, v} = \frac{n - 1}{2}$ for every $v$, and for more general groups, it is known that $\sigma_{\varpi, v}$ can be positive even when $\varpi$ is cuspidal.  Sarnak's density hypothesis proposes the following more robust estimate
\begin{equation}\label{eq:cng2lgcnhy}
  N_v(\sigma, \mathcal{F}) \ll_{n,\epsilon} (\# \mathcal{F}  )^{1 - \frac{2 \sigma}{ n - 1} + \epsilon }
\end{equation}
for all reasonable families $\mathcal{F}$.  This may be understood as interpolating linearly between the trivial bound $\# \mathcal{F}$ and the upper bound $\O(1)$ for $\sigma = \frac{n - 1}{2}$, which is sharp when $\mathcal{F}$ contains a constant function.

We generalize and strengthen the main theorem from \cite{AB}, as follows: %to cover arbitrary moduli:
\begin{theorem}\label{th_mt}
  Let $M,\epsilon>0$.  Let $q\in \mathbb{N}$ be arbitrary.  Fix a place $v$ of $\mathbb{Q}$, and if $v = p$ is finite, assume that $p \nmid q$.  Let $\mathcal{F}_{\mathrm{cusp}} = \mathcal{F}_{\Gamma(q),\mathrm{cusp}}(M)$ be a maximal orthogonal collection of cusp forms $\varpi$ for $\Gamma(q) \subseteq \mathrm{SL}_n(\mathbb{Z})$ and with archimedean spectral parameter $\Vert \mu\Vert \leq M$.  There exists a constant $K$, depending only on $n$, such that for each $\sigma \geq 0$, we have
  \begin{equation*}
    N_v(\sigma,\mathcal{F}_{\mathrm{cusp}}) \ll _{v,\epsilon,n} M^K \cdot [\mathrm{SL}_n(\mathbb{Z})\colon \Gamma(q)]^{1-\frac{2\sigma}{n-1}(1+\frac{1}{n+1})+\epsilon}.
  \end{equation*}
\end{theorem}
\begin{rem}
  If $M$ is not too small, then $[\mathrm{SL}_n(\mathbb{Z})\colon \Gamma(q)] \asymp_{n, M} \# \mathcal{F}$ as $q \rightarrow \infty$, so Theorem \ref{th_mt} improves upon Sarnak's density hypothesis \eqref{eq:cng2lgcnhy} via the additional term $\frac{1}{n+1}$.  For $q$ prime, a slightly weaker improvement was obtained in \cite{BM}.
\end{rem}

The classification of the discrete spectrum by Moeglin and Waldspurger allows us to treat the residual spectrum inductively.

\begin{cor}\label{cor_disc}
  Let $M,\epsilon>0$ and let $q\in \mathbb{N}$ be arbitrary. Fix a place $v$ of $\mathbb{Q}$, and if $v = p$ is finite, assume that $p \nmid  q$. Let $\mathcal{F}_{\mathrm{disc}} = \mathcal{F}_{\Gamma(q),\mathrm{disc}}(M)$ be a maximal orthogonal collection of automorphic forms for $\Gamma(q) \subseteq  \mathrm{SL}_n(\mathbb{Z})$ appearing in the discrete spectrum with archimedean spectral parameter $\Vert \mu\Vert  \leq M$. There exists a constant $K$ depending only on $n$, such that
  \begin{equation}
    N_v(\sigma,\mathcal{F}_{\mathrm{disc}}) \ll _{v,\epsilon,n} M^K \cdot [\mathrm{SL}_n(\mathbb{Z})\colon \Gamma(q)]^{1-\frac{2\sigma}{n-1}+\epsilon}, \nonumber
  \end{equation}
  for $\sigma\geq	 0$.
\end{cor}
\begin{rem}
  The exponent in Corollary \ref{cor_disc} is sharp, in view of the contribution of the constant function to $\mathcal{F}_{\mathrm{disc}}$.
\end{rem}

By bootstrapping further, we may extend our estimates to the complete spectrum.  For the sake of applications, it is convenient to formulate this extension in terms of a weighted sum/integral over the full spectrum:

\begin{cor}\label{cor_full_spec}
  There exists a constant $K > 0$, depending only on $n$, with the following property.  Let $M, T > 1$, $q$ arbitrary, and suppose that $T \leq M^{-K} q^{n+1}$. Fix a place $v$ of $\mathbb{Q}$. If $v = p$ is finite, assume that $p \nmid q$.  Then
  \begin{equation*}
    \underset{\| \mu_{\varpi} \| \leq M}{\int_{\Gamma(q)}}  T^{2 \sigma_{\varpi,v}}  d\varpi  \ll_{v, \varepsilon, n} M^{K} q^{\varepsilon} \cdot   [\mathrm{SL}_n(\Z)\colon \Gamma(q)] .
  \end{equation*} 
\end{cor}
Here the left hand side denotes a combined sum/integral over all spectral components $\varpi$ in $L^2(\Gamma(q) \backslash {\rm SL}_n(\mathbb{R})/{\rm SO}_n(\mathbb{R}))$ with $\| \mu_{\varpi} \| \leq M$ that are (at least) eigenfunctions at $v$ and $\infty$; cf.\ \cite[(2.18)]{AB} for the notation. 
In particular, $d\varpi$ is the counting measure on the discrete part of the spectrum. 

\subsection{The methods}\label{sec:cnfm30cgoc}
Our results generalize those of \cite[Theorem~1.1]{AB} concerning the squarefree case.  The proofs given there require certain spectral and geometric estimates, but the methods relied crucially on the squarefreeness assumption: in the spectral arguments, to reduce to local computations that involve only \emph{depth zero} supercuspidals, and in the geometric arguments, to count the number of double cosets occurring in the definition of Kloosterman sums for congruence subgroups of prime level.  The extensions in this paper require new ideas, discussed in the following paragraphs, that we have attempted to develop robustly so as to be more broadly useful:
\begin{itemize}
\item Our representation-theoretic results should be useful in any application of the $\GL_n$ Kuznetsov formula to automorphic forms of given depth.
\item Our counting results provide a prototypical example of how to bound Kloosterman sets optimally when the congruence subgroup  does not exhaust the torus of $\mathrm{SL}_n(\mathbb{Z})$. 
\end{itemize}

Starting with the spectral side, the key estimate is an average lower bound for first Fourier coefficients, which reduces to a local problem concerning Bessel distributions associated to supercuspidal representations.  We indicate the main statement in Section~\ref{sec:cnh2efy71t} and give the proof in Section~\ref{sec_f}. 
% Understanding the general asymptotics of such distributions is a challenging problem in local harmonic analysis.  In the case of the principal congruence subgroup,
It turns out that the averaged first Fourier coefficient counts the multiplicity of certain regular types,
% regular $K$-types of a given conductor,
see Lemma~\ref{lm:red_to_mult} below for a precise statement.  Producing the desired lower bound thus amounts to showing that such regular types always exist.
% This is purely a representation theoretic problem, which is solved in Section~\ref{sec:ex_reg_ty}.
To solve this representation-theoretic problem requires deep results concerning the construction of regular representations over finite rings and their multiplicity in the Gelfand--Graev representation.  The proof is inspired by ideas coming from Kirillov's orbit method, which have been relevant for several recent advances in the analytic theory of automorphic forms (cf.\  \cite{NV}).

% \pn{The following two paragraphs could be tightened up.  In particular, the novelty should be conveyed more forcefully and specifically, if possible, I think}

The geometric side is estimated in Section~\ref{sec_g}.  It features sums of Kloosterman sums. The simplest bound for a Kloosterman sum is a bound for the underlying double cosets. This may be called the trivial bound  because it neglects savings from the additive characters, but it is far from trivial: we need to count the double cosets uniformly in the level, which is to say that we need to fully exploit the fact that the diagonal matrices in the principal congruence subgroup have large index in the torus of $\SL_n(\mathbb{Z})$; cf.\ the passage from \eqref{eq:triv_gen_ks} to \eqref{general_expectation}. 
For squarefree levels, an elementary   argument  in \cite{AB} based on  counting solutions to systems of congruences sufficed, but this fails even in the prime power case.  Instead, in the proof of Lemma~\ref{lem:kszeta} below, we organize the combinatorics systematically via certain $p$-adic integrals  to obtain the sharp bound \eqref{general_expectation} for the critical Weyl element $w_{\ast}$ uniformly for all moduli and all levels $q$. 

% Key to the treatment in \cite{AB} for squarefree moduli was an estimate for the size of the Kloosterman set given in \cite[Lemma~4.3]{AB}.  This estimate depends on the level $q$ and on the ($p$-adic valuation of the) moduli.  While it is sharp in the dependence on the level, it is very lossy in moduli dependence.  This was just sufficient for the case of squarefree levels, but is not enough for the general case.  Indeed, even in the prime power case this inaccuracy is not tolerable, so that we are forced to prove a refined counting result.  This is the content of Lemma~\ref{lem:kszeta}.  The bound provided there is essentially sharp. % and suffices for our geometric estimate.

% At the heart of the proof of \cite[Lemma~4.3]{AB} is a cumbersome but direct counting argument, and it seems hard to extract further savings in this way.  To circumvent this difficulty, we take a more systematic approach in the proof of Lemma~\ref{lem:kszeta}.  Using the Iwahori factorization of the principal congruence subgroup and some ideas borrowed from \cite{DR}, we can reduce the counting problem to the evaluation of a relatively concrete $p$-adic integral.  This integral is then evaluated by brute force, crucially exploiting that we only need to estimate the Kloosterman sets for one particularly nice Weyl element.

% \pn{emphasize key idea is to organize the combinatorics systematically via certain $p$-adic integrals, which is very difficult to see by hand, or whatever}

\subsection{Lower bounds for Whittaker functions and the orbit method}\label{sec:cnh2efy71t}
We have noted above that one of the two pillars of our proof is a lower bound for Fourier coefficients, established in Section~\ref{sec_f}.  Here we record one formulation of the main local result of that section, and indicate why it is natural from the perspective of the orbit method.

Let $F$ be a non-archimedean local field of characteristic zero, with ring of integers $\mathfrak{o}$ and maximal ideal $\mathfrak{p}$.  Let $\mathfrak{q} \subseteq \mathfrak{o}$ be an ideal.  Let $\psi$ be a unitary character of $F$, trivial on $\mathfrak{q}$ but not on $\mathfrak{p}^{-1} \mathfrak{q}$, and denote also by $\psi$ the character $u \mapsto \psi(\sum u_{i, i + 1})$ of the strictly upper-triangular subgroup $N \leq \GL_n(F)$.  (We caution that $\psi$ is normalized differently in Section~\ref{sec_f}.)  Let $\pi$ be a generic irreducible representation of $\pi$, realized in its $\psi$-Whittaker model.  It is easy to see that each $W \in \pi^{K(\mathfrak{q})}$ is supported on matrices whose Iwasawa $A$-component $\diag(a_1, \dotsc, a_n)$ satisfies $\lvert a_i \rvert \leq \lvert a_{i + 1} \rvert$.  On the other hand, there is no obvious symmetry argument for why $W$ should vanish at the identity.  We establish the following (see Remark~\ref{rem:cnh3r9n3n2}):

\begin{theorem}\label{theorem:cnh3r840lw}
  Let notation be as above.
  \begin{enumerate}
  \item\label{enumerate:cnh38jx2q0} Suppose the residual characteristic $p$ of $F$ satisfies $p > 2 n$.  If $\pi^{K(\mathfrak{q})} \neq 0$, then there exists $W \in \pi^{K(\mathfrak{q})}$ with $W(1) \neq 0$.
  \item In general, there is an ideal $\mathfrak{d} \subseteq \mathfrak{o}$, independent of $\pi$, such that if $\mathfrak{q} \subseteq \mathfrak{d}$ and $\pi^{K(\mathfrak{d}^{-1} \mathfrak{q})} \neq 0$, then there exists $W \in \pi^{K(\mathfrak{q})}$ with $W(1) \neq 0$.
  \end{enumerate}
\end{theorem}
We expect that the conclusion of assertion \eqref{enumerate:cnh38jx2q0} holds without the assumption concerning residue characteristic.  The case $\mathfrak{q} = \mathfrak{o}$ amounts to the well-known fact that spherical Whittaker functions do not vanish at the identity.

We now give some heuristic justification of Theorem \ref{theorem:cnh3r840lw} via the metaphor of microlocal analysis and the orbit method (glossing over all details).  One associates to $\pi$ a coadjoint orbit $\mathcal{O}_\pi$, which may be understood as a conjugacy class of matrices describing the eigenvalues for $\pi$ under the action of small group elements.  The Whittaker functional $W \mapsto W(1)$ corresponds in this picture to elements of that coadjoint orbit of the form, e.g., for $n = 3$,
\begin{equation*}
  \begin{pmatrix}
    \ast    & \ast & \ast \\
    T & \ast & \ast \\
    0 & T & \ast \\
  \end{pmatrix},
\end{equation*}
with $T$ a generator of $\mathfrak{q}^{-1}$ corresponding to $\psi$.  Our result is related to the fact that if $\mathcal{O}_\pi$ contains some matrix with entries no larger than $T$, then it contains one of the above shape.  This fact follows readily from the theory of rational canonical form, using that $\mathcal{O}_\pi$ consists of regular matrices.

The actual proof of Theorem~\ref{theorem:cnh3r840lw} involves a systematic application of the theory of types developed by Howe, Moy and others.  One can see shadows in that argument of the above sketch, with a key role played by regular matrices.

\subsection{Applications}\label{sec:cnfm30ci4o} The two primary arithmetic applications of the density theorem are the following counting and lifting results for matrices, which for $n=2$ go back to Sarnak--Xue \cite{SX}.  We refer to \cite{GK} for a general discussion, additional context and some history.

\begin{cor}[Uniform Counting]\label{cor:count}
  For any $q\in \N$ and any $\epsilon>0$, $R \geq 1$ we have
  \begin{equation}
    \#\{ \gamma\in \Gamma(q)\colon \Vert \gamma\Vert\leq R  \} \ll_{\epsilon, n} (Rq)^{\epsilon}\Big(\frac{R^{n(n-1)}}{q^{n^2-1}}+R^{\frac{n(n-1)}{2}}\Big).\nonumber
  \end{equation}
\end{cor}

The key point here is uniformity simultaneously in $q$ and $R$. 

\begin{cor}[Optimal Lifting]\label{cor:lift}
  For each $\epsilon>0$, there exists $\delta>0$ such that for every $q\in \N$, the number of matrices $\overline{\gamma}\in \mathrm{SL}_n(\Z/q\Z)$ that do not have a lift $\gamma\in \mathrm{SL}_n(\Z)$ with $\Vert \gamma\Vert \leq q^{1+\frac{1}{n}+\epsilon}$ is at most $\O_{\epsilon, n}(q^{n^2-1-\delta})$.
\end{cor}

Corollaries~\ref{cor:count} and ~\ref{cor:lift} generalize \cite[Theorem~5 and 6]{JK} to arbitrary $q$. See also \cite[Theorem~1.4 and 1.5]{AB} for conditional versions.

% Indeed the only point in \cite{JK} where $q$-squarefree is needed is when the density theorem from \cite{AB} is applied. 

% \begin{rem}
%   Let us put the exponent from Theorem~\ref{th_mt} into some perspective. First of all, if $n=2$ this agrees with the classical exponent achieved in \cite{Hum, Hux}. Secondly, one can always turn 
A density theorem can often be turned into a spectral gap. For instance, in the case of the Hecke congruence subgroup $\Gamma_0(q)$ with $q$ prime, the argument  goes as follows.  Suppose $\sigma$ is the real part of a Langlands parameter at $v$ of a cusp form contributing to $L^2_{\mathrm{cusp}}(\Gamma_0(q)\backslash \mathrm{SL}_n(\mathbb{R})/\mathrm{SO}_n(\mathbb{R}))$.  Then, by the density theorem from \cite{Bl}, we have
\begin{equation}
  1 \ll_{\epsilon, n, v} q^{1-\frac{4\sigma}{n-1}+\epsilon}\text{ for $q$ prime.}\nonumber
\end{equation}
This implies $\sigma\leq \frac{n-1}{4}$.  For $n=2$, this translates to Selberg's famous $\frac{3}{16}$-Theorem.  For $n=3$, this agrees with the Jacquet--Shalika bound, which says $\sigma\leq\frac{1}{2}$.  For all other $n$, the statement is not very interesting.

If we run the same argument for the principal congruence subgroup, then we can exploit that non-trivial eigenspaces have large multiplicity.  Indeed, the strong density hypothesis of Theorem~\ref{th_mt} gives
\begin{equation}
  q^{\frac{n(n-1)}{2}} \ll_{\epsilon, n, v} q^{(n^2-1)(1-\frac{2\sigma}{n-1}(1+\frac{1}{n+1}))+\epsilon}. \nonumber
\end{equation}
This translates again into $\sigma\leq \frac{n-1}{4}$.  In this sense, our saving is similar in strength to that in \cite{Bl}.
% , which agrees with the bound obtained from the strong density theorem for the Hecke congruence subgroup in the case of prime $q$.
For $n=3$, we tighten the screws and obtain the following non-trivial improvement.
% \end{rem}

\begin{theorem}\label{th_gl3}
  In the setting of Theorem \ref{th_mt} with $n = 3$,
  % Let $M,\epsilon>0$ and let $q\in \mathbb{N}$ be arbitrary. Further, let $\mathcal{F}_{\mathrm{cusp}} = \mathcal{F}_{\Gamma(q),\mathrm{cusp}}(M)$ be the set of cuspidal automorphic forms for $\Gamma(q) \subseteq  \mathrm{SL}_3(\mathbb{Z})$ with archimedean spectral parameter $\Vert \mu\Vert  \leq M$. Fix a place $v$ of $\mathbb{Q}$, and if $v = p$ is finite, assume that $p \nmid  q$.
  there exists a constant $K$ such that
  \begin{equation*}
    N_v(\sigma,\mathcal{F}_{\mathrm{cusp}}) \ll_{v,\epsilon} M^K \cdot q^{8-12\sigma+\epsilon}
  \end{equation*}
  for all $\sigma\geq	 0$.
\end{theorem} 

% \begin{rem}
An amusing consequence of the argument indicated above is the bound
\begin{equation*}
  \sigma_{\varpi,v}\leq \frac{5}{12}
\end{equation*}
towards the (generalized) Ramanujan conjecture on ${\rm GL}(3)$ for every cusp form $\varpi$ which is spherical at $v$.  This is worse than the best known bound (i.e.\ $\sigma_{\varpi,\infty}\leq \frac{5}{14}$) obtained from progress towards functoriality, but improves on the Jacquet--Shalika barrier (i.e.\ $\sigma_{\varpi,v}\leq \frac{1}{2}$) and makes no recourse to functorial lifts. \medskip
% \end{rem}

The proof of Theorem \ref{th_mt} uses, among other things, the Kuznetsov formula, and is thus restricted to subgroups with cusps.  Nevertheless, it is important to note that it can be transferred to inner forms of $\mathrm{SL}_n$ using the Jacquet--Langlands correspondence.  We take the opportunity to do so in the most basic case, establishing the (spherical) density hypothesis for certain co-compact quotients of $\mathrm{SL}_n(\R)$. The proof is particularly streamlined by the fact that we have density result for $\Gamma(q)$ wih not necesarrily squarefree $q$ available.

We require a bit more notation.  Let $B$ be a degree $n\geq 3$ division algebra over $\Q$ and let $\mathcal{O}\subseteq B$ be an order. Write $\mathcal{O}^1\subseteq \mathcal{O}^{\times}$ for the subgroup of norm one units in the order $\mathcal{O}$.  Let $\mathrm{discr}(\mathcal{O})$ denote the discriminant of $\mathcal{O}$. We assume that $B$ is split over $\R$ and fix an embedding $\phi\colon D\to \mathrm{Mat}_{n}(\R)$ such that $\phi(D)$ is a $\Q$-form of $\mathrm{Mat}_{n}(\R)$. Then $\Gamma_{\mathcal{O}}=\phi(\mathcal{O}^1)\subseteq \mathrm{SL}_n(\R)$ is a co-compact arithmetic subgroup of $\mathrm{SL}_n(\R)$ (see, for example, \cite[Proposition~6.8.9]{Mo}).  For $q\in \mathbb{N}$, we define the principal congruence subgroups by $\Gamma_{\mathcal{O}}(q) := \Gamma_{\mathcal{O}}\cap \Gamma(q).\nonumber$ The space $L^2(\Gamma_{\mathcal{O}}(q)\backslash \mathrm{SL}_n(\R)/\mathrm{SO}_n)$ features a discrete decomposition involving Hecke--Maa\ss\ forms $\varpi$.  As in the non-compact case, we can attach a set of Langlands parameters $\{\mu_{\varpi,v}(j)\colon 1\leq j\leq n \}$ to each unramified place $v$.  Let $\mathcal{F}_{\Gamma_{\mathcal{O}}(q)}(M)$ denote the family of Hecke--Maa\ss\ forms $\varpi$ for $\Gamma_{\mathcal{O}}(q)$ with archimedean spectral parameter $\Vert\mu_{\varpi,\infty}\Vert\leq M$.  We can now state the following theorem:

\begin{theorem}\label{th_cocomp}
  Let $\mathcal{O}$ be as above and let $M,\epsilon>0$.  Further, let $q\in \mathbb{N}$ be such that $(q,\mathrm{discr}(\mathcal{O}))=1$.  Fix a place $v$ of $\mathbb{Q}$, and if $v = p$ is finite, assume that $p \nmid  q\, \mathrm{discr}(\mathcal{O})$. There exists a constant $K$, depending only on $n$, such that
  \begin{equation}
    N_v(\sigma,\mathcal{F}_{\Gamma_{\mathcal{O}}(q)}(M)) \ll _{\mathcal{O},v,\epsilon,n} M^K \cdot [\Gamma_{\mathcal{O}}\colon \Gamma_{\mathcal{O}}(q)]^{1-\frac{2\sigma}{n-1}+\epsilon}, \nonumber
  \end{equation}
  for $\sigma\geq	 0$.
\end{theorem} 

\begin{rem}
  It would be interesting to estimate $N_v(\sigma,\mathcal{F}_{\Gamma_{\mathcal{O}}(q)}(M))$ as the order itself varies (see \cite{FHMM} for when $n = 2$).  This requires a refined understanding of the local Jacquet--Langlands transfer at places where $\mathcal{O}$ ramifies. % lead to explicit estimates, with reasonable dependence on the order?
\end{rem}

% It is an interesting problem to look at the dependence on $\mathcal{O}$.  A natural question would be the following. 

\subsection{Odds and Ends}\label{sec:cnfm30ckjc}
% \pn{I started tightening up a bit here, could do more}

% At the end of this introduction, we want to point towards some, from our point of view, interesting questions and problems:
We record some further questions and problems:
\begin{enumerate}
\item Our density theorem uses the quantity $\sigma_{\varpi,v}=\max_j\vert \Re(\mu_{\varpi,v}(j))\vert$ as a measure for non-temperedness.  This is not the only possibility.  Another natural option is to use $p_{\varpi,v}$, defined to be the infimum of all $q$ such that the $p$-adic matrix coefficient of $\varpi$ is in $L^q(\mathrm{GL}_n(\Q_p)/Z(\Q_p))$.  This notion was employed in Sarnak's original formulation of the density hypothesis, see \cite{Sa, GK, SX} for example.  For $n\geq 4$, the two measures of non-temperedness $p_{\varpi,v}$ and $\sigma_{\varpi,v}$ do not carry the same information.  We emphasize, however, that our approach via the Kuznetsov formula can in principle access more information than just $\sigma_{\varpi,v}$ by modifying the set-up in the spectral side (e.g.\ using different Hecke operators or archimedean test functions).  These modifications of course also change the analysis on the geometric side.  Taking a closer look at this might be interesting in view of applications of the Kuznetsov formula in higher rank beyond the density hypothesis.
\item Even in the realm of $\mathrm{GL}_n$, there are many natural generalizations of Theorem~\ref{th_mt}.  For instance, one could work over general number fields (the case $n=2$ is included in \cite{FHMM}).  Since the most difficult steps in the proof are local in nature, this should impose no major difficulties.  Indeed, the case of totally fields should be rather obvious, while the general case would need a little bit of work over the complex places. Another extension would be to include non-spherical representations at the archimedean places. This seems to be a much harder problem.
\item The asymptotic parameter in our density theorem is the principal congruence subgroup's level, or equivalently, its covolume while the spectral radius $M$ in Theorem \ref{th_mt} is essentially fixed. It is equally natural and interesting to view $M$ as the growing parameter and keep $q$ fixed, or to aim for a hybrid bound. 
  % \pn{I can guess, but not clear what's meant by this next variation without mentioning which earlier result this is a variation on} A natural and very interesting variation is to consider the radius of the spectral ball as growing parameter. Our result gives only very weak polynomial dependence in this aspect, but one would actually expect that
  % \begin{equation}
  %   \# \{ \varpi\in \mathcal{F}_{\Gamma,\mathrm{cusp}}(M)\colon \sigma_{\varpi,p}\geq \sigma \} \ll_{\Gamma,p,\epsilon}M^{\frac{(n-1)(n+2)}{2}(1-\frac{2}{n-1}\sigma)+\epsilon}.\label{eq:arch_dens}
  % \end{equation}
  An archimedean density hypothesis (with applications explained in \cite{JK2}) is known for $\mathrm{SL}_3$ by \cite{BBR}. See also \cite{Ja} for a different archimedean family. %In general attacking this density hypothesis with the Kuznetsov formula leads to interesting analytic problems at the archimedean places, which seem to be very hard. Note that \eqref{eq:arch_dens} has nice applications, for example to optimal diophantine exponents. This is explained in \cite{JK2}. Once \eqref{eq:arch_dens} is available one can of course ask for hybrid versions (i.e. being explicit in $\Gamma$ and $M$) and other variations such as controlling $\sigma_{\varpi,\cdot}$ at several (unramified) places simultaneously.
\item It is a key feature of \cite{AB, Bl} and the paper at hand that on the geometric side, the Weyl element $w_{\ast}$, defined in \eqref{speical_Weyl} below, plays a special role. %This is a rather simple element and we have by now a very good understanding of the associated Kloosterman sums (and sets), see also \cite{BM}. As a result, we can estimate the contribution of $w_{\ast}$ to the geometric side of the Kuznetsov formula quite effectively. 
  Our methods to treat of $w_{\ast}$ are powerful enough that in the present paper, the contribution of $w_{\ast}$ is not the bottleneck in the treatment of the geometric side, but rather our insufficient understanding of Kloosterman sums (and sets) of other Weyl elements, which would show up if one wants to further improve the estimate from Theorem~\ref{th_mt}.  This, among other things, should serve as a good motivation to study Kloosterman sums and (ramified) Kloosterman sets associated to arbitrary Weyl elements. In fact, other Weyl elements become relevant straight away when trying to establish the density hypothesis for lattices $\Gamma$ different from $\Gamma_0(q)$ and $\Gamma(q)$; cf.\  \cite{A1}.
\item Last but not least, it is of course interesting to consider the density hypothesis beyond $\mathrm{SL}_n$.  In principle, there is clear recipe how the Kuznetsov formula should be applied to the problem of counting exceptional generic cusp forms for quasi-split (classical) groups. However, in order to implement this strategy, one is faced with interesting problems in local harmonic analysis: one needs some quantitative estimates for Bessel distributions as well as some understanding of ramified Kloosterman sets. The latter can be translated into an estimate for local intertwining operators. To us, it seems that solving these problems in general is very hard.  In addition, as the Kuznetsov formula can never say anything about non-generic forms, one needs to add these artificially. This happens already for $\mathrm{SL}_n$, where the discrete, non-cuspidal spectrum needs to be treated differently. In the case of classical groups, one can hope to use the endoscopic classification to remedy the situation.  Prototypical results in this direction are given in \cite{A2} and \cite{EGG}, for example.
\end{enumerate}

\subsection{Notation} \label{sec:nor}
We conclude this introduction by recording some notation used throughout the paper.

For the rest of paper, $n$ and the place $v$ will be kept fixed and we will from now on suppress them from  the asymptotic notation. In addition, all implied constants may depend on $\epsilon$ where applicable, and following usual practice, the value of $\epsilon$ may change from line to line. 

% Throughout we are trying to stick to the notation developed in \cite{AB}. At some points, we will need some notions which did not occur in \cite{AB}. We will introduce these carefully throughout the text whenever needed. However, we will end this introduction by giving a quick summary of the most basic notation that will be used frequently.

% The places $v$ of $\Q$ are given by primes $p=2,3,\ldots$ together with the archimedean place $\infty$. Of course we have $\Q_{\infty}=\R$, and $\vert \cdot \vert=\vert\cdot\vert_{\infty}$ is the usual absolute value. 
For a finite place $v=p$ we define the $p$-adic valuation $v_p$ (normalised by $v_p(p)=1$) on $\Q_p$ and the corresponding absolute value $\vert x\vert_p = p^{-v_p(x)}$. We write $p^m \parallel q$ if $v_p(q) = m$. As in \cite{AB} and many other sources we use the self-explanatory notation $(a, b^{\infty})$ and $a \mid b^{\infty}$ for positive integers $a, b$. The local zeta factor is given by $\zeta_p(s)=(1-p^{-s})$. Note that we can express the Euler totient function $\varphi$ as
% \begin{equation*}
$\varphi(q) = \# (\mathbb{Z}/q\mathbb{Z})^{\times} = q\cdot\prod_{p\mid q} \zeta_p(1).$ 
% \end{equation*}

The Weyl group $W$ of $\mathrm{SL}_n(\R)$ is isomorphic to $S_n$. For each $w\in W$ we will once and for all fix a representative in $\mathrm{SL}_n(\Z)$, which will also be denoted by $w$. Note that in order to force the determinant to be $1$ we will chose a sign in the first row of our representative. Of special interest will be Weyl elements of the form
\begin{equation}\label{admis}
  w_{d_1,\ldots,d_r} = \left(
    \begin{matrix}
      & & & I_{d_1}^{\pm} \\ &  & I_{d_2} & \\ & \iddots & & \\ I_{d_r} & & &   
    \end{matrix}
  \right),  
\end{equation}
where $I_d$ denotes the $d\times d$ identity matrix and $I_d^{\pm}=\diag(\pm 1,1,\ldots,1)$. Specifically the long Weyl element is given by $w_l=w_{1,\ldots,1}$. A special role is played by the element
\begin{equation}
  w_{\ast} = w_{1,n-2,1} = \left(
    \begin{matrix}
      0&0&-1\\0&I_{n-2}&0\\1&0&0
    \end{matrix}\right) \label{speical_Weyl}
\end{equation}
and also
\begin{equation*}
  w_1 := w_{1, n - 1},
  \qquad
  w_1 ' := w_{n - 1, 1}.
\end{equation*}
For $n=3$, we encounter the Weyl elements
\begin{equation}
  w_1=w_{1,2}=\left(
    \begin{matrix}
      0& 1 & 0\\0&0&1 \\ 1 & 0 & 0 
    \end{matrix}
  \right),\, w_1'=w_{2,1}=\left(
    \begin{matrix}
      0& 0 & 1\\1&0&0 \\ 0 & 1 & 0 
    \end{matrix}
  \right) \text{ and }w_l=w_{\ast} = \left(
    \begin{matrix}
      0&0&-1\\ 0&1&0 \\ 1&0&0
    \end{matrix}
  \right)\nonumber
\end{equation}
Given $w\in W$ the corresponding permutation is given by $i\mapsto j$ if $w_{ij}\neq 1$.

For a (commutative) ring $R$ (with unit) we define the following important subgroups of $\mathrm{GL}_n(R)$. The diagonal torus is denotes by $T(R)\cong (R^{\times})^n$ and it contains the centre $Z(R)\cong R^{\times}$. Furthermore we write $T_0(R) = T(R)\cap \mathrm{SL}_n(R)$. The group of unipotent upper triangular matrices is called $U(R)$. Given $w\in W$ we write
\begin{equation}
  U_w(R) =w^{-1}U(R)^{\top}w\cap U(R).\nonumber
\end{equation}
The Bruhat decomposition over $\Q$ is given by
\begin{equation}
  \mathrm{SL}_n(\Q) = \bigsqcup_{w\in W}G_w(\Q) \text{ for }G_w(\Q) = U(\Q)T_0(\Q)wU_w(\Q). \nonumber
\end{equation} 
Note that $B = TU$ (resp. $B_0=T_0U$) is the standard minimal parabolic subgroup of $\mathrm{GL}_n$ (resp. $\mathrm{SL}_n$). Other standard parabolic subgroups will usually be denoted by $P$.

The following matrices in $T$ (resp.\ $T_0$) will be of special interest. First, given $q\in \N$ we define
\begin{equation*}%\label{dq}
  D_q=\diag(q^{n-1},q^{n-2},\ldots, q,1)\in T(\Q). 
\end{equation*}
Furthermore, given $c=(c_1,\ldots,c_{n-1})\in \mathbb{N}^{n-1}$ we write
\begin{equation}
  c^{\ast} = \diag\Big(\frac{1}{c_{n-1}},\frac{c_{n-1}}{c_{n-2}},\ldots,\frac{c_2}{c_1},c_1\Big)\in T_0(\Q).\nonumber
\end{equation}
Given $c$ as above it will be convenient to set $c_0=1=c_n$. We define $c_{\mathrm{op}} = (c_{n-1},\ldots,c_1)$ by reversing the order of the entries.

Over $\R$ the all important compact subgroup is $K_{\infty} = \mathrm{SO}_n(\R)$. For a finite place $v=p$ we set $K_p=\mathrm{GL}_n(\Z_p)$. The local principal congruence subgroup is given by $K_p(q) = (1+q\mathrm{Mat}_{n}(\Z_p))\cap K_p$. (Note that we take the intersection with $K_p$ only to ensure that $K_p(q)=K_p$ whenever $p\nmid q$.) We also fix the (standard) Iwahori subgroup $B_{0,p}\subseteq K_p$ given by
\begin{equation}
  B_{0,p} = \{ k\in K_p\colon [k\text{ mod }p]\in TU(\mathbb{Z}/p\Z)\}.\nonumber
\end{equation}
Finally set $K_p(q)^{\natural} = D_q^{-1}K_p(q)D_q\subseteq \mathrm{GL}_n(\Q_p)$.

Globally we will encounter the subgroups $K=\prod_v K_v$, $K(q) = K_{\infty}\cdot \prod_{p}K_p(q)$ and $K(q)^{\natural} = K_{\infty}\cdot \prod_{p}K_p(q)^{\natural}$. These are the adelic counterparts to the classical lattices $\mathrm{SL}_n(\Z)$, $\Gamma(q)$ amd $\Gamma(q)^{\natural} = D_q^{-1}\Gamma(q)D_q$. The corresponding (classical) quotients are
\begin{equation}
  X_q = \Gamma(q)^{\natural}\backslash \mathrm{SL}_n(\R)/\mathrm{SO}_n(\R) \text{ and }X(q) = \Gamma(q)\backslash \mathrm{SL}_n(\R)/\mathrm{SO}_n(\R).\nonumber
\end{equation}

At the finite places we fix the standard additive characters $\psi_p$ of $\Q_p$ which are trivial on $\Z_p$ but non-trivial on $p^{-1}\Z_p$. We use this character to define the Fourier transform\footnote{This is the inverse of the usual Fourier transform, which we choose for notational convenience.}
\begin{equation*}
  \widehat{f}(Y)=\int_{\mathrm{Mat}_n(\mathbb{Q}_p)}f(-X)\overline{\psi_p(\mathrm{Tr}(XY))}dX.
\end{equation*}
Here $dX$ is normalized to be self-dual. In the context of this Fourier transform it will also be useful to introduce the notation $f_t(X) = f(t^{-1}X)$ for a function $f\colon \mathrm{Mat}_n(\mathbb{Q}_p) \to \mathbb{C}$ and $t\in \Q_p$.

We lift the characters $\psi_p$ to the non-degenerate character $\boldsymbol{\psi}_{\Q_p}$ of $U(\Q_p)$ given by
\begin{equation}
  \boldsymbol{\psi}_{\Q_p}(u) = \psi_p(u_{1,2}+\ldots+u_{n-1,n}).\nonumber
\end{equation}
These are in particular trivial on $U(\Z_p)$.

At the archimedean place we define $\psi_{\infty}(x) =e(x)=e^{2\pi ix}$. For $N=(N_1,\ldots,N_{n-1})\in \Z^{n-1}$ we lift $e(\cdot)$ (or $\psi_{\infty}$) to $U(\R)$ by setting
\begin{equation}\label{char}
  \theta_{N}(u) = e(N_{n-1}u_{1,2}+\ldots+N_1u_{n-1,n}). 
\end{equation}
We will also write $\theta_N^v(u) = \theta_{N}(v^{-1}uv)$ for $v \in V$, the set of diagonal matrices with entries $\pm 1$. 

Let us briefly talk about measures on $\mathrm{GL}_n(\Q_p)$. First of all we normalize the Haar measure on $K_p$ to be a probability measure.  We equip $P(\Q_p)$ with a left invariant measure $dp$ and define the modulus $\delta_P$ so that $\delta_P^2\cdot dp$ is right invariant. We can then define the Haar measure of $\mathrm{GL}_n(\Q_p)$ via
\begin{equation*}
  \int_{\mathrm{GL}_n(\mathbb{Q}_p)}f(g) dg = \int_{K_p}\int_{P(\mathbb{Q}_p)} f(kp)\delta_P(p)^2dkdp.
\end{equation*}

Finally, let $K$ be either a finite field or a finite extension of $\mathbb{Q}_p$. Let $\mathcal{P}^0(n)$ be the set of ordered partitions $\alpha=(\alpha_1,\ldots,\alpha_{r(\alpha)})$ of $n$. Thus we have $\alpha_1+\ldots+\alpha_{r(\alpha)} = n$ and $\alpha_1\geq \alpha_2\geq \ldots \alpha_{r(\alpha)}>0$. There is a standard one to one correspondence between $\mathcal{P}^0(n)$ and nilpotent $\mathrm{GL}_n(K)$ orbits in $\mathrm{Mat}_n(K)$, which we write as
\begin{equation}
  \alpha \mapsto \mathcal{O}_{\alpha} = \mathrm{Ad}(\mathrm{GL}_n(K))Y_{\alpha},\nonumber
\end{equation} 
where the matrix $Y_{\alpha}$ is constructed from $\alpha$ as follows. For $1\leq i\leq \alpha_1$ we let $\alpha_i' = \#\{j\in 1,\ldots,r(\alpha)\colon \alpha_j\geq i \}$. Now we let $Y_{\alpha}\in \mathrm{Mat}_n(K)$ be the nilpotent element in Jordan canonical form with block sizes $\alpha_i'$. An important example is $Y_{1,\ldots,1}$, which is the regular nilpotent element in Jordan canonical form.

\section{The spectral estimate}\label{sec_f}
Given a cuspidal automorphic representation, we recall the notation $\pi\mid L^2(X_q)$ and $V_{\pi}$ from \cite[Section~2.6.3]{AB}: roughly speaking, $V_{\pi} = V_{\pi}^{(q)}$ is the de-adelization of the space $\pi^{K(q)^{\natural}}$.  In the following, we will also consider the de-adelization of the slightly larger space $\pi^{K(q')^{\natural}}$ for $q \mid q'$, which we denote by $V_{\pi}^{(q')}$.
% In order to keep track of the $q$-dependence, we will refine the notation from \cite{AB} slightly and write $V_{\pi}^{(q)}$ instead of just $V_{\pi}$.  \pn{confused by this convention, since, e.g., we use $V_\pi$ in place of $V_\pi^{(q)}$ in the next paragraph}

Following the discussion in \cite[Section~2.5]{AB}, we define the $N$th Fourier coefficient of a classical cusp form $\varpi\in V_{\pi}^{(q)}$ for $N \in \mathbb{N}^{n-1}$ by
\begin{equation}
  \mathcal{W}_{\varpi}(y,N) = \int_{U(\mathbb{Z})\backslash U(\R)}\varpi(xy)\theta_N(x)^{-1}dx = \frac{A_{\varpi}(N)}{N^{2\eta}}W_{\mu_{\varpi}}(N\cdot \mathrm{y}(y)), \nonumber
\end{equation}
where $\mathrm{y}(\diag(y_1, \ldots, y_n)) = (\frac{y_{n-1}}{y_n}, \ldots, \frac{y_1}{y_2})$ and $\eta = (\frac{1}{2}j(n-j))_{1 \leq j \leq n-1}$ are as in \cite[(2.4) and (2.6)]{AB}, the power $N^{2\eta}$ is taken componentwise, and the archimedean Whittaker function $W_{\mu}$ is given in \cite[(2.16)]{AB}.

The goal of this section is to prove the following crucial estimate, generalizing \cite[Theorem~1.3]{AB}:

\begin{prop}\label{pro_four_coeff}
  There exists $q_0=q_0(n)$, depending only on $n$, so that for each $q \in \mathbb{N}$ and each cuspidal automorphic representation $\pi \mid L^2(X_q)$, we have
  \begin{equation}\label{prop21}
    \sum_{\varpi\in \mathrm{ONB}(V_{\pi}^{(q')})} \vert A_{\varpi}(1)\vert^2 \gg (\Vert\mu_{\pi}\Vert q)^{-\epsilon} \dim(V_{\pi}^{(q)}) \cdot \frac{\mathcal{N}_q}{\mathcal{V}_q},\nonumber
  \end{equation}
  where $q'=q_0\cdot q$ and
  \begin{equation}\label{constants}
    \mathcal{V}_q := [{\rm SL}_n(\mathbb{Z}) : \Gamma(q)] = q^{n^2 - 1+ o(1)}, \quad \mathcal{N}_q := [\Gamma(q)^{\natural} \cap U(\mathbb{Q}) : U(\mathbb{Z})] = q^{n(n-1)(n-2)/6}.
  \end{equation}
\end{prop}

\begin{rem}
  We strongly expect that one can actually take $q_0=1$, but it is convenient to use the flexibility to enlarge the level slightly.
  We refer to Remark~\ref{rem:exp_q_0} below for more details. We will see that $q_0=\prod_{p\leq 2n}p^{d}$, for $d$ as in Lemma~\ref{depth_pos_case}, is admissible.
\end{rem}

The proof of Proposition \ref{pro_four_coeff} may be reduced, as in \cite[Section~6.1]{AB}, to a local statement, which we now formulate over any non-archimedean local field $F$ of characteristic zero.  We denote by $\mathfrak{o}$ the ring of integers, $\mathfrak{p}$ its maximal ideal, $\mathfrak{k} := \mathfrak{o} / \mathfrak{p}$ the residue field, and $q$ (resp.\ $p$) the cardinality (resp. characteristic) of $\mathfrak{k}$.  We fix a uniformizer $\varpi$, i.e., a generator of $\mathfrak{p}$, as well as an unramified character $\psi_F : F \rightarrow \U(1)$, thus $\psi_F$ is trivial on $\mathfrak{o}$ but not on $\mathfrak{p}^{-1}$.

The group $\mathrm{GL}_n$ and all relevant subgroups are defined over $F$.  We denote by $K=\mathrm{GL}_n(\mathfrak{o})$ the standard maximal compact subgroup and, for a nonnegative integer $m$, by $K(m) := \ker(K \rightarrow \GL_n(\mathfrak{o}/\mathfrak{p}^m))$ the corresponding principal congruence subgroup.  Of course, if $F=\mathbb{Q}_p$, then $K=K_p$ and $K(m) = K_p(p^m)$. In what follows, we will usually think $m=v_p(q')$.

Let $\pi$ be a generic irreducible unitary representation of $\GL_n(F)$, equipped with an invariant inner product $\langle , \rangle$.  For each nonnegative integer $m$ and nondegenerate unitary character $\boldsymbol{\psi}$ of $U(F)$, we define the Bessel distribution
\begin{equation}\label{eq:cngrzst9kl}
  \mathcal{S}_{\pi,m}(\boldsymbol{\psi})
  :=
  \sum_{v\in \mathrm{ONB}(\pi^{K(m)})}\int_{U(F)}^{\mathrm{st}} \langle \pi(u)v,v\rangle\boldsymbol{\psi}(u)^{-1}du,%\nonumber
\end{equation}
where $\int_{U(F)}^{\mathrm{st}}$ denotes a stable integral in the sense of Lapid--Mao \cite[\S2.1]{MR3267120}.  We are particularly interested in $\boldsymbol{\psi} = \boldsymbol{\psi}_m$ of the form
\begin{equation*}
  \boldsymbol{\psi}_m(u) = \psi_F\left(\varpi^{-m}\sum a_ju_{j,j+1}\right),
\end{equation*}
where $m \in \mathbb{Z}_{\geq 1}$ and  $a_1,\ldots,a_{n-1}\in \mathfrak{o}^{\times}$.  Let $\delta(\pi)$ denote the ``principal depth'' of $\pi$, namely, the smallest nonnegative integer for which the fixed subspace $\pi^{K(\delta(\pi))}$ is nontrivial.  The main local result of this section is then the following.
\begin{prop}\label{proposition:cngrlcwv3h}
  With notation and assumptions as above, there exists $d = d(n,F) \in \mathbb{Z}_{\geq 0}$, depending only upon $n$ and $F$, so that for each natural number $m \geq \delta(\pi) + d$, we have $\mathcal{S}_{\pi, m}(\boldsymbol{\psi}_m) \geq 1$.  Moreover, for $p > 2 n$, we may take $d = 0$.
\end{prop}
\begin{rem}\label{rem:cnh3r9n3n2}
  We note that by \cite[Lemma 4.4]{MR3267120}, one has $$\mathcal{S}_{\pi,m}(\psi) = \sum_{v\in \mathrm{ONB}(\pi^{K(m)})} \lvert \ell_{\psi}(v) \rvert^2$$ for a suitable $\psi$-Whittaker functional $\ell_{\psi}$ on $\pi$.  This explains why Proposition~\ref{proposition:cngrlcwv3h} implies Theorem~\ref{theorem:cnh3r840lw}.  (One can directly argue the converse, but we do not need to do so.)
\end{rem}
\begin{proof}[Proof of Proposition~\ref{pro_four_coeff}, assuming Proposition~\ref{proposition:cngrlcwv3h}]
  For any $q_0$, we have, as in the proof of \cite[Prop 6.1]{AB}, the lower bound
  \begin{equation*}
    \sum_{\varpi\in \mathrm{ONB}(V_{\pi}^{(q')})} \vert A_{\varpi}(1)\vert^2 \gg (\Vert\mu_{\pi}\Vert q')^{-\epsilon}   \frac{\mathcal{N}_{q'}}{\mathcal{V}_{q'}}
    (q')^{n(n - 1)/2}
    \prod_{p | q'} \mathcal{S}_p,
  \end{equation*}
  where $\mathcal{S}_p$ is given, for $p^{m_p} || q'$, by
  \begin{equation*}
    \sum_{v\in \mathrm{ONB}(\pi_p^{K_p(p^{m_p})})}\int_{U(\mathbb{Q}_p)}^{\mathrm{st}} \langle \pi_p(u)v,v\rangle\boldsymbol{\psi}_{\mathbb{Q}_p}(D_{q'} u D_{q'}^{-1})^{-1} \, d u.
  \end{equation*}
  Using the definition of $D_{q'}$, we verify readily that $\mathcal{S}_p$ is of the form $\mathcal{S}_{\pi_p, {m_p}}(\boldsymbol{\psi}_{m_p})$, for suitable $a_j \in \mathbb{Z}_p^\times$, chosen so that $\boldsymbol{\psi}_{\Q_p}(D_{q'}^{-1}uD_{q'}) = \boldsymbol{\psi}_m(u)$.\footnote{We remark that there is an unfortunate typo in \cite[Section~6.1]{AB}, where one should essentially replace $D_q$ by $D_q^{-1}$ throughout.  Here we rectify this by modifying the character in the definition of the Bessel distribution.}  We take $q_0 := \prod_{p \leq 2 n} p^{d(n,\mathbb{Q}_p)}$, with exponents as in the conclusion of Proposition \ref{pro_four_coeff}.  Since $\pi \mid L^2(X_q)$, we then have $m_p \geq \delta(\pi_p) + d(n,\mathbb{Q}_p)$.  The hypotheses of that proposition are thus satisfied, so we have $\mathcal{S}_p \geq 1$.  The estimate $(q')^{n(n-1)/2 + \epsilon} \gg \dim(V_{\pi}^{(q')})$ is given by \cite[Lemma A.1]{MR3981603}.  We may swap the constants $\mathcal{N}_{q'}$ and $\mathcal{V}_{q'}$ for their counterparts involving $q$ at the cost of a constant factor, depending only upon $q_0$.  The required estimate follows.
\end{proof}

The proof of Proposition~\ref{proposition:cngrlcwv3h} occupies the remainder of this section.  Many cases of that proposition were established in \cite[\S6]{AB}, and the reductions given there apply here.  In particular, by \cite[Lemma 6.10]{AB}, the proof reduces to the case of supercuspidal $\pi$.  For such $\pi$, we have a mild simplification: matrix coefficients have compact support modulo the center, so the stable integral in \eqref{eq:cngrzst9kl} may replaced by an ordinary integral.  The case $\delta(\pi) = 0$ is trivial (it concerns characters on $\GL_1(F)$), while the case $\delta(\pi) = 1$ (with $d = 0$) follows from \cite[Lemma 6.5]{AB}.  Our remaining task is thus to address the case of supercuspidal $\pi$ with $\delta(\pi) \geq 2$.  The proof is completed in \S\ref{sec:cngrz3r57v} below.

% \pn{maybe add a remark here indicating what this says about local Whittaker functions, if I can think of something good to say today}

\subsection{Reducing the support of the integral}\label{sec:cnfm30cn1b}
We start by showing that we can replace the integral over $U(F)$ in the definition of $\mathcal{S}_{\pi,p}(\boldsymbol{\psi}_m)$ by an integral over $U(\mathfrak{o})$:
\begin{lemmy}\label{eq:cngv4310dw}
  Let $\pi$ be a supercuspidal representation of $\GL_n(F)$.  For each $m$, we have
  \begin{equation*}
    \mathcal{S}_{\pi,m}(\boldsymbol{\psi}_m) = \sum_{v\in \mathrm{ONB}(\pi^{K(m)})} \int_{U(\mathfrak{o})} \langle \pi(u)v,v\rangle\boldsymbol{\psi}_{m}(u)^{-1}du. %\label{eq_trunc_int}
  \end{equation*}
\end{lemmy}
\begin{proof}
  The proof is a slight refinement of \cite[Proposition~3.1]{La}. The argument relies purely on bi-$K(m)$-invariance of the matrix coefficients and character orthogonality.
  
  We start by giving the argument for the case of $\mathrm{GL}_2$, where our main trick already becomes obvious. Let $A$ denote the subgroup of $T$ given by matrices of the form $a(y)=\diag(a,1)$. Further, let $f(g)=\langle \pi(g)v,v\rangle$. We will only use that $f(g)$ is bi-$A(1+\mathfrak{p}^m)$-invariant. Put $n(x) = \left(
    \begin{smallmatrix}
      1&x\\0&1
    \end{smallmatrix}
  \right)$. Now we simply compute
  \begin{align}
    &\int_{U(F)}\langle \pi(u)v,v\rangle\boldsymbol{\psi}_{m}(u)^{-1}du = \int_{F}f(n(x))\psi_{F}(a_1x\varpi^{-m})dx \nonumber\\
    &\quad = \int_{\mathfrak{o}}\int_{F}f(a(1+\varpi^my)^{-1}n(x)a(1+\varpi^my))\psi_{F}(a_1x\varpi^{-m})dx\, dy \nonumber\\
    &\quad = \int_{\mathfrak{o}}\int_{F}f(n((1+\varpi^my)^ {-1}x))\psi_{F}(a_1x\varpi^{-m})dx\, dy \nonumber\\
    &\quad = \int_{F}f(n(x))\psi_{F}(a_1x\varpi^{-m})\int_{\mathfrak{o}}\psi_{F}(a_1yx)dy\, dx = \int_{\mathfrak{o}}f(n(x))\psi_{F}(a_1x\varpi^{-m})dx.\nonumber
  \end{align}
  This is exactly the statement we are after.

  Now we turn towards the general case. We use the usual embedding
  \begin{equation*}
    \mathrm{GL}_{n-1}(F)\ni h \mapsto \left(
      \begin{matrix}
        h&0\\0&1
      \end{matrix}
    \right)\in \mathrm{GL}_n(F)
  \end{equation*}
  and identify $\mathrm{GL}_{n-1}(F)$ with its image. Let $U^{(n-1)}(F)=U(F)\cap \mathrm{GL}_{n-1}(F)$ and set
  \begin{equation}
    W(F) = \left\{ w=\left(
        \begin{matrix}
          I_{n-1} & \mathbf{w}\\0&1 
        \end{matrix}
      \right)\colon \mathbf{w}\in F^{n-1}\right\}.\nonumber
  \end{equation}
  Now define
  \begin{equation}
    f(h) = \int_{W(F)}\langle \pi(hw)v,v\rangle\boldsymbol{\psi}_{m}(w)^{-1}dw.\nonumber
  \end{equation}
  In particular, we can write
  \begin{equation}
    \int_{U(F)}\langle \pi(u)v,v\rangle\boldsymbol{\psi}_{m}(u)^{-1}du = \int_{U^{(n-1)}(F)}f(u)\boldsymbol{\psi}_{m}(u)^{-1}du. \nonumber
  \end{equation}
  Note that $f$ (as a function on $\mathrm{GL}_{n-1}(F)$) is bi-invariant under
  \begin{equation}
    \left\{ \left(
        \begin{matrix}
          k&0\\0&1
        \end{matrix}
      \right)\colon k\equiv I_{n-2} \text{ mod }\mathfrak{p}^m  \right\} \subseteq \mathrm{GL}_{n-1}(F).\nonumber
  \end{equation}
  To see this, put $k_i=\left(
    \begin{matrix}
      k_i' & 0 \\ 0& I_2
    \end{matrix}
  \right)$ with $k_i'\equiv I_{n-2} \text{ mod }\mathfrak{p}^m$, $i=1,2$ and write $\mathbf{w}=\left(
    \begin{matrix}
      \mathbf{w}' \\ w_{n-1}
    \end{matrix}
  \right)$. Since $k_1,k_2\in K(m)$ we have
  \begin{align}
    f(k_1hk_2) &= \int_{W(F)}\langle \pi(hk_2wk_2^ {-1})v,v\rangle\boldsymbol{\psi}_{m}(w)^{-1}dw \nonumber\\
               &
                 = \int_{W(F)}\langle \pi(hw)v,v\rangle\boldsymbol{\psi}_{m}(w)^{-1}dw = f(h).\nonumber
  \end{align}
  In the middle step, we have used that conjugating $w$ with $k_2$ leaves $w_{n-1,n}$ untouched:
  \begin{equation}
    k_2wk_{2}^{-1} = \left(
      \begin{matrix}
        I_{n-2} & 0 & k_2'\mathbf{w}' \\ 0&1&w_{n-1,n} \\0&0 &1
      \end{matrix}
    \right).\nonumber
  \end{equation}

  This invariance of $f$ allows us to argue inductively with $n=2$ being the initial case. We arrive at
  \begin{equation}
    \int_{U^{(n-1)}(F)}f(u)\boldsymbol{\psi}_{m}(u)^{-1}du=\int_{U^{(n-1)}(\mathfrak{o})}f(u)\boldsymbol{\psi}_{m}(u)^{-1}du.\nonumber
  \end{equation}
  This implies that
  \begin{equation*}
    \int_{U(F)}\langle \pi(u)v,v\rangle\boldsymbol{\psi}_{m}(u)^{-1}du = \int_{U^{n-1}(\mathfrak{o})}\boldsymbol{\psi}_{m}(u)^{-1}\int_{W(F)}\langle \pi(uw)v,v\rangle\boldsymbol{\psi}_{m}^{-1}(w)dw\, du. 
  \end{equation*}
  We are done as soon as we can show that the integral over $W(F)$ can be replaced by an integral over $W(\mathfrak{o})$. Given $\mathbf{y} = \left(
    \begin{matrix}
      y_1,\ldots,y_{n-1}
    \end{matrix}
  \right) = \left(
    \begin{matrix}
      \mathbf{y}' & y_{n-1}
    \end{matrix}
  \right)\in \mathfrak{o}^{n-1}$ we define
  \begin{equation}
    a(\varpi^m\mathbf{y}) =% I_n + \left(
    % \begin{matrix}
  %   0 & 0 \\ \varpi^m\mathbf{y} & 0 \\ 0&0
  % \end{matrix}
  % \right) =
    \left(
      \begin{matrix}
        1_{n-2} & 0 & 0\\ \varpi^m\mathbf{y}' & 1+\varpi^my_{n-1} & 0\\0&0&1  
      \end{matrix}
    \right)\in K(m)\cap \mathrm{GL}_{n-1}(F).\nonumber 
  \end{equation}
  For $u\in U^{n-1}(\mathfrak{o})\subseteq K$ and $\mathbf{y}\in \mathfrak{o}^{n-1}$ we have
  \begin{equation*}
    ua(\varpi^m\mathbf{y})^{-1}u^{-1} \in K(m)\cap \mathrm{GL}_{n-1}(F). %, \text{ for all }\mathbf{y}\in \mathfrak{o}_F^{n-1}.
  \end{equation*}
  This allows us to compute
  \begin{align}
    &\int_{W(F)}\langle \pi(uw)v,v\rangle\boldsymbol{\psi}_{m}(w)^{-1}dw \nonumber \\
    &= \int_{W(F)}\int_{\mathfrak{o}^{n-1}}\langle \pi(ua(\varpi^m\mathbf{y})^{-1}wa(\varpi^m\mathbf{y}))v,v\rangle\boldsymbol{\psi}_{mp}(w)^{-1}d\mathbf{y}\,dw \nonumber \\
    &= \int_{W(F)}\int_{\mathfrak{o}^{n-1}}\langle \pi(uw)v,v\rangle\boldsymbol{\psi}_{m}(a(\varpi^m\mathbf{y})wa(\varpi^m\mathbf{y})^{-1})^{-1}d\mathbf{y}\, dw \nonumber \\
    &= \int_{W(F)}\int_{\mathfrak{o}^{n-1}}\langle \pi(uw)v,v\rangle\psi_{F}(a_{n-1}(\mathbf{y}'\cdot \mathbf{w}'+y_{n-1}w_{n-1}+\varpi^{-m}w_{n-1}))^{-1}d\mathbf{y}\, dw \nonumber\\
    &= \int_{W(F)}\langle \pi(uw)v,v\rangle\boldsymbol{\psi}_{m}(w)^{-1}\int_{\mathfrak{o}^{n-1}}\psi_{F}(a_{n-1}(\mathbf{y}'\cdot \mathbf{w}'+y_{n-1}w_{n-1}))^{-1}d\mathbf{y} \, dw \nonumber \\
    &=\int_{W(\mathfrak{o})}\langle \pi(uw)v,v\rangle\boldsymbol{\psi}_{m}(w)^{-1}dw.\nonumber
  \end{align}
  In the last step we have simply used (additive) character orthogonality. Combining the $W$ and the $U^{n-1}$-integral completes the proof.
\end{proof}

\subsection{Regular representations of \texorpdfstring{$K$}{K}}\label{sec:cnfm30cpqb}
Let $\pi$ be a smooth admissible representation of $\GL_n(F)$, and let $m \geq 1$.  Since $K(m)$ is normal in $K$, the space $\pi^{K(m)}$ is $K$-invariant, so we may view it as a representation of $K$, or of the quotient
\begin{equation*}
  G_m = \mathrm{GL}_n(\mathfrak{o}/\mathfrak{p}^m) \cong K/K(m).
\end{equation*}
We can decompose the latter representation into irreducible constituents:\footnote{We will not distinguish between a representation of $G_m$ and its lift to $K$. Hopefully this will not lead to confusion.}
\begin{equation}
  \pi^{K(m)} = \bigoplus_{\sigma\in \widehat{G_m}} m_{\pi,m}(\sigma)\cdot \sigma.\nonumber
\end{equation}
Since $\pi$ is admissible, almost all of the multiplicities $m_{\pi,m}(\sigma)$ are zero.

Assume henceforth that $m \geq 2$.  The quotient $K(m-1)/K(m)$ is then abelian, and we have an isomorphism
\begin{equation}\label{eq:cngrzw1yis}
  \Mat_n(\mathfrak{k}) \rightarrow \text{character group of } K(m-1) / K(m)  
\end{equation}
given by $\beta \mapsto \theta_\beta$, where
\begin{equation*}%\label{eq:earlier_constr}
  \theta_{\beta}(x) := \psi_F\left(\mathrm{tr}\Big(\beta \frac{x-1}{\varpi^m}\Big)\right).
\end{equation*}
Strictly speaking, we first lift $\beta$ to an element of $\Mat_n(\mathfrak{o})$, then observe that the above formula defines a character of $K(m-1)$, independent of the choice of lift and trivial on $K(m)$.

\begin{rem}\label{remark:cngrk8lcsv}
  Characters of the form $\theta_{\beta}$ (and certain generalizations) play a major role in type theory and in the construction of supercuspidal representations for $\mathrm{GL}_n(F)$, as described for example in \cite{BK}. The construction is usually specified in terms of an additive character $\tilde{\psi}_F$ of conductor $\mathfrak{p}_F$.  Here we work with the unramified character $\psi_F$, following \cite{How3} for example. The price is a simple shift in indices that must be taken into account when to referring to certain parts of the literature.
\end{rem}

There are several equivalent characterizations of when an element $\beta$ of $\Mat_n(\mathfrak{k})$ is \emph{regular}.  One is that the centralizer $C_{\mathrm{Mat}_n(\mathfrak{k})}(\beta)$ has dimension $n$; another is that $\beta$ admits a cyclic vector (cf.\ \cite[Section~3]{HN}).  Following \cite{Hi}, we make the following definition.

\begin{defn}
  For $m > 1$, an irreducible representation $\sigma$ of $G_m$ is \emph{regular} if $\sigma \vert_{K(m-1)}$ contains $\theta_{\beta}$ for some regular ${\beta}\in \mathrm{Mat}_{n}(\mathfrak{k})$.
\end{defn}

Regular representations play an important role in the representation theory of $\mathrm{GL}_n$ over finite local principal ideal rings.  They can be constructed using, e.g., ideas from Kirillov's orbit method.  We refer to \cite{St} for a nice survey.

We are now ready to prove the central result of this section:
\begin{lemmy}\label{lm:red_to_mult}
  For $m \geq 2$, we have
  \begin{equation}
    \mathcal{S}_{\pi,m}(\boldsymbol{\psi}_m) = M(\pi,m), \qquad  M(\pi,m) := \sum_{\substack{\sigma\in \widehat{G_m}\\ \text{\rm regular}}} m_{\pi,m}(\sigma).  \nonumber
  \end{equation}
  % where
  % \begin{equation}
  %   M(\pi,m) = \sum_{\substack{\sigma\in \widehat{G_m}\\ \text{\rm regular}}} m_{\pi,m}(\sigma).\nonumber
  % \end{equation}
\end{lemmy}
The proof has some similarities to the proof of \cite[Lemma~6.5]{AB}. However, we use deep facts concerning generalized Gelfand--Graev representations for $\mathrm{GL}_n(\mathfrak{o}/\mathfrak{p}^m)$.
\begin{proof}
  As in \cite{AB}, we note that the character $\boldsymbol{\psi}_{m}$ is trivial on $U(\mathfrak{p}^m)$ and descends to a non-degenerate character of $U(\mathfrak{o}/\mathfrak{p}^m)$, which we denote by $\widetilde{\boldsymbol{\psi}}$.  Using that $\mathcal{S}_{\pi,m}(\boldsymbol{\psi}_m)$ is given by an integral over $U(\mathfrak{o})$ (Lemma~\eqref{eq:cngv4310dw}), we can write
  \begin{align}
    \mathcal{S}_{\pi,m}(\boldsymbol{\psi}_m) &= |U(\mathfrak{o}/\mathfrak{p}^m)|^{-1} \sum_{v\in \mathrm{ONB}(\pi^{K(m)})} \sum_{u\in U(\mathfrak{o}/\mathfrak{p}^m)} \langle \pi(u)v,v\rangle\widetilde{\boldsymbol{\psi}}(u)^{-1} \nonumber\\
                                             &=|U(\mathfrak{o}/\mathfrak{p}^m)|^{-1} \sum_{\sigma\in \widehat{G_m}} m_{\pi,m}(\sigma)\sum_{u\in U(\mathfrak{o}/\mathfrak{p}^m)} \mathrm{tr}(\sigma(u))\widetilde{\boldsymbol{\psi}}(u)^{-1}\nonumber \\
                                             &= \sum_{\sigma\in \widehat{G_m}} m_{\pi,m}(\sigma) \langle \chi_{\sigma}\vert_{U(\mathfrak{o}/\mathfrak{p}^m)},\widetilde{\boldsymbol{\psi}}\rangle_{U(\mathfrak{o}/\mathfrak{p}^m)}.\nonumber
  \end{align}
  Here we write $\chi_{\sigma}$ for the character of $\sigma$ and for a finite group $H$ we set $\langle f,g\rangle_H = \frac{1}{\# H}\sum_{h\in H}f(h)\overline{g(h)}$. This brings us in a situation where we can apply character theory.

  Let $\mathcal{G}_{\widetilde{\boldsymbol{\psi}}}=\mathrm{Ind}_{U(\mathfrak{o}/\mathfrak{p}^m)}^{G_m}(\widetilde{\boldsymbol{\psi}})$ denote the Gelfand--Graev representation of $G_m$, and write $\widetilde{\boldsymbol{\psi}}^{G_m}$ for the associated character.  Using Frobenius Reciprocity, we obtain
  \begin{equation}
    \langle \chi_{\sigma}\vert_{U(\mathfrak{o}/\mathfrak{p}^m)},\widetilde{\boldsymbol{\psi}}\rangle_{U(\mathfrak{o}/\mathfrak{p}^m)} = \langle \chi_{\sigma},\widetilde{\boldsymbol{\psi}}^{G_m}\rangle_{G_m} = \dim_{\mathbb{C}} \Hom_{G_m}(\sigma,\mathcal{G}_{\widetilde{\boldsymbol{\psi}}}).\nonumber
  \end{equation}
  By combining \cite[Proposition~5.7]{Hi} and \cite[Theorem~1.1]{PS}, we find that
  \begin{equation}
    \dim_{\mathbb{C}} \Hom_{G_m}(\sigma,\mathcal{G}_{\widetilde{\boldsymbol{\psi}}}) = 
    \begin{cases}
      1 &\text{ if $\sigma$ is regular,}\\
      0&\text{ else.}
    \end{cases}
    \nonumber
  \end{equation}
  Putting everything together yields
  \begin{align}
    \mathcal{S}_{\pi,m}(\boldsymbol{\psi}_m) &= \sum_{\sigma\in \widehat{G_m}} m_{\pi,m}(\sigma) \langle \chi_{\sigma}\vert_{U(\mathfrak{o}/\mathfrak{p}^m)},\widetilde{\boldsymbol{\psi}}\rangle_{U(\mathfrak{o}/\mathfrak{p}^m)} \nonumber\\
                                             &= \sum_{\sigma\in \widehat{G_m}} m_{\pi,m}(\sigma)\cdot \dim_{\mathbb{C}} \Hom_{G_m}(\sigma,\mathcal{G}_{\widetilde{\boldsymbol{\psi}}})= \sum_{\substack{\sigma\in \widehat{G_m},\\ \text{regular}}} m_{\pi,m}(\sigma)\nonumber
  \end{align}
  and the proof is complete.
\end{proof}

This last result tells us that in order to show $\mathcal{S}_{\pi,m}(\boldsymbol{\psi}_m)\geq 1$, it suffices to exhibit a regular representation $\sigma\in \widehat{G_m}$ with $m_{\pi,m}(\sigma)\neq 0$. This is an algebraic question that we address in the following subsection.

\subsection{Regular types in supercuspidal representations}\label{sec:ex_reg_ty}
In the previous sections, we have reduced the problem of estimating $\mathcal{S}_{\pi,m}(\boldsymbol{\psi}_m)$ to a purely algebraic one concerning the representation $\pi$, closely related to the theory of types and to the construction and classification of supercuspidal representations as developed in \cite{BK}, for instance.  However, it turns out that our problem is slightly different in flavor.  In fact, it leads to some phenomenon that we find interesting in their own right, cf.\ Remark \ref{rem28} below. % For example we arrive at the curious question whether \eqref{conj_loc} is true for all (supercuspidal) representations $\pi$ of $\mathrm{GL}_n(F)$.

\begin{lemmy}\label{lm:exist_reg}
  There exists $d = d(n,F) \in \mathbb{Z}_{\geq 0}$ with the following property.  Let $\pi$ be an irreducible supercuspidal representation of $\GL_n(F)$, and let $m \geq 2$ with $m \geq \delta(\pi) + d$ (i.e., $\pi^{K(m - d)}\neq \{0\}$).  Then there exists a regular representation $\sigma\in \widehat{G_m}$ with $m_{\pi,m}(\sigma)>0$.  Moreover, if the residual characteristic of $F$ satisfies $p>2n$, then we may take $d = 0$.
\end{lemmy}
\begin{proof}
  We outline the proof, whose details occupy the remainder of this subsection.

  We denote in what follows by $\rho(\pi)$ the \emph{depth} of $\pi$, defined (as in, e.g., \cite[\S5]{BH}) to be the minimum of $j/e$ taken over integers $j \geq 0$ and hereditary orders $(\mathfrak{A}, \mathfrak{P}, e)$ for which $\pi$ has a nonzero vector fixed by $1 + \mathfrak{P}^{j+1}$.

  We note that the depth $\rho(\pi)$ determines the ``principal depth'' $\delta(\pi)$ by the formula $\delta(\pi) = \lceil \rho(\pi) \rceil + 1$ (see \cite[Lem 4]{BH} or \cite[Proposition~2.1]{MY}).  In particular, the hypothesis $m \geq \delta(\pi) + d$ is equivalent to $m \geq \rho(\pi)+1 + d$.
  
  In the case $p \leq 2 n$, we appeal directly to Lemma~\ref{depth_pos_case}, below, using the value of $d$ given there.

  In the case $p > 2 n$ (so that $d =0$), our assumption $m \geq \delta(\pi)$ translates to $m \geq \rho(\pi) + 1$, and the conclusion follows immediately from Lemma~\ref{depth_pos_case} (when $m > \rho(\pi) + 1$) and Lemma~\ref{expl_depth} (when $m = \rho(\pi) + 1$).
\end{proof}

\begin{rem}
  Note that in \cite[Section~8.8]{HN}, it is verified that certain supercuspidal representations are what the authors call \textit{regular at depth $\mathfrak{q}^2$}.  This is a strong property, which in particular implies certain instances of the lemma above. However, for our application, it is important that our conclusion holds for \emph{all} supercuspidal representations $\pi$.
\end{rem}

\begin{rem}\label{rem28}
  The statement of Lemma~\ref{lm:exist_reg} suffices for our purposes, but we expect for any supercuspidal representation $\pi$, we have
  \begin{equation}\label{eq:cng2k3qbne}
    M(\pi,m)=
    \begin{cases}
      1 &\text{ if }m\geq \rho(\pi)+1,\\ 0&\text{ else.} 
    \end{cases} %\label{conj_loc}
  \end{equation}
  As we will see below (Remark~\ref{rem:exp_q_0}), the lower bound $\geq$ in \eqref{eq:cng2k3qbne} can be deduced from the Hales--Moy--Prasad conjecture; we expect the upper bound also follows, with further care. 
\end{rem}

The rest of this subsection is devoted to proving Lemma~\ref{lm:exist_reg}.  Recall the isomorphism $\beta \mapsto \theta_\beta$ from \eqref{eq:cngrzw1yis}.  Let $\beta \in \Mat_n(\mathfrak{k})$.  We call $\theta_\beta$ \emph{standard minimal} if $\beta$ is non-nilpotent, or equivalently, if the corresponding coset in $\Mat_n(\mathfrak{o})$ does not contain nilpotent elements.

\begin{rem}\label{remark:cngrk8lkkn}
  Our notion of ``standard minimal'' is from \cite{HM}.  Since we are assuming $m>1$, the extraordinary case that would appear for $m=1$ does not need to be considered at the moment.  We have included the adjective ``standard'' following the terminology of \cite[\S4]{Mur}.  %Note that in general, the notion of standard minimal $K$-type is closely related to the more general notion of unrefined minimal $K$-types introduced by Moy and Prasad. See for example \cite[Lemma~4.2]{Mur}.
\end{rem}

We may decompose the fixed subspace $\pi^{K(m)}$ as a direct sum of one-dimensional spaces on which $K(m-1)$ acts by the characters $\theta_\beta$, each occurring with some multiplicity:
\begin{equation*}
  \pi^{K(m)} = \bigoplus_{\beta \in \Mat_n(\mathfrak{k})} m_{\pi, m}(\beta) \cdot  \theta_\beta,
  \qquad
  m_{\pi, m}(\beta) \in \mathbb{Z}_{\geq 0}.
\end{equation*}
We denote by $\mathcal{B}_m(\pi)$ the set of $\beta$ for the multiplicity is positive, or equivalently, for which $\theta_\beta$ occurs in the restriction of $\pi^{K(m)}$ to $K(m-1)$.  To prove Lemma~\ref{lm:exist_reg}, it suffices to show that $\mathcal{B}_m(\pi)$ contains a regular element. This will be done in several steps.

An important tool in the proof is a quantitative version of the local character expansion due to Howe and Harish--Chandra, which we will briefly recall now. The character $\Theta_{\pi}$ of $\pi$ is the distribution defined via
\begin{equation}
  \Theta_{\pi}(f) = \mathrm{Tr}(\pi(f)) \text{ with }\pi(f)v = \int_{\mathrm{GL}_n(F)}f(g)\pi(g) v \, d g, \text{ for } f\in C_c^{\infty}(\mathrm{GL}_n(F)).\nonumber
\end{equation}
The local character expansion alluded to above allows us to asymptotically evaluate $\Theta_{\pi}$ using certain distributions associated to nilpotent orbits. Let $\alpha\in\mathcal{P}^0(n)$ be an ordered partition of $n$. Besides the nilpotent orbit $\mathfrak{o}_{\alpha}$, we can associate to $\alpha$ the standard parabolic subgroup $P_{\alpha}$ with Levi component $M_{\alpha}=\mathrm{GL}_{\alpha_1}\times \ldots\times \mathrm{GL}_{\alpha_{r(\alpha)}}$. Let $\eta_{\alpha}=\mathrm{Ind}_{P_{\alpha}(F)}^{\mathrm{GL}_n(F)}(1)$ denote the representation parabolically induced from the trivial representation of $P_{\alpha}(1)$. We are now ready to state the desired character expansion:

\begin{prop}[Howe, DeBacker]\label{how_exp_quant}
  There is a constant $C_0=C_0(n,F) \in \mathbb{Z}_{\geq 1}$ such that the following holds. Let $\pi$ be a supercuspidal representation of $\mathrm{GL}_n(F)$. Then there are constants $c_{\mathfrak{o}}(\pi)\in \C$ such that
  \begin{equation}
    \Theta_{\pi}(f) = \sum_{\alpha \in \mathcal{P}^{0}(n)}
    c_{\mathfrak{o}_{\alpha}}(\pi) \Theta_{\eta_{\alpha}}(f) \text{ for }f\in C_c^{\infty}(K_{\lfloor \rho(\pi)\rfloor +C_0}).\label{loc_char_exp}
  \end{equation}
  The coefficients of this expansion satisfy $c_{\mathfrak{o}}(\pi)\in \Z$ and $c_{\mathfrak{o}_{(1,\ldots,1)}}(\pi)=1$. Even more, if $p>2n$, then we can take $C_0=1$.
\end{prop}
The proof of this result with an unspecified value of $C_0$ reduces to a careful inspection of Howe's original argument. When $p>2n$, what we need is a direct consequence of DeBacker's results towards the Hales--Moy--Prasad conjecture. For the convenience of the reader, we will provide some details here.
\begin{proof}	
  Key to the proof is a careful analysis of certain spaces of distributions. In general, given an $\mathrm{Ad}_{\GL_n(F)}$-invariant subset $X$ of $\mathrm{Mat}_n(F)$, we write $\mathcal{J}(X)$ for the set of invariant distributions with support in $X$.  Basic examples of invariant sets are given by $X_j = \mathrm{Ad}_{\mathrm{GL}_n(F)}\mathrm{Mat}_n(\mathfrak{p}^j)$ with $j\in \Z$. In the current setting, a distribution is simply a linear functional on $C_c^{\infty}(\mathrm{Mat}_n(F))$. In particular, given a linear subspace $\mathcal{L}\subseteq C_c^{\infty}(\mathrm{Mat}_n(F))$ and $D\in \mathcal{J}(X)$, we write $\mathrm{res}_{\mathcal{L}}D$ for the restriction of $D$ to $\mathcal{L}$. Of particular interest are the spaces $\mathcal{L}_i=C_c(\mathrm{Mat}_n(F)/\mathrm{Mat}_n(\mathfrak{p}^i))$ and $\mathcal{L}^j=C^{\infty}(\mathrm{Mat}_n(\mathfrak{p}^j))$. For later reference, we record that
  \begin{equation}
    f\in \mathcal{L}^j \text{ if and only if }\widehat{f}\in \mathcal{L}_{-j}.\label{duality}
  \end{equation}
  This is a simple consequence of the definition of the Fourier transform and character orthogonality. 
  
  The first important space of distributions is $\mathcal{J}(\mathcal{N})$, where $\mathcal{N}$ is the invariant set of nilpotent elements in $\mathrm{Mat}_n(F)$. Canonical examples of distributions in this space are obtained by integrating over nilpotent orbits. Indeed, a nilpotent orbit $\mathfrak{o}_{\alpha}$ carries a natural $\mathrm{GL}_n(F)$-invariant Radon measure $\mu_{\mathfrak{o}_{\alpha}}$, and we can view $\mu_{\mathfrak{o}_{\alpha}}$ as distribution in $\mathcal{J}(\mathcal{N})$. It is well known that (see for example \cite[Lemma~3.3]{Ha})
  \begin{equation}
    \mathcal{J}(\mathcal{N}) = \langle \mu_{\mathfrak{o}_{\alpha}}\colon \alpha\in \mathcal{P}^0(n)\rangle_{\C}.\label{nilp}
  \end{equation} 
  
  According to the usual definition of the Fourier transform of distributions, we have
  \begin{equation}
    \mu_{\mathfrak{o}_{\alpha}}(\widehat{f})= \widehat{\mu}_{\mathfrak{o}_{\alpha}}(f). \nonumber
  \end{equation} 
  After normalizing the measures $\mu_{\mathfrak{o}_{\alpha}}$ appropriately, the distributions $\widehat{\mu}_{\mathfrak{o}_{\alpha}}$ can be matched with the characters $\Theta_{\eta_{\alpha}}$.  Let us make this more precise. We abbreviate $e(x) = 1+x$ and view it as a map $\mathrm{Mat}_n(F)\to \mathrm{Mat}_n(F)$. Note that $Y_j=e(X_j)$ is a conjugation invariant neighbourhood of $I_n$ in $\mathrm{GL}_n(F)$. We define $e^{\ast\ast}D(f) = D(f\circ e)$ for $f\in C_c^{\infty}(Y_j)$ and a distribution $D$ on $X_j$.  It is shown in \cite[Lemma~5]{How2}, or in \cite{vD}, that we can choose the normalization such that
  \begin{equation}
    e^{\ast\ast}\widehat{\mu}_{\mathfrak{o}_{\alpha}} = \Theta_{\eta_{\alpha}}\label{eq:orbs}
  \end{equation}
  on $Y_j$ for all $j> 0$.\footnote{A common alternative formulation of this fact is as follows. It is a result due to Harish--Chandra that the distribution $\widehat{\mu}_{\mathfrak{o}_{\alpha}}$ is represented by a locally integrable function, which is locally constant on the regular elements in $\mathrm{Mat}_n(F)$. By abuse of notation, this locally integrable function is also denoted by $\widehat{\mu}_{\mathfrak{o}_{\alpha}}$. In this setting, one can show that the germ of the character $\Theta_{\eta_{\alpha}}$ agrees with $\widehat{\mu}_{\mathfrak{o}_{\alpha}}$ on the set of regular topological nilpotent elements in $\mathrm{Mat}_n(F)$.}

  A key result from \cite{How2} is his Proposition~1, which we can now state as
  \begin{equation}
    \mathrm{res}_{\mathcal{L}_{-j}} \mathcal{J}(X_{-i}) \subseteq  \mathrm{res}_{\mathcal{L}_{-j}}\mathcal{J}(\mathcal{N}) \text{ for }j-i\geq C.\label{pre_hom}
  \end{equation} 
  Here, $C$ is some constant depending only on $F$ and $n$.  More precisely, one gets \eqref{pre_hom} for $i=0$ by applying \cite[Propostion~1]{How2} with $X=X_0$. The constant then depends on $X_0=\mathrm{Ad}_{\mathrm{GL}_n(F)}\mathrm{Mat}_n(\mathfrak{o})$, which we consider as fixed. The case of more general $i\neq 0$ is obtained by a simple re-scaling argument, which is discussed on \cite[p.312]{How2}.
  
  Next we want to apply the pre-homogeneity result \eqref{pre_hom} to a suitably modified version of the character $\Theta_{\pi}$ of the supercuspidal representation $\pi$. We follow the discussion on \cite[p. 319]{How2} and write
  \begin{equation*}
    \langle v,u\rangle \Theta_{\pi}(f) = d(\pi)\int_{{\rm GL}_n(F)/Z}\int_{{\rm GL}_n(F)}f(g)\theta_{u,v}(hgh^{-1})dg\, dh,
  \end{equation*}
  where $u,v\in \pi$, $d(\pi)$ is the formal degree of $\pi$, $f\in C_c^{\infty}(G)$, $Z$ is the centre of ${\rm GL}_n(F)$ and
  \begin{equation}
    \theta_{u,v}(g)=\langle \pi(g)v,u\rangle \nonumber
  \end{equation}
  is the matrix coefficient.  Put $s_0=\lfloor \rho(\pi)\rfloor+2\geq 2$ and $C_0 = C+2$. 
  
  To construct a suitable matrix coefficient $\theta_{v,v}$ we will have to allude to the classification of suppercuspidal representations of $\textrm{GL}_n(F)$ given in \cite[Theorem~8.4.1]{BK} for example.\footnote{We will be as brief as possible on these matters and refer the reader to \cite{BK} for a detailed exposition.} Since $\pi$ is supercuspidal, there is a maximal type $(J,\lambda)$, a finite field extension $E/F$ and an extension $\Lambda$ of $\lambda$ to $E^\times J$ such that
  \begin{equation*}
    \pi = \textrm{c-Ind}_{E^{\times}J}^{\textrm{GL}_n}(\Lambda).
  \end{equation*} 
  Without loss of generality, we can assume that $J\subseteq K$.  By the construction of $\lambda$ from a simple character of a simple stratum (see \cite[Section~3 and~5]{BK}), we see that $\Lambda$ is $K(s_0)$-invariant.  In particular, we can choose a unit vector $v\in \pi$ that is $K(s_0)$-invariant and for which $\theta_{v,v}$ is supported in $E^{\times}J$.

  We denote by $\widetilde{\Theta}_{\pi}$ the distribution on $\Mat_n(F)$ given as follows: for $f\in C_c^{\infty}(\mathrm{Mat}_n(F))$,
  \begin{equation*}
    \widetilde{\Theta}_{\pi}(f) = \Theta_{\pi}([f\cdot\mathbbm{1}_{X_{1}}]\circ e^{-1}).
  \end{equation*}
  By construction, $\widetilde{\Theta}_{\pi}\in \mathcal{J}(X_1)$.  Using suitably normalized measures, we also get
  \begin{equation}
    \widetilde{\Theta}_{\pi}(f) = d(\pi) \int_{{\rm GL_n}(F)/Z}\int_{X_{1}}f(x)[\theta_{v,v}\circ e](\mathrm{Ad}_h x)\, d x\, d h.\nonumber
  \end{equation}
  Let $\widehat{\Theta}_{\pi}\in \mathcal{J}(\mathrm{Mat}_n(F))$ denote the Fourier transform of $\widetilde{\Theta}_{\pi}$, defined using the trace pairing.  Then, by definition, we have
  \begin{equation}
    \widehat{\Theta}_{\pi}(\widehat{f}) = \widetilde{\Theta}_{\pi}(f) =
    d(\pi) \int_{{\rm GL}_n(F)/Z}\int_{\mathrm{Mat}_n(F)}\widehat{f}(x)\widehat{\theta}^{(1)}_{v,v}(\mathrm{Ad}_h x) \, d x \, d h, \nonumber
  \end{equation}
  where $\widehat{\theta}^{(1)}_{v,v}$ is the Fourier transform of $\tilde{\theta}=(\theta_{v,v}\circ e)\cdot \mathbbm{1}_{X_{1}}$.  Note that $\tilde{\theta}(X)=0$ unless $X\in X_1$ and $1+X\in E^{\times}J$. The first containment shows that $\det(1+X)\in \mathcal{O}^{\times}$ and one deduces that $1+X\in K$. With this at hand, one verifies that $\tilde{\theta}(X+Y) = \tilde{\theta}(X)$ for $Y\in \textrm{Mat}_n(\mathfrak{p}^{s_0}).$ The upshot is that the Fourier transform $\widehat{\theta}^{(1)}_{v,v}$ of $\tilde{\theta}$ has support in $\mathrm{Mat}_n(\mathfrak{p}^{-s_0})$. In particular,
  \begin{equation}
    \widehat{\Theta}_{\pi}\in \mathcal{J}(X_{-s_0}).\nonumber
  \end{equation}
  Using \eqref{pre_hom} and \eqref{nilp} allows us to write
  \begin{equation}
    \mathrm{res}_{\mathcal{L}_{-i}}\widehat{\Theta}_{\pi} = \sum_{\alpha \in \mathcal{P}^0(n)}c_{\mathfrak{o}_{\alpha}}(\pi)	\mathrm{res}_{\mathcal{L}_{-i}}\mu_{\mathfrak{o}_\alpha}\nonumber
  \end{equation}
  as long as $i-s_0\geq C$. In particular, if $f\in \mathcal{L}_i$, then one uses \eqref{duality} to obtain
  \begin{equation}
    \widetilde{\Theta}_{\pi}(f) = \widehat{\Theta}_{\pi}(\widehat{f}) =  \sum_{\alpha \in \mathcal{P}^0(n)} c_{\mathfrak{o}_{\alpha}}(\pi)\mu_{\mathfrak{o}_{\alpha}}(\widehat{f}).\nonumber
  \end{equation}
  The expansion as stated in the theorem now follows by using \eqref{eq:orbs}. The statement concerning the coefficients is given in \cite[Theorem~3]{How2}.

  We turn to situation when $p>2n$.  In this case, we can use a quantitative version of the character expansion due to S.\ DeBacker.  To see this, let us write $\mathfrak{g}_{\mathrm{reg}}$ for the set of regular elements in $\mathrm{Mat}_n(F)$.  Similarly, we denote the set of topological nilpotent elements by $\mathfrak{g}_{tn}$.  As in \cite[Section~1.3]{De}, we define the sets $\mathfrak{g}_{r}$ and $\mathfrak{g}_{r+}$.  The precise definition of these is not important to us; we will only use that $\mathfrak{g}_{0+} = \mathfrak{g}_{\mathrm{tn}}$ and $\mathfrak{p}^j \mathfrak{g}_{\mathrm{tn}}\subseteq \mathfrak{g}_{r+}$ if $j\geq r$.  By \cite[Theorem~3.5.2]{De}, we have
  \begin{equation}
    \Theta_{\pi}(1+x) =\sum_{\alpha\in \mathcal{P}^0(n)}a_{\alpha}(\pi)\widehat{\mu}_{\mathfrak{o}_{\alpha}}(x) \text{ for } x\in \mathfrak{g}_{\mathrm{reg}}\cap \mathfrak{g}_{\rho(\pi)+}.\nonumber
  \end{equation}
  Here, we slightly abuse notation and write $\Theta_{\pi}$ for the character of $\pi$ and for the corresponding germ. Note that the statement in \cite{De} holds under the assumption of certain hypotheses, all of which will be true for sufficiently large residual characteristic $p$.  That $p>2n$ suffices for $\mathrm{GL}_n$ is stated in \cite[(4.1)]{Mur2}, see also \cite[Theorem~11.5]{Mur}, which is all we need.  Indeed, after using \eqref{eq:orbs} to replace $\widehat{\mu}_{\mathfrak{o}_{\alpha}}$ with $\Theta_{\eta_{\alpha}}$ and integrating, we find that \eqref{loc_char_exp} holds for functions $f$ supported in $Y_j$ with $j>\rho(\pi)$.  We replace $j>\rho(\pi)$ by $j\geq \lfloor \rho(\pi)\rfloor +1$.  We are done since $K_j\subseteq Y_j$.
\end{proof}

\begin{rem}\label{rem:exp_q_0}
  The Hales--Moy--Prasad conjecture predicts that the local character expansion of Howe and Harish--Chandra always holds on $\mathfrak{g}_{\mathrm{reg}}\cap \mathfrak{g}_{\rho(\pi)+}$. In particular, this would imply that in Proposition~\ref{how_exp_quant}, one can always take $C_0=1$. However, to the best of our knowledge, this is still open for $\mathrm{GL}_n(F)$ in the cases $p\leq 2n$ and $n\geq 4$. Note that for $n=3$, the conjecture was proven without restrictions on $p$ in the PhD thesis of S.\ DeBacker. Some partial results towards the unrestricted Hales--Moy--Prasad conjecture were obtained in the 2022 PhD thesis of Y.\ Chen. 
\end{rem}

The preparations are now in place to treat the case $m > \rho(\pi)+1$
of Lemma~\ref{lm:exist_reg}.  As we saw in the proof of Lemma \ref{lm:equiv}, in this case we need to locate a regular nilpotent element in $\mathcal{B}_{\pi}(m)$.  We obtain the following result:

\begin{lemmy}\label{depth_pos_case}
  There exists $d = d(n,F) \in \mathbb{Z}_{\geq 0}$, depending only upon $F$ and $n$, with the following property.  Let $\pi$ be an irreducible supercuspidal representation of $\GL_n(F)$.  Let $m \geq 2$ be a natural number with $m > \rho(\pi) + 1 + d$.  Then $\mathcal{B}_m(\pi)$ contains a regular (nilpotent) element.  Moreover, if the residual characteristic of $F$ satisfies $p > 2 n$, then the conclusion holds with $d = 0$.
\end{lemmy}
\begin{proof}
  We take $d := C_0 - 1$, with $C_0$ as in the conclusion of Proposition~\ref{how_exp_quant}.  Let $Y_{\reg} \in \Mat_n(\mathcal{O})$ be any lift of the regular nilpotent Jordan block $Y_{(1, \dotsc, 1)} \in \Mat_n(\mathfrak{k})$.  Let $f \in C_c^\infty(G)$ denote the function supported on $K [m-1]$, and given there by $\operatorname{vol}(K(m-1))^{-1} \theta_{Y_{\reg}}^{-1}$.  This function acts in any representation as the $K(m-1)$-equivariant projection onto the $\theta_{Y_{\reg}}$-component of subspace of $K(m)$-fixed vectors.  Our assumptions concering $m$ imply that Proposition~\ref{how_exp_quant} (with $C_0 = d + 1$) applies to $\Theta_\pi(f)$, giving
  \begin{equation*}
    m_{\pi, m}(Y_{(1, \dotsc, 1)}) = \Theta_\pi(f) = \sum_{\alpha \in \mathcal{P}^0(n)} c_{\mathcal{O}_\alpha}(\pi) \Theta_{\eta_\alpha}(f)
    = \sum_{\alpha \in \mathcal{P}^0(n)} c_{\mathcal{O}_\alpha}(\pi) m_{\eta_\alpha, m}(Y_{(1, \dotsc, 1)}).
  \end{equation*}
  For the induced representations $\eta_{\alpha}$, it is easy to compute the multiplicities by hand.  Indeed, following the argument in \cite[Lemma~6]{How2}, we see that
  \begin{equation*}
    m_{\eta_{\alpha},m}(Y_{1,\ldots,1}) = \delta_{\alpha=(1,\ldots,1)}\cdot C_n(m), 	%\label{import_obs}
  \end{equation*}
  where $C_n(m) = m_{\eta_{(1,\ldots,1)},m}(Y_{1,\ldots,1})>0$ depends only on $n,m$.  We conclude that $m_{\pi, m}(Y_{(1, \dotsc, 1)}) > 0$, as required.
\end{proof}

\begin{rem}
  Note that \cite[Lemma~6.4]{CS} directly implies that there is $c'=c'(n)$ such that for sufficiently large $m$, the set $\mathcal{B}_{c'm}(\pi)$ contains a regular nilpotent element whenever $\pi^{K(m)}\neq \{0\}$. Unfortunately, this statement is too weak for us, since we cannot afford to lose powers of $p$ that depend on $m$ (i.e $v_p(q)$).
\end{rem}

As preparation for our treatment of the case $\rho(\pi)+1=m$, we record some simple observations concerning the behavior of regular elements and depth under twisting.
\begin{lemmy}\label{lm:equiv}
  Let $\pi$ be a supercuspidal representation, let $\chi$ be a character of $F^{\times}$ with conductor exponent $a(\chi)$ and let $m \geq 1$ be an integer.  Assume that (at least) one of the following conditions hold:
  \begin{enumerate}
  \item $m\geq a(\chi)-1$.
  \item $m \geq a(\chi)$ and $m - 1 > \rho(\pi)$.
  \end{enumerate}  
  Then $\mathcal{B}_{m}(\pi)$ contains a regular element if and only if $\mathcal{B}_m(\pi \otimes \chi)$ contains a regular element.
\end{lemmy}
\begin{proof}
  In either case, we have $m \geq a(\chi)$.  We may view $(\chi \circ \det) |_{K(m-1)}$ as a character of $K(m-1)/K(m)$, necessarily of the form $\theta_b$ for some scalar matrix $b \in \mathfrak{k} \hookrightarrow \Mat_n(\mathfrak{k})$.  The map $\beta \mapsto \beta + b$ then defines a bijection $\mathcal{B}_m(\pi) \rightarrow \mathcal{B}_m(\pi \otimes \chi)$.
  
  In the first case where $m \geq a(\chi) - 1$, the restriction of $\chi \circ \det$ to $K(m-1)$ is trivial, so $b = 0$ and $\mathcal{B}_m(\pi) = \mathcal{B}_m(\pi \otimes \chi)$.

  In the remaining case that $m - 1 > \rho(\pi)$, we see from \cite[Corollary~4.2]{HM} that each $\beta \in \mathcal{B}_m(\pi)$ is \emph{not} standard minimal, i.e., that $\beta$ is nilpotent.  The conclusion then follows from the fact that a nilpotent matrix $\beta \in \Mat_n(\mathfrak{k})$ is regular if and only if its sum $\beta + b$ with a scalar is regular.
\end{proof}

\begin{rem}\label{rem:depth}
  The proof of Lemma~\ref{lm:equiv} can be extended to produce the estimate
  \begin{equation}
    \rho(\pi \otimes \chi) \leq \max(a(\chi)-1,\rho(\pi)),\nonumber
  \end{equation}
  with equality for $a(\chi)-1\neq \rho(\pi)$. This is well known and follows for example from \cite[Proposition~2.2]{Co} after recalling that for $\mathrm{GL}_n$ the conductor exponent $a(\pi)$ of $\pi$ relates to the depth by $a(\pi)=n\rho(\pi)+n$.
\end{rem}

We turn finally to the case when $\rho(\pi)+1=m$. Here, \cite[Theorem~4.3]{HM} (see \cite[Theorem~4.1]{Mur} for a formulation closer to ours) implies that for all $\beta\in \mathcal{B}_m(\pi)$, the character $\theta_{\beta}$ of $K(m-1)$ is standard minimal.  The characteristic polynomial of $\beta$ will be denoted by $\mathrm{Ch}_{\beta}(Z)\in \mathfrak{k}[Z]$. Since $\theta_{\beta}$ occurs in the supercuspidal representation $\pi$, we can apply \cite[Proposition~4.6]{Ku}.  That result asserts that $\theta_{\beta}$ is ``non-split'' in the sense of \cite[Definition~4.5]{Ku}, which means that $\mathrm{Ch}_{\beta}(Z)$ is a power $f(Z)^r$ of some irreducible monic $f(Z)\in \mathfrak{k}[Z]$. Note that $n=r\lambda$, where $\lambda=\deg(f)$, and the polynomial $f$ depends only on $\pi$ but not on $\beta\in \mathcal{B}(\pi,m)$.

\begin{lemmy}\label{expl_depth}
  Assume that the residual characteristic $p$ of $F$ satisfies $p>2n$. Further, let $m>1$ and let $\pi$ be a supercuspidal representation with depth $\rho(\pi) = m-1$. Then there exists a regular $\beta \in \mathcal{B}_m(\pi)$.
\end{lemmy}
\begin{proof}
  We first note that if $\lambda=n$, then any $\beta \in \mathcal{B}_m(\pi)$ is automatically regular.  Indeed, in that case, the minimal polynomial equals the characteristic polynomial, which is one characterization of regularity.

  Now assume that $1<\lambda <n$.  In this case, the problem can be reduced to $\mathrm{GL}_r(E)$ for an extension $E/F$ using the Howe--Moy isomorphism of Hecke algebras.  To see this, we will adapt some notation from \cite[Section~3]{HM} (as well as from \cite{Mur}). Choose an arbitrary $\beta \in \mathcal{B}_m(\pi)$, and write $\beta = \overline{s}+\overline{n}_{\beta}$, where $\overline{s}\in Z(\mathfrak{k})$ is the semisimple part.  With $f(Z)^r$ the characteristic polynomial of $\beta$ as above, we have $f(\bar{s}) = 0$.  Since $f$ is irreducible, we see that $\bar{s}$ generates a finite extension $\mathfrak{l}/\mathfrak{k}$ of degree $\lambda$.  By choosing a $\mathfrak{k}$-basis for $\mathfrak{l}$, we may identify $\Mat_r(\mathfrak{l})$ with a subalgebra of $\Mat_n(\mathfrak{k})$.  Th element $\overline{n}_\beta$ lies in that subalgebra, since it commutes with $\overline{s}$.  We choose the basis so that the $\overline{n}_{\beta}$ defines a strictly upper-triangular element of that subalgebra.  Let $E/F$ be the unramified field extension of degree $\lambda$, which comes with a map $\mathfrak{o}_E \rightarrow \mathfrak{k} '$.  Let $s \in \mathfrak{o}_E$ be any preimage of $\bar{s}$ under this map.  Then $s$ generates $E/F$.  We lift our $\mathfrak{k}$-basis of $\mathfrak{l}$ to an $\mathfrak{o}$-basis of $\mathfrak{o}_E$, which gives an $F$-basis of $E$ and hence inclusions $E \subseteq \Mat_\lambda(F)$ and $E \subseteq \Mat_r(E) \subseteq \Mat_n(F)$ for which $\mathfrak{o}_E = E \cap \Mat_\lambda(\mathfrak{o})$.  Our assumption $\lambda > 1$ gives $\bar{s} \neq 0$, from which it follows that $\Mat_r(E) \subseteq \Mat_n(F)$ is simply the centralizer of $s$.  We define the (hereditary) order
  \begin{equation}
    \mathfrak{Q} := \left[
      \begin{matrix}
        \mathrm{Mat}_{\lambda}(\mathfrak{o}) & \dots & \dots & \mathrm{Mat}_{\lambda}(\mathfrak{o}) \\\varpi \mathrm{Mat}_{\lambda}(\mathfrak{o}) & \ddots & & \vdots \\ \vdots & \ddots & \ddots & \vdots\\ \varpi \mathrm{Mat}_{\lambda}(\mathfrak{o}) & \dots & \varpi\mathrm{Mat}_{\lambda}(\mathfrak{o}) & \mathrm{Mat}_{\lambda}(\mathfrak{o}) 
      \end{matrix}
    \right]\subseteq \mathrm{Mat}_r(\mathfrak{o}_E).\nonumber
  \end{equation}
  The group $Q:=Q(0):=\mathfrak{Q}^{\times}$ is a parahoric subgroup of $\mathrm{GL}_n(F)$.  Set $\tilde{\mathfrak{Q}}:=\mathfrak{Q}\cap \mathrm{Mat}_r(E)$ and $\tilde{Q}:=\tilde{Q}(0):=\tilde{\mathfrak{Q}}^{\times}$, so that $\tilde{Q}$ is an Iwahori subgroup of $\mathrm{GL}_r(E)$.  The element
  \begin{equation}
    \tilde{t} := \left(
      \begin{matrix}
        0 & I_{r-1} \\ \varpi & 0
      \end{matrix}
    \right) \in \mathrm{Mat}_{r}(E) \nonumber
  \end{equation}
  normalizes $\mathfrak{Q}$, and satisfies $\tilde{t}^r = \varpi I_{r}$.  For $l>0$, we define the subgroups
  \begin{equation}
    Q(l):=1+\tilde{t}^l\mathfrak{Q} \quad \text{ and }\quad \tilde{Q}(l):=Q(l)\cap \mathrm{GL}_r(E).\nonumber
  \end{equation}
  Set $\tilde{K}(i) := K(i)\cap \mathrm{GL}_r(E)$, and write $\tilde{\theta}_{\beta}$ for the restriction of $\theta_{\beta}$ to $\tilde{K}(m-1)$.  We have the inclusions
  \begin{equation}
    K(m)\subseteq Q(r(m-1))\subseteq K(m-1).\nonumber
  \end{equation}
  The upshot of this construction is that
  \begin{equation}
    \theta_{\beta}\vert_{Q(r(m-1))} = \theta_{\overline{s}}\vert_{Q(r(m-1))} \text{ and }\tilde{\theta}_{\beta}\vert_{\tilde{Q}(r(m-1))} = \tilde{\theta}_{\overline{s}}\vert_{\tilde{Q}(r(m-1))}.\nonumber
  \end{equation}
  We consider the Hecke algebra
  \begin{equation}
    \tilde{\mathcal{H}}:=\mathcal{H}(\mathrm{GL}_r(E)//\tilde{Q}(r(m-1)),\tilde{\theta}_{\overline{s}}).\nonumber
  \end{equation}
  From $\theta_{\overline{s}}\vert_{Q(r(m-1))}$, one can construct a representation $\sigma$ on a certain compact open subgroup $J$, following \cite[p.32]{How3}.
  % \pn{feels like we should add some more words about how $J$ and $\sigma$ relate to $\pi$.  I don't have the book handy at the moment.  Also, perahps this is a place where we use $p \leq 2 n$?}\ea{I also do not have the book with me at the moment, but the construction was rather technical (but standard). An extension of it can be found in the discussion leading towards Theorem~4.6 of \textit{Hecke Algebra Isomorphisms for GL(n) over a p-adic field} by Howe and Moy. (This is easier accessible, but I liked [How2] because it explicitly deals with the principal congruence subgroup.)}\ea{In [How2] there is no assumption like $p> 2n$ mentioned. In this inductive step of the argument we only inherit the restriction $p> 2n$ from [Mur1]. I also believe that for the parts of [Mur1] we use this assumption is irrelevant. (But I have kept it here in case I missed something.)}
  This leads to the Hecke algebra
  \begin{equation}
    \mathcal{H}:=\mathcal{H}(\mathrm{GL}_n(F)//J,\sigma).\nonumber
  \end{equation} 
  By \cite[Section~3, Theorem~1.1]{How3} there is an isomorphism
  \begin{equation}
    \eta\colon \tilde{\mathcal{H}}\to \mathcal{H},\nonumber
  \end{equation}
  which preserves natural $L^2$-structures. Given a supercuspidal representation $\pi$ of $\mathrm{GL}_n(F)$, we obtain a representation $\tilde{\pi}:=\eta^{\ast}(\pi)$ of $\mathrm{GL}_r(E)$ (not to be confused with the contragredient).  According to \cite[p.33, (c)]{How3}, the representation $\tilde{\pi}$ is also supercuspidal. As discussed in \cite[p.428]{Mur} (which assumes $p > 2 n$), if $\pi$ contains $(Q(r(m-1)),\theta_{\overline{s}})$, then $\tilde{\pi}$ contains $(\tilde{Q}(r(m-1)),\tilde{\theta}_{\overline{s}})$.  Furthermore, $\eta$ matches $Q(i)$ and $\tilde{Q}(i)$-stable subspaces (for $i$ not too small). This allows one to relate the standard minimal characters $\theta_{\beta}$ appearing in $\pi$ (see \cite[Corollary~6.3]{Mur}) with those appearing in $\tilde{\pi}$ (see \cite[Corollary~6.5]{Mur}). %Here $\tilde{\tau}$ is the $\tilde{K}(m-1)$-representation on $\tilde{\pi}^{\tilde{K}(m-1)}$ obtained by restricting $\tilde{\pi}$.
  An application of \cite[Proposition~8.4]{Mur} allows us to conclude that $\pi$ contains a regular standard minimal character $\theta_{\beta}$ if and only if $\tilde{\pi}$ contains one.  Indeed, because any regular element having semisimple part $\bar{s}$ is conjugate to $\bar{s} + Y_\alpha$, the representation $\pi$ (resp.\ $\tilde{\pi}$) is regular precisely when it contains the character attached to the element $\bar{s} + Y_\alpha$ of $\Mat_n(\mathfrak{k})$ (resp.\ $\Mat_r(\mathfrak{l})$), but Murnaghan's result relates the multiplicities of these characters by a positive scalar.  We have thus reduced the problem from $\mathrm{GL}_n(F)$ to $\mathrm{GL}_r(E)$ with $r<n$.

  It is now clear that as soon as we can treat the case $1=\lambda$ (or $n=r$), the result follows by induction on $n$.  Therefore we turn towards this exceptional situation.  Without loss of generality, we can assume that $\overline{s}=\mathrm{diag}(\overline{a},\ldots,\overline{a})$ with $\overline{a}\in \mathfrak{k}^{\times}$.  We fix any character $\chi\colon F^{\times}\to S^1$ that is trivial on $\varpi$, has conductor $\mathfrak{p}^m$ and satisfies
  \begin{equation}
    \chi(1+\varpi^{m-1}y) = \psi_F\Big(\frac{\overline{a}y}{\varpi}\Big) \text{ for all }y\in \mathfrak{o}.\nonumber
  \end{equation}
  As in the proof of Lemma~\ref{lm:equiv}, we see that $[\chi\circ \det]\vert_{K(m-1)} = \theta_{\overline{s}}$ and $a(\chi) \leq m$.  In particular, we find that the twisted representation $\pi' := \pi \otimes \chi$ contains $\theta_{\overline{n}_{\beta}}$ when restricted to $K(m-1)$.  Observe that by Remark~\ref{rem:depth}, we have $\rho(\pi')\leq \rho(\pi)$.  However, $\theta_{\overline{n}_{\beta}}$ is not standard minimal.  In particular, we must have $\rho(\pi')<\rho(\pi)$.  Indeed, equality would contradict \cite[Corollary~4.2]{HM}.  We are now in the situation where Lemma~\ref{depth_pos_case} applies to $\pi'$.  As a result, we see that $\mathcal{B}_m(\pi')$ contains a regular element.  We conclude the proof by using Lemma~\ref{lm:equiv}, which is possible because $m-1=a(\chi^{-1})-1 >\rho(\pi')$ and $\chi^{-1}\otimes \pi'=\pi$. 
\end{proof}

\subsection{Completion of the proof}\label{sec:cngrz3r57v}
We now complete the proof of Proposition~\ref{proposition:cngrlcwv3h}.  We have already noted that it suffices to treat the case of supercuspidal $\pi$ and $m \geq 2$.
We choose $d$ as in the conclusion of Lemma~\ref{lm:exist_reg}; in particular, $d = 0$ when $p > 2 n$.  We are given $m \geq 2$ with $m \geq \delta(\pi) + d$, and must verify that $\mathcal{S}_{\pi, m}(\boldsymbol{\psi}_m) \geq 1$.  In view of Lemma~\ref{lm:red_to_mult}, our task is to show that $\pi^{K(m)}$ contains some regular representation of $G_m = \GL_n(\mathfrak{o}/\mathfrak{p}^m)$.  This follows from Lemma~\ref{lm:exist_reg}.

\section{Properties of Kloosterman sets and sums}\label{sec:cnfm30d32o}
Here we recall some properties of Kloosterman sums that are important for our application.  We start by defining local and global Kloosterman sets.  We refer to Subsection \ref{sec:nor} for some notation used in the following.  Given $q\in \mathbb{N}$, $w\in W$ and $c\in \mathbb{N}^{n-1}$, we set
\begin{equation}
  X_{q,w}(c) = \{ xc^{\ast}wy\in U(\mathbb{Z})\backslash G_{w}(\mathbb{Q})\cap  \Gamma(q)^{\natural}/U_{w}(\mathbb{Z})\}. \nonumber
\end{equation}
We refer to this as the Kloosterman set of modulus $c$ (for the Weyl element $w$). It will be convenient to also define the corresponding (local) $p$-adic versions
\begin{equation}
  X_{q,w}^{(p)}(c) = \{ xc^{\ast}wy\in U(\mathbb{Z}_p)\backslash G_{w}(\mathbb{Q}_p)\cap K_p(q)^{\natural}/U_{w}(\mathbb{Z}_p)\}.\nonumber
\end{equation}

% Note that the Kloosterman sets are defined for any Weyl element $w\in W$ and any modulus $c$. Furthermore, 
We have the convenient properties:
\begin{enumerate}
\item For $p\nmid qc_1\cdots c_{n-1}$, we have $X_{q,w}^{(p)}(c) = U(\mathbb{Z}_p)c^{\ast}wU_{w_{\ast}}(\mathbb{Z}_p)$.
\item The diagonal embedding
  \begin{equation}
    \mathfrak{d}\colon X_{q,w}(c) \to \prod_{p\mid qc_1\cdots c_{n-1}} X_{q,w}^{(p)}(c) \nonumber
  \end{equation}
  is a bijection.
\item For $p\nmid q$, we have $K_p(q)^{\natural} = {\rm GL}_n(\mathbb{Z}_p)$.
\item\label{enumerate:cng12l96pw} In general, for
  \begin{equation*}
    (d_1\cdots d_{n-1},p)=1 \text{ and }p\nmid q,
  \end{equation*}
  we have
  \begin{equation}
    \# X_{q,w}^{(p)}(d c) = \# X_{1,w}^{(p)}(c) \nonumber
  \end{equation}
  via the usual bijection $x(d c )^{\ast} w y \mapsto (d^{\ast})^{-1}xd^{\ast}c^{\ast}wy$.
\end{enumerate}
In particular, for $q=1$ we have the nice factorization
\begin{equation}
  X_{1,w}(c) \cong \prod_{p\mid c_1\cdots c_{n-1}}X_{1,w}(c^{(p)}), \nonumber
\end{equation}
where $c^{(p)}=(p^{v_p(c_1)},\ldots, p^{v_p(c_{n-1})})$. Note that the bijection (Chinese Remainder Theorem) can be made explicit in principle.

Even though the Kloosterman sets are defined for all Weyl elements $w$ and all moduli $c$,  we will only be interested in special situations. We call $w$ admissible if it is of the form $w=w_{n_1,\ldots,n_r}$ as in \eqref{admis} for natural numbers satisfying $n_1+\ldots+n_r=n$. Furthermore, we say that a tuple $(w,c)$ is $(\theta_{M},\theta_N^v)$-relevant for a pair of characters $\theta_M, \theta_N$ and $v\in V$ as in \eqref{char}, if $w$ is admissible and if $c$ satisfies
\begin{equation}
  M_{n-i}\frac{c_{n-i+1}c_{n-i-1}}{c_{n-i}^2} = \frac{v_{w(i)+1}}{v_{w(i)}}N_{n-w(i)},\label{compat_coord}
\end{equation}
for all $1\leq i \leq n-1$ with $w(i)+1=w(i+1)$. This condition is precisely \cite[(4.3)]{AB} and is actually equivalent to the compatibility assumption \eqref{compat} below.

The structure of the lattice $\Gamma(q)^{\natural}$ and the Weyl element $w$ imposes some divisibility constraints on $c$ in order for $X_{w,q}(c)\neq \emptyset$. The following result collects the conditions relevant for our purposes:
\begin{lemmy}\label{lm:divisibility}
  Let $n\geq 4$ and let $w$ be admissible.
  \begin{enumerate}
  \item If $w\in W\setminus\{ 1,w_{\ast}\}$, then $X_{q,w}(c) = \emptyset$ unless there is $i\in \{1,\ldots,n-1\}$ such that $q^{n+2}\mid c_i$.
  \item If $X_{q,w_{\ast}}(c)\neq \emptyset$, then $q^n\mid c_i$ for all $i=1,\ldots,n-1$.
  \end{enumerate}
\end{lemmy}
\begin{proof}
  Almost everything we need follows from \cite[Lemma~4.2]{AB}. The only issue arises for the Weyl element $w_1$. Indeed, according to the remark after \cite[Lemma~4.2]{AB}, we only have $X_{q,w_1}(c)=\emptyset$ unless
  \begin{equation*}
    (n-1)v_p(\gamma)\geq (n+1)v_p(q),
  \end{equation*}
  where $c=(m\gamma^{n-1},\gamma^{n-2},\ldots,\gamma^2,\gamma)$ for $\gamma\in \mathbb{N}$.  This is slightly weaker than what we are claiming, but this can be fixed using the following trick.

  Suppose we have $c=(c_1,\ldots,c_{n-1})$ such that $X_{q,w_1}(c)\neq \emptyset$. Thus, there are $x\in U(\Z)\backslash U(\Q)$ and $y\in U_{w_1}(\Q)/U_{w_1}(\Z)$ such that
  \begin{equation}
    xc^{\ast}w_1y \in \Gamma(q)^{\natural}.\nonumber
  \end{equation}
  We observe that the map $\gamma \mapsto w_lg^{-\top}w_l^{-1}$ preserves $\Gamma(q)^{\natural}$. Further, we compute that $w_lw_1^{\top}w_l^{-1}=w_1'$ and $w_l(c^{\ast})^{-\top}w_l^{-1} = c_{\mathrm{op}}^{\ast}$ with $c_{\mathrm{op}}=(c_{n-1},\ldots,c_1)$. In particular we find that, for $\tilde{x}=w_lx^{-\top}w_l^{-1}$, $\tilde{y}= w_ly^{-\top}w_{l}^{-1}$ we have 
  \begin{equation}
    \tilde{x}c_{\mathrm{op}}^{\ast}w_1'\tilde{y}\in \Gamma(q)^{\natural}.\nonumber
  \end{equation} 
  Since $\tilde{x}\in U(\Z)\backslash U(\Q)$ and $\tilde{y}\in U_{w_{1}'}(\Q)/U_{w_1'}(\Z)$ we find that
  \begin{equation}
    X_{q,w_1'}(c_{\mathrm{op}})\neq \emptyset.\nonumber
  \end{equation}
  Using the result of the Lemma for $w_1'$ we find that at least one entry of $c_{\mathrm{op}}$ and thus at least one entry of $c$ must be divisible by $q^{n+2}$. This completes the proof also for $w_1$.
\end{proof}

Understanding the size of the Kloosterman sets is (in general) a more complicated problem. We have the trivial bound
\begin{equation}
  \# X_{q,w}^{(p)}(c) \ll  \mathcal{N}_{p^{v_p(q)}}\vert c_1\cdots c_{n-1}\vert_p^{-1}\label{eq:triv_gen_ks}
\end{equation}
with $\mathcal{N}_q$ as in \eqref{constants}. For $(q,p)=1$, this is due to Dabrowski and Reeder \cite{DR}. %For $q>1$, the additional factor $\mathcal{N}_q$ appears because $U(\mathbb{Q})\cap \Gamma(q)^{\natural}\not\subseteq U(\mathbb{Z})$. 
An indication of how this bound can be proved is given in Remark~\ref{proof_triv} below.

For our application, we need something stronger than \eqref{eq:triv_gen_ks}. Producing such general bounds leads to quite complicated counting problems. %A taste of this can be found in \cite[Lemma~4.3]{AB}, where the Weyl element $w_{\ast}$ is treated. 
The general expectation is that the bound
\begin{equation}
  \# X_{q,w}^{(p)}(c) \ll  \frac{\mathcal{N}_{p^{v_p(q)}}}{p^{v_p(q)(n-1) + o(1)}}\cdot \vert c_1\cdots c_{n-1}\vert_p^{-1} \label{general_expectation}
\end{equation}
holds and is essentially sharp. Here
 $$\mathcal{N}_{p^{v_p(q)}}  = [K_p(q)^{\natural}\cap U(\mathbb{Q}_p) :  U(\mathbb{Z}_p)] $$%\subsetneq K_p(q)^{\natural}\cap U(\mathbb{Q}_p)$ 
 and 
$$p^{v_p(q)(n-1)} = [T_0(\mathbb{Z}_p) : K_p(q)^{\natural}\cap T_0(\mathbb{Z}_p)].$$
% and the saving of $p^{v_p(q)(n-1)}$ comes from $K_p(q)^{\natural}\cap T_0(\mathbb{Q}_p) \subsetneq T_0(\mathbb{Z}_p)$. 
%Note that \eqref{general_expectation} agrees with the expectation for the Weyl element $w_{\ast}$ stated in the remark below \cite[Lemma~4.3]{AB}.

We will first look at the case $n=3$, where the direct counting is still doable.

\begin{lemmy}\label{lm:gl3_set_est_vor}
  We have the following results:
  \begin{enumerate}
  \item Let $M=N=(m,1)$.  Then $(w_1,c)$ is $(\theta_M,\theta_N^v)$-relevant if
    \begin{equation}
      c=(\pm m\gamma^2,\gamma) \text{ (where the sign $\pm$ depends on $v$).}	\nonumber	
    \end{equation}	
    Similarly, if $(w_1',c)$ is $(\theta_M,\theta_N^v)$-relevant, then
    \begin{equation}
      c=\Big(\gamma,\pm \frac{1}{m}\gamma^2\Big).	\nonumber	
    \end{equation}

  \item For $w_1$ and $w_1'$ the size of the local Kloosterman sets depends only on the valuation of the moduli.  More precisely, we have
    \begin{equation*}
      \# X_{q,w_1}^{(p)}(c) = \# X_{q,w_1}^{(p)}((p^{v_p(c_1)},p^{v_p(c_2)}))\quad \text{and} \quad \# X_{q,w_1'}^{(p)}(c) = \# X_{q,w_1'}^{(p)}((p^{v_p(c_1)},p^{v_p(c_2)})).
    \end{equation*}
  \item For $p\mid q$ and $c=(c_1,c_2)$ we have
    \begin{equation}
      \# X_{q,w_1'}^{(p)}(c) \ll \delta_{\substack{p^{2v_p(q)}\mid c_1,\\ p^{4v_p(q)}\mid c_2}}\cdot \frac{\mathcal{N}_{p^{v_p(q)}}}{p^{2v_p(q)}}\cdot \vert c_1c_2\vert_p^{-1}\nonumber
    \end{equation}
    and
    \begin{equation}
      \# X_{q,w_1}^{(p)}(c) \ll \delta_{\substack{p^{2v_p(q)}\mid c_2,\\ p^{4v_p(q)}\mid c_1}}\cdot \frac{\mathcal{N}_{p^{v_p(q)}}}{p^{2v_p(q)}}\cdot \vert c_1c_2\vert_p^{-1}\nonumber
    \end{equation}
  \end{enumerate}
\end{lemmy}
\begin{proof}
  The statement (1) for $w_1$ (resp.\ $w_1'$) follows directly from \eqref{compat_coord} with $i=1$ (resp.\ $i=2$).

  Turning to (2), we begin with the case of $w_1'$. We may reduce to the case that $c$ is of the form $(p^t,p^{r+t})$.
  Indeed, suppose $c=(ap^{\alpha},bp^{\beta})$ with $a,b\in \mathbb{Z}_p^{\times}$. Let $xc^{\ast}w_1'y\in K_p(q)^{\natural}$, so that $x$ and $y$ contribute to $X_{q,w_1'}^{(p)}(c)$. For $t\in T(\Z_p)$ put $t^{w_1'} = w_1't(w_1')^{-1}$ . We have
  \begin{equation}
    \tilde{x}t^{-1}c^{\ast}t^{w_1'}w_1'\tilde{y} \in K_p(q)^{\natural},\nonumber
	\end{equation}
	where $\tilde{x}=t^{-1}xt$ and $\tilde{y}=t^{-1}yt$. For $t=\diag(t_1,t_2,t_3)$ one easily computes that
	\begin{equation}
		t^{-1}c^{\ast}t^{w_1'} = \left(
		\begin{matrix}
			\frac{t_3}{t_1b}p^{-\beta} & 0&0\\0& \frac{t_1b}{t_2a}p^{\beta-\alpha} &0\\0&0&\frac{t_2a}{t_3}p^{\alpha} 
		\end{matrix}
		\right). \nonumber
	\end{equation}
	Thus by choosing $t_1=b^{-1}$, $t_2= a^{-1}$ and $t_3=1$ we find that conjugation with $t$ defines a bijection between $X_{q,w_1'}^{(p)}(c)$ and $X_{q,w_1'}^{(p)}((p^{\alpha},p^{\beta}))$.
	
	Having made the reduction to $c=(p^t,p^{r+t})$ we count elements in the Kloosterman set directly. Recall that we must count the $x,y,z,\alpha,\beta\in \mathbb{Q}_p/\mathbb{Z}_p$ for which the matrix
  \begin{equation*}
    \left(
      \begin{matrix}
        1 & x & y \\ 0& 1 & z \\ 0&0&1
      \end{matrix}
    \right)c^{\ast}w_1'\left(
      \begin{matrix}
        1 & 0 & \alpha \\ 0&1&\beta \\ 0&0&1
      \end{matrix}
    \right) = \left( 
      \begin{matrix}
        p^rx & p^ty & p^{-t-r}+p^rx\alpha+p^ty\beta \\ p^r & p^tz & p^r\alpha+p^tz\beta \\ 0& p^t& p^t\beta 
      \end{matrix}
    \right)
  \end{equation*}
  lies in
  \begin{equation*}
    K_p(q)^\natural = \GL_3(\mathbb{Q}_p) \cap \left( 1 +
      \begin{pmatrix}
        q \mathbb{Z}_p        & \mathbb{Z}_p & q^{-1} \mathbb{Z}_p \\
        q^2 \mathbb{Z}_p                              & q \mathbb{Z}_p &  \mathbb{Z}_p \\
        q^3 \mathbb{Z}_p                              & q^2 \mathbb{Z}_p &  q \mathbb{Z}_p \\
      \end{pmatrix} \right).
  \end{equation*}
  From the below-diagonal entries, we see that we can assume $r,t\geq 2v_p(q)$.  From the diagonal, we obtain
  \begin{equation}
    p^rx,p^tz,p^t\beta\in 1+q\mathbb{Z}_p.\nonumber
  \end{equation}
  We set $x=p^{-r}(1+qx')$ for $x'\in \mathbb{Z}_p/p^{r-v_p(q)}\mathbb{Z}_p$. Similarly define $z=p^{-t}(1+qz')$ and $\beta = p^{-t}(1+q\beta')$ for $z',\beta'\in \mathbb{Z}_p/p^{t-v_p(q)}\mathbb{Z}_p$. The middle entry of the top row suggests to put $y=p^{-t}y'$ with $y'\in \mathbb{Z}_p/p^t\mathbb{Z}_p$. The last entry of the middle row allows us to define $\alpha$ as
  \begin{equation}
    \alpha=p^{-r}\alpha' - p^{-r-t}(1+qz')(1+q\beta') \text{ for }\alpha'\in \mathbb{Z}_p/p^r\mathbb{Z}_p.\nonumber
  \end{equation}
  Estimating trivially at this point shows that there are at most $p^{3t+2r-3v_p(q)}$ possibilities for
  \begin{displaymath}
    \begin{split} 
      &(x',y',z',\alpha',\beta')\\
      & \in (\mathbb{Z}_p/p^{r-v_p(q)}\mathbb{Z}_p) \times ( \mathbb{Z}_p/p^t\mathbb{Z}_p) \times (\mathbb{Z}_p/p^{t-v_p(q)}\mathbb{Z}_p) \times  (\mathbb{Z}_p/p^r\mathbb{Z}_p)\times (\mathbb{Z}_p/p^{t-v_p(q)}\mathbb{Z}_p) .
    \end{split}
  \end{displaymath}
  Substituting all previous definitions into the upper right corner, we obtain
  \begin{equation}
    q p^{-r-t}(x'+z'+\beta'+q(x'z'+x'\beta'+z'\beta')+q^2x'z'\beta')-p^{-r}\alpha'(1+qx')-p^{-t}y'(1+q\beta')\in  \frac{1}{q}\mathbb{Z}_p.\nonumber
  \end{equation}
  or equivalently
  \begin{equation}
    (x'+z'+\beta'+q(x'z'+x'\beta'+z'\beta')+q^2x'z'\beta')-\frac{p^t}{q} \alpha'(1+qx')-\frac{p^r}{q}y'(1+q\beta')\in  \frac{p^{r+t}}{q^2}\mathbb{Z}_p.\nonumber
  \end{equation}
  It is not obvious from this condition that this is well-defined, but this is clear from the derivation. We can hence enlarge the range of $x'$ (say) to $\mathbb{Z}_p/p^{r + t-2v_p(q)}\mathbb{Z}_p$ at the cost of dividing our count by $p^{t - v_p(q)}$. Then the previous condition determines $x'$, and so the final upper bound for the possibilities of $(x,y,z,\alpha,\beta)$ is $p^{2t+r-v_p(q)}$. The result follows after recalling that $\mathcal{N}_{p^{v_p(q)}}=p^{v_p(q)}$ in this case.

  Finally, we do not have to count the elements in $X_{q,w_1}^{(p)}(c)$ directly.  Instead, we use the same trick as in the proof of Lemma~\ref{lm:divisibility} and define the involution $\iota(g) =w_lg^{-\top}w_l^{-1}$. Note that $\iota(K_p(q)^{\natural}) = K_p(q)^{\natural}$. Note that, if $y\in U_{w_1}(\mathbb{Q}_p)/U_{w_1}(\mathbb{Z}_p)$, then $\iota(y)\in U_{w_1'}(\mathbb{Q}_p)/U_{w_1'}(\mathbb{Z}_p)$. Furthermore $\iota(w_1) = w_1'$ and $\iota(c^{\ast}) = c_{\mathrm{op}}^{\ast}$. The result of these observations is that $\iota$ induces a bijection between $X_{q,w_1}^{(p)}(c)$ and $X_{q,w_1'}^{(p)}(c_{\mathrm{op}})$. Since we have already counted the latter set above, we are done.
\end{proof}

The long Weyl element $w_l$, which for $n=3$ agrees with $w_{\ast}$, features a slightly different behaviour.

\begin{lemmy}\label{deg3_prime_long}
  The following statements hold.
  \begin{enumerate}
  \item We have $\# X_{q,w_{\ast}}^{(p)}(c) = \# X_{q,w_{\ast}}^{(p)}(c_{\mathrm{op}})$ and
    \begin{equation}
      \# X_{p,w_{\ast}}^{(p)}((ac_1,c_2)) = \# X_{p,w_{\ast}}^{(p)}((c_1,a^{-1}c_2)),\label{eq:swap}
    \end{equation}
    for $a\in \mathbb{Z}_p^{\times}$.
  \item If $X_{q,w_{\ast}}^{(p)}(c)\neq \emptyset$, then we must have $v_p(c_1),v_p(c_2)\geq 3v_p(q)$. Moreover, if $3v_p(q)\leq \min(v_p(c_1),v_p(c_2))<4v_p(q)$, then $X_{q,w_{\ast}}^{(p)}(c)= \emptyset$ unless $c_1\equiv c_2\, (\text{\rm mod }p^{4v_p(q)}).$
  \item For $\epsilon > 0$, we have the general upper bound
    \begin{equation}
      \# X_{q,w_{\ast}}^{(p)}(c) \ll \frac{\mathcal{N}_{v_p(q)}}{p^{2v_p(q)(1 - \epsilon)}}\cdot \vert c_1c_2\vert_p^{-1}. \nonumber
    \end{equation}
  \end{enumerate}
\end{lemmy}
\begin{proof}
  To see the first claim, one uses ideas from the proof of Lemma~\ref{lm:gl3_set_est_vor}. Indeed, observing that $\iota(w_{\ast}) = w_{\ast}$, we can swap   $c$ and $c_{\mathrm{op}}$. Furthermore, conjugation by $t=\diag(t_1,t_2,t_3)$ induces a bijection between $X_{p,w_{\ast}}^{(p)}((c_1,c_2))$ and $X_{p,w_{\ast}}^{(p)}((\frac{t_1}{t_3}c_1,\frac{t_1}{t_3}c_2))$. Thus \eqref{eq:swap} follows by choosing $t_1=t_2=1$ and $t_3=a$.
  
  We now come to the counting result. Using the observations above, we can assume without loss of generality that $c=(p^u,sp^{u+v})$ with $u>0$, $v \geq 0$ and $s\in \mathbb{Z}_p^{\times}$. As before, we directly compute
  \begin{equation*}
    \left(
      \begin{matrix}
        1 & x & y \\ 0&1&z \\ 0&0&1
      \end{matrix}
    \right) c^{\ast} w_{\ast}\left(
      \begin{matrix}
        1 & \beta &\alpha \\ 0 &1&\gamma \\ 0&0&1
      \end{matrix}
    \right) = \left( 
      \begin{matrix}
        y p^u & xsp^v+y\alpha p^u & -\frac{1}{sp^{u+v}}+x\gamma sp^v + y\beta p^u \\ zp^u & sp^v+p^u\alpha z & s\gamma p^v+z\beta p^u \\ p^u & \alpha p^u & \beta p^u
      \end{matrix}
    \right) \in K_p(q)^{\natural}.
  \end{equation*}
  We immediately see that we can assume $u\geq 3v_p(q)$ and put
  \begin{align}
    z &= p^{2v_p(q)-u}z' \text{ for }z'\in \mathbb{Z}_p/p^{u-2v_p(q)}\mathbb{Z}_p,\nonumber\\
    \alpha &= p^{2v_p(q)-u}\alpha' \text{ for }\alpha'\in \mathbb{Z}_p/p^{u-2v_p(q)}\mathbb{Z}_p,\nonumber \\
    y &= p^{-u}(1+p^{v_p(q)}y') \text{ for }y'\in \mathbb{Z}_p/p^{u-v_p(q)}\mathbb{Z}_p \text{ and }\nonumber \\
    \beta &= p^{-u}(1+p^{v_p(q)}\beta') \text{ for }\beta'\in \mathbb{Z}_p/p^{u-v_p(q)}\mathbb{Z}_p.\nonumber
  \end{align}
  The remaining conditions read
  \begin{align}
    &sp^v+\alpha'z' p^{4v_p(q)-u} \in 1+q\mathbb{Z}_p, \label{middel} \\
    & \gamma sp^v + z'p^{2v_p(q)-u} + \beta'z'p^{3v_p(q)-u} \in \mathbb{Z}_p, \label{dright} \\
    & xsp^v+\alpha'p^{2v_p(q)-u}+\alpha'y'p^{3v_p(q)-u} \in \mathbb{Z}_p\text{ and }\label{uleft} \\
    & p^{-u}-\frac{1}{sp^{u+v}}+\gamma x sp^v +  p^{v_p(q)-u}(y'+\beta')+p^{2v_p(q)-u}y'\beta' \in \frac{1}{q}\mathbb{Z}_p.\label{uright}
  \end{align}
  We consider $3v_p(q)\leq u<4v_p(q)$. Then we see that \eqref{middel} can only be true if $s\equiv 1\text{ mod }p^{4v_p(q)-u}$ and $v=0$. (Without loss of generality we assume $s=1+s'p^{4v_p(q)-u}$.) This completes in particular the proof of (2). 
  
  The conditions \eqref{dright} and \eqref{uleft} now determine $\gamma = -s^{-1}(z'p^{2v_p(q)-u}+\beta'z'p^{3v_p(q)-u} )$ and $x=-s^{-1}(\alpha'p^{2v_p(q)-u}+\alpha'y'p^{3v_p(q)-u})$. Estimating trivially at this point gives at most $p^{4u-6v_p(q)}$ possible choices $(y',z',\alpha',\beta')$ which determine $x$ and $\gamma$. As before we invoke now the last condition \eqref{uright} which after inserting all previous definitions reads
  \begin{displaymath}
    \begin{split}
      \frac{q^4}{p^{2u}} \Big( \frac{\alpha'z'}{s} (1 +\beta' q)(1 + y'q) + \frac{s'}{s}  + \frac{  p^u}{q^3} (\beta'+ y' + \beta'y' q)\Big) \in \frac{1}{q} \mathbb{Z}_p . 
    \end{split}
  \end{displaymath}
  By the same argument in the previous lemma that determines $\alpha' z'$ modulo $p^{2u}/q^5$, and we get a total count of
  \begin{equation*}
    (2u - v_p(q))p^{2u - v_p(q)}
  \end{equation*}
  for the 6-tuple $(y, z, \alpha, \beta, x, y)$, as desired for (3) (as $\mathcal{N}_{v_p(q)} = q$ in the present case).

  % However a careful analysis of the final condition \eqref{uright} yields another saving of $p^{2u-5v_p(q)}$, which gives the desired bound. Let us just note that for $u=3v_p(q)$ condition \eqref{uright} simplifies to
  % \begin{equation}
  %   y'+\beta'+z'\alpha' \in q\mathbb{Z}_p\nonumber
  % \end{equation}
  % and one easily sees that this determines $y'$ modulo $q$. We will omit the general case here since we will see a similar %equation below.

  We turn towards the remaining cases namely $u\geq 4v_p(q)$ and $v\geq 0$. First observe that \eqref{middel} yields
  \begin{equation}
    \alpha'z' \equiv p^{u-4v_p(q)}-sp^{u+v-4v_p(q)} \, (\text{mod }p^{u-3v_p(q)}). \nonumber
  \end{equation}
  This has now always solutions and without loss of generality we can assume that
  \begin{equation}
    \alpha' = (z')^{-1}(1-sp^v)p^{u-4v_p(q)} + \alpha''p^{u-3v_p(q)} \text{ for }\alpha''\in \mathbb{Z}_p/p^{v_p(q)}\mathbb{Z}_p.\nonumber
  \end{equation}
  (The other cases will contribute at most the same to the Kloosterman set.) Next we use \eqref{dright} and \eqref{uleft} to choose
  \begin{equation}
    x=-s^{-1}p^{-v}(\alpha'(1+p^{v_p(q)}y')p^{2-u}+x'), \quad \gamma=-s^{-1}p^{-v}(z'(1+p^{-v_p(q)}\beta')p^{2-u}+\gamma') \nonumber
  \end{equation}
  for $x',\gamma'\in \mathbb{Z}_p/p^v\mathbb{Z}_p$. Estimating trivially at this point gives a bound of the form $\ll p^{3u+2v-3v_p(q)}$. The final condition \eqref{uright} reads
  \begin{multline}
    p^{2v_p(q)-u-v}\left[\alpha''z'+\beta'+y'+ p^{v_p(q)}(y'\beta'+\alpha''y'z'+\alpha''z'\beta'+x'z')+p^{2v_p(q)}x'z'\beta' \right] \\
    +p^{-v-v_p(q)}\left[(z')^{-1}\gamma'+p^{v_p(q)}((z')^{-1}y'\gamma'+\alpha''\gamma')+p^{2v_p(q)}(x'\gamma'+\alpha''y'\gamma')\right] \\
    +p^{4v_p(q)-u}(\alpha''y'z'\beta')+p^{-2v_p(q)}((z')^{-1}\gamma') \in \mathbb{Z}_p.\nonumber
  \end{multline}
  By the now familiar argument as in previous situations, this lets us save $p^{u+v-2v_p(q)}$ and the result follows.
\end{proof}

For $\mathrm{GL}_n$ with $n\geq 4$ we will study the Weyl element $w_{\ast}$ more carefully. Here we will still count the elements in the (ramified) Kloosterman sets, but we will do so not as directly as in the case $n=3$.

\begin{lemmy}\label{lem:kszeta}
  For $\alpha,\beta\geq 0$ and $s\in \mathbb{Z}_p^{\times}$ we put $c(\alpha,\beta;s) = (p^{\alpha},p^{\alpha+\beta}s,\ldots,p^{\alpha+(n-2)\beta}s^{n-2})$. We have
  \begin{equation}
    \# X_{q,w_{\ast}}^{(p)}(c(\alpha,\beta;s)) \ll \frac{\mathcal{N}_{p^{v_p(q)}}}{p^{v_p(q)(n-1)}}p^{(n-1+\epsilon)\alpha+\frac{(n-1)(n-2)}{2}\beta}.\nonumber
  \end{equation}
  Furthermore, $X_{q,w_{\ast}}^{(p)}(c(\alpha,\beta;s)) =\emptyset$ unless $\alpha\geq n v_p(q)$.
\end{lemmy}
\begin{proof}
  To simplify notation, we write
  \begin{equation}
    t_0=c(\alpha,\beta;s)^{\ast}=\left(
      \begin{matrix}
        p^{-\alpha-(n-2)\beta}s^{-(n-2)} & 0& 0\\ 0& sp^{\beta}\cdot I_{n-2} & 0\\ 0&0&p^{\alpha}
      \end{matrix}
    \right)\in T_0(\Q_p).\nonumber
  \end{equation}	

  Let us fix $Y\in U_{w_{\ast}}(\Q_p)$ for a moment and assume that $Xt_0w_{\ast}Y=k$ and $X't_0w_{\ast}Y=k'$ with $k,k'\in K_p(q)^{\natural}$ and $X, X' \in U(\mathbb{Q}_p)$. Then we have
  \begin{equation}
    X'X^{-1}= k'k^{-1} \in K_p(q)^{\natural}\cap U(\mathbb{Q}_p).\nonumber
  \end{equation}
  In other words, if $t_0w_{\ast}Y\in U(\Q_p)K_p(q)^{\natural}$, then $X\in U(\Q_p)$ with $Xt_0w_{\ast}Y \in K_p(q)^{\natural}$ is uniquely determined up to left multiplication by elements in $K_p(q)^{\natural}\cap U(\mathbb{Q}_p)$. Thus we find that
  \begin{equation}
    \# X_{q,w_{\ast}}^{(p)}(c(\alpha,\beta;s)) = \mathcal{N}_{p^{v_p(q)}}\cdot \#\{Y\in U_{w_{\ast}}(\Q_p)/U_{w_{\ast}}(\Z_p)\colon t_0w_{\ast}Y \in U(\Q_p)K_p(q)^{\natural} \}.\nonumber
  \end{equation}
  Recall that $K_p(q)$ has a Iwahori decomposition with respect to the minimal Borel subgroup, so we can write any $k\in K_p(q)^{\natural}$ as
  \begin{equation}
    k = D_q^{-1}u_kt_k\overline{u}_kD_q \text{ with }t_k\in T_0(1+q\mathbb{Z}_p),\, u_k\in U(q\mathbb{Z}_p) \text{ and }\overline{u}_k\in U(q\mathbb{Z}_p)^{\top}.\nonumber
  \end{equation}
  Thus we can write
  \begin{equation}
    w_{\ast}Y = u_Yt(Y)\overline{u}_Y \in U(\Q_p)T_0(\Q_p)D_q^{-1}U(p\Z_p)^{\top}D_q.\nonumber
  \end{equation}
  If $Y$ contributes to the count for $\# X_{q,w_{\ast}}^{(p)}(c(\alpha,\beta;s))$ we must have $t_Y\in t_0^{-1}\cdot T_0(1+q\Z_p)$. We can now rewrite our count as follows
  \begin{align}
    \# X_{q,w_{\ast}}^{(p)}(c(\alpha,\beta;s)) &= \mathcal{N}_{p^{v_p(q)}}\cdot \int_{U_{w_{\ast}}(\Q_p)} \mathbbm{1}_{U(\Q_p)t_0^{-1}K_p(q)^{\natural}}(w_{\ast}Y)dY \nonumber\\
                                   &= \int_{U_{w_{\ast}}(\Q_p)} \mathbbm{1}_{U(q\mathbb{Z}_p)^{\top}}(D_q\overline{u}_YD_q^{-1})\mathbbm{1}_{t_0^{-1}T_0(1+q\Z_p)}(t(Y))dY.\label{eq:triv}
  \end{align}
  So far, our argument works for arbitrary Weyl elements. However, we will continue to explicitly compute the decomposition $w_{\ast}Y = u_Yt(Y)\overline{u}_Y$. This will crucially use the simple shape of $w_{\ast}$.

  Write $D_q=\diag(q^{n-1},\tilde{D}_q,1)$. In the coordinates
  \begin{equation}
    U_{w_{\ast}}(\Q_p) = \left\{Y=\left(
        \begin{matrix}
          1&y\mathbf{x}\tilde{D}_q&y\\ 0&I_{n-2}& \tilde{D}_q^{-1}\mathbf{z} \\0&0&1
        \end{matrix}
      \right)\colon \mathbf{x}^{\top},\mathbf{z}\in \Q_p^{n-2},\, y\in \Q_p\right\}, \nonumber
  \end{equation}
  the measure on $U_{w_{\ast}}(\Q_p)$ is given by $\vert y\vert^{n-2}_pd\mathbf{x}\, dy\, d\mathbf{z}$. Put
  \begin{equation}
    D_qw_{\ast}YD_q^{-1} = r(\mathbf{x},y,\mathbf{z}) = \left(
      \begin{matrix}
        0&0& -q^{n-1} \\ 0& I_{n-2} & \mathbf{z} \\ q^{1-n} & y\mathbf{x} & y 
      \end{matrix}
    \right).\nonumber
  \end{equation}
  On the support of the integral, we can write $r(\mathbf{x},y,\mathbf{z}) = ut\overline{u}$ for $u=u(\mathbf{x},y,\mathbf{z})\in U(\mathbb{Q}_p)$, $t=t(\mathbf{x},y,\mathbf{z})\in t_0^{-1}T_0(1+q\Z_p)$ and $\overline{u}=\overline{u}(\mathbf{x},y,\mathbf{z}) \in U(q\Z_p)^{\top}$. We arrive at the expression
  \begin{multline}
    \# X_{q,w_{\ast}}^{(p)}(c(\alpha,\beta;s)) = \mathcal{N}_{p^{v_p(q)}}\cdot  \int_{\Q_p^{n-2}}\int_{\Q_p}\int_{\Q_p^{n-2}} \mathbbm{1}_{U(q\mathbb{Z}_p)^{\top}}(\overline{u}(\mathbf{x},y,\mathbf{z}))\\ \cdot \mathbbm{1}_{t_0^{-1}T_0(1+q\Z_p)}(t(\mathbf{x},y\mathbf{z}))\vert y\vert^{n-2}_pd\mathbf{x}\, dy\, d\mathbf{z}.\nonumber
  \end{multline}

  To compute this integral, we will try to explicitly solve
  \begin{equation*}
    b\cdot r(\mathbf{x},y,\mathbf{z}) \in U(q\mathbb{Z}_p)^{\top}, \label{cond_w_ast_mat}
  \end{equation*}
  for $b=t(\mathbf{x},y,\mathbf{z})^{-1}u(\mathbf{x},y,\mathbf{z})^{-1}$. We can essentially write down $b$ explicitly as follows:
  \begin{multline}
    \underbrace{\left(
        \begin{matrix}
          y(1-\mathbf{x}\cdot \mathbf{z}) & -q^{n-1}y\mathbf{x} & q^{n-1} \\ 0& B & -\frac{1}{y}\cdot B\mathbf{z}\\ 0&0&\frac{1}{y}
        \end{matrix}
      \right)}_{=b}\underbrace{\left(
        \begin{matrix}
          0&0&-q^{n-1} \\ 0 & I_{n-2} & \mathbf{z} \\ q^{-(n-1)} & y\mathbf{x} & y 
        \end{matrix}
      \right)}_{=r(\mathbf{x},y,\mathbf{z})} \\ = \left(
      \begin{matrix}
        1 & 0 & 0\\ -\frac{1}{yq^{n-1}}\cdot B\mathbf{z} & B(I_{n-2}-\mathbf{z}\cdot \mathbf{x}) & 0 \\ \frac{1}{q^{n-1}y} & \mathbf{x} & 1  
      \end{matrix}
    \right).\nonumber
  \end{multline}
  Observe that
  \begin{equation}
    \det(I_{n-2}-\mathbf{z}\cdot \mathbf{x}) = 1-\mathbf{x}\cdot \mathbf{z} = \det(B)^{-1}.\nonumber
  \end{equation}
  In order for $B(I_{n-2}-\mathbf{z}\cdot \mathbf{x})$ to be in $U(\mathbb{Q}_p)^{\top}$ we choose
  \begin{equation}
    B = B_{\mathbf{x},\mathbf{z}} = \left( 
      \begin{matrix} 
        \frac{d_{n-2}}{d_{n-1}} & \frac{x_2z_1}{d_{n-1}}  & \cdots & \frac{x_{n-3}z_1}{d_{n-1}} & \frac{x_{n-2}z_1} {d_{n-1}} \\
        0 & \frac{d_{n-3}}{d_{n-2}} & \ddots & \ddots & \frac{x_{n-2}z_2}{d_{n-2}} \\
        \vdots & \ddots &\ddots & \ddots  & \vdots\\ 
        0 & \ddots  & \ddots  & \frac{d_2}{d_3} & \frac{x_{n-2}z_{n-3}}{d_3} \\ 
        0 &0  &\cdots & 0 & d_2^{-1} 
      \end{matrix}
    \right).\nonumber
  \end{equation}
  where
  \begin{equation}
    d_i = d_i(\mathbf{x},\mathbf{z}) = 1-\sum_{j=n-i}^{n-2}x_jz_j \text{ for }i=2,\ldots,n-1. \nonumber
  \end{equation}
  In particular, $d_{n-1}=1-\mathbf{x}\cdot \mathbf{z}.$

  We make the change of variables $\mathbf{z}\mapsto\mathbf{z}'$ where $z_{n-2}'=1-z_{n-2}x_{n-2}$ and $z_i'=z_{i+1}'-x_iz_i$. Thus we obtain
  \begin{multline}
    \# X_{q,w_{\ast}}^{(p)}(c(\alpha,\beta;s)) = \mathcal{N}_{p^{v_p(q)}}\cdot  \int_{\Q_p^{n-2}}\int_{\Q_p}\int_{\Q_p^{n-2}} \mathbbm{1}_{U(q\mathbb{Z}_p)^{\top}}(\overline{u}(\mathbf{x},y,\mathbf{z}'))\\ \cdot \mathbbm{1}_{t_0^{-1}T_0(1+q\Z_p)}(t(\mathbf{x},y,\mathbf{z}'))\cdot \vert y'\vert_p^{n-1} \frac{d\mathbf{x}\, dy\, d\mathbf{z}'}{\vert x_1\cdots x_{n-2}y\vert_p}.\nonumber
  \end{multline}
  The upshot is that $d_i = z_{n-i}'$.  One computes inductively that in the new coordinates we have
  \begin{equation}
    b\cdot r(\mathbf{x},y,\mathbf{z}') = \left(
      \begin{matrix}
        1 & 0 & 0 &\ldots & 0 & 0 \\
        \frac{1}{q^{n-1}yx_1}(1-\frac{z_2'}{z_1'}) & 1 & 0 & \ddots & 0 & 0 \\
        \frac{1}{q^{n-1}yx_2}(1-\frac{z_3'}{z_2'})& \frac{x_1}{x_2}(1-\frac{z_3'}{z_2'}) & 1 & \ddots & \ddots &  \vdots \\
        \vdots  & \ddots & \ddots & \ddots & \ddots & \vdots  \\
        \vdots & \ddots & \ddots & 1 & 0& 0 \\
        \frac{1}{q^{n-1}yx_{n-2}}(1-\frac{1}{z_{n-2}'}) & \frac{x_1}{x_{n-2}}(1-\frac{1}{z_{n-2}'}) & \ddots &\frac{x_{n-3}}{x_{n-2}}(1-\frac{1}{z_{n-2}'}) & 1 & 0\\
        \frac{1}{q^{n-1}y} & x_1 & \vdots& x_{n-3} & x_{n-2} & 1
      \end{matrix}
    \right).\nonumber
  \end{equation}
  Furthermore, since $b=t(\mathbf{x},y,\mathbf{z})^{-1}u(\mathbf{x},y,\mathbf{z})^{-1}$, we have
  \begin{equation}
    t(\mathbf{x},y,\mathbf{z})=\mathrm{diag}\Big(\frac{1}{yz_1'},\frac{z_1'}{z_2'},\ldots,z_{n-2}',y\Big).\nonumber
  \end{equation}
  In particular, since $t(\mathbf{x},y,\mathbf{z})\in t_0^{-1}\cdot T_0(1+q\mathbb{Z}_p)$ we find that
  \begin{equation}
    z_i' \in (sp^{\beta})^{i+1-n}(1+q\Z_p) \text{ and } y\in p^{-\alpha}(1+q\Z_p),\nonumber
  \end{equation}
  for $i=1,\ldots,n-2$, on the support of the integral. Thus we can execute the $\mathbf{z}'$- and $y$-integral to obtain
  \begin{multline}
    \# X_{q,w_{\ast}}^{(p)}(c(\alpha,\beta;s)) = \frac{\mathcal{N}_{p^{v_p(q)}}}{p^{v_p(q)}} \cdot p^{(n-1)\alpha+\frac{(n-1)(n-2)}{2}\beta} \\ \cdot\int_{(1+q\Z_p)^{n-2}}\int_{\Q_p^{n-2}} \mathbbm{1}_{U(q\mathbb{Z}_p)^{\top}}(\overline{u}'(\mathbf{x},\mathbf{u})) \frac{d\mathbf{x}\, d\mathbf{u}}{\vert x_1\cdots x_{n-2}\vert_p},\nonumber
  \end{multline}
  where we have introduced the new notation
  \begin{multline}
    \overline{u}'(\mathbf{x},\mathbf{u})=b\cdot r(q\mathbf{x},p^{-\alpha},\mathbf{z}') \\ = \left(
      \begin{matrix}
        1 & 0 & 0 &\ldots & 0 & 0 \\
        \frac{p^{\alpha}}{q^{n}x_1}(1-u_1sp^{\beta}) & 1 & 0 & \ddots & 0 & 0 \\
        \frac{p^{\alpha}}{q^{n}x_2}(1-u_2sp^{\beta})& \frac{x_1}{x_2}(1-u_2sp^{\beta}) & 1 & \ddots & \ddots &  \vdots \\
        \vdots  & \ddots & \ddots & \ddots & \ddots & \vdots  \\
        \vdots & \ddots & \ddots & 1 & 0& 0 \\
        \frac{p^{\alpha}}{q^{n}x_{n-2}}(1-u_{n-2}sp^{\beta}) & \frac{x_1}{x_{n-2}}(1-u_{n-2}sp^{\beta}) & \ddots &\frac{x_{n-3}}{x_{n-2}}(1-u_{n-2}sp^{\beta}) & 1 & 0\\
        p^{\alpha}q^{1-n} & qx_1 & \vdots& qx_{n-3} & qx_{n-2} & 1
      \end{matrix}
    \right).\nonumber
  \end{multline}
  after changing $\mathbf{x}$ to $q\mathbf{x}$ and for the obvious choice of $\mathbf{z}'$ given by $\mathbf{u}$ and $sp^{\beta}$. %It remains to be seen that the $\mathbf{x}$-integral is bounded by $p^{(2-n)v_p(q)}$.

  We take a look at the condition $\overline{u}'(\mathbf{x},\mathbf{u})\in U(q\Z_p)^{\top}$. First, by looking at the bottom row we see that $\alpha\geq n v_p(q)$ (proving the last statement of the lemma) and that $\mathbf{x}\in \Z_p^{n-2}$. The remaining support conditions are now rather simple.

  First, if $\beta>0$, then the first column simply gives $v_p(x_i) \leq \alpha - (n+1) v_p(n)$. Thus the $x_i$-integral is at most $\alpha$, while the $\textbf{u}$-integral is $p^{(n-2)v_p(q)}$, and we conclude
  % \begin{equation}
  %   x_i^{-1}\in q^{n+1}p^{-\alpha}\Z_p.\nonumber
  % \end{equation}
  % Forgetting all the other conditions allows us to estimate
  \begin{equation}
    \# X_{q,w_{\ast}}^{(p)}(c(\alpha,\beta;s)) \ll  \frac{\mathcal{N}_{p^{v_p(q)}}}{p^{(n-1)v_p(q)}} \cdot p^{(n-1+\epsilon)\alpha+\frac{(n-1)(n-2)}{2}\beta} \nonumber
  \end{equation}
  as desired.

  Finally, suppose $\beta=0$. In this case, we write $x_i=x_i'p^{m_i}$ with $m_i\geq 0$ for $i=1,\ldots, n-2$. We define
  \begin{equation}
    l_i = m_i- M_i + v_p(q), \quad M_i = \min(\alpha-nv_p(q),m_1,\ldots,m_{i-1}).\nonumber
  \end{equation}
  Then, each row gives the condition
  \begin{equation}
    u_i\in (s^{-1}+p^{l_i}\Z_p)\cap (1+q\Z_p).\nonumber
  \end{equation}
  We evaluate the $u_i$-integrals row by row, starting at the bottom. For each row (i.e., $i=1,\ldots,n-2$), we consider the cases $l_i\leq 0$ and $l_i>0$ separately. We start with $i=n-2$.  If $l_{n-2}\leq v_p(q)$, then the $u_{n-2}$-integral can be estimated by $q^{-1}$ and we obtain the bound
  % \begin{equation}
  $0\leq m_{n-2}\leq M_{n-2}$ so that the $x_{n-2}$-integral is at most $M_{n-2} + 1 \leq \alpha$.%\nonumber
  % \end{equation}

  On the other hand if $l_{n-2}>v_p(q)$, then the $u_{n-2}$-integral can be estimated by $q^{-l_{n-2}}$ and we have the condition $m_{n-2}>M_{n-2}\geq 0$. Then the combined $x_{n-2}, u_{n-2}$-integral is at most
  \begin{equation*}
    \sum_{m_{n-2} > M_{n-2}} p^{-(m_{n-2} - M_{n-2} + v_p(q))} \ll q^{-1}.
  \end{equation*}
  % One easily executes the $x_{n-2}$-integral (i.e. the $m_{n-2}$-sum) and the corresponding $u_{n-2}$-integrals to obtain a final contribution of at most
  % \begin{equation}
  %   \ll_{\epsilon} p^{\epsilon\alpha}q^{\epsilon-1}.\nonumber
  % \end{equation}
  Running the same argument for $i=n-3,n-4,\ldots,1$ completes the proof.
\end{proof}

\begin{rem}
  A slightly more careful analysis reveals that if $nv_p(q)\leq \alpha <(n+1)v_p(q)$, then $X_{q,w_{\ast}}^{(p)}(c(\alpha,\beta;s)^{\ast})=\emptyset$ unless $\beta=0$ and $s\in 1+p^{(n+1)v_p(q)-\alpha}\Z_p$. This precisely generalizes the support conditions from Lemma~\ref{deg3_prime_long} to $n\geq 4$ and $w_{\ast}$. We omit the details, because we will not be able to exploit this in our geometric estimate below.
\end{rem}

\begin{rem}\label{proof_triv}
  From \eqref{eq:triv}, which is valid for all Weyl elements $w$, we can easily produce a trivial bound for the ramified Kloosterman sets. Indeed we observe that $D_q^{-1}U(q\Z_p)D_q^{\top}\subseteq U(\Q_p)^{\top}$. Thus weakening the conditions yields
  \begin{align}
    \# X_{q,w}^{(p)}(c) &\leq \mathcal{N}_{p^{v_p(q)}}\cdot \int_{U_w(\Q_p)} \mathbbm{1}_{U(\mathbb{Z}_p)^{\top}}(\overline{u}_Y)\cdot \mathbbm{1}_{(c^{\ast})^{-1}T_0(\Z_p)}(t(Y))dY \nonumber\\
                        &= \mathcal{N}_{p^{v_p(q)}}\cdot \int_{U_w(\Q_p)} \mathbbm{1}_{U(\mathbb{Q}_p)c^{\ast}K_p}(wY)dY. \nonumber
  \end{align}
  From here one can proceed as in \cite{DR}. More precisely, one uses Mellin inversion to produce an average over $t_0$. This average can be computed explicitly when interpreted as the intertwining operator associated to $w$ applied to the spherical element in an induced representation. This leads to \eqref{eq:triv_gen_ks}.
\end{rem}

We now turn our attention to Kloosterman sums. These are defined by summing certain (additive) characters over the Kloosterman sets. For these sums to be well defined, we impose the following restriction on the characters: Let $N, M \in \mathbb{Z}^{n-1}$, $c \in \mathbb{N}^{n-1}$, $w \in W$, $v \in V$.  Then, provided that
\begin{equation}\label{compat}
  \theta_M\big(c^{\ast} w x w^{-1} (c^{\ast})^{-1}\big) = \theta_N^v(x)
\end{equation}
for all $x \in w^{-1}U(\mathbb{Q}) w \cap U(\mathbb{Q})$, the Kloosterman sum
\begin{equation*}%\label{klo}
  S^v_{q, w}(M, N, c) = \sum_{x  c^{\ast} w y \in X_{q,w}(c) } \theta_M(x )\theta_N^v(y)
\end{equation*}
is well-defined, see \cite[Proposition 1.3]{Fr}. If \eqref{compat} is not met, we define
\begin{equation*}
  S^v_{q, w}(M, N, c) = 0.
\end{equation*}
Note that \eqref{compat} is equivalent to \eqref{compat_coord}. Thus $S^v_{q, w}(M, N, c) = 0$ unless $(w,c)$ is $(\theta_M,\theta_N^v)$-relevant.

The Kloosterman sum for the trivial Weyl element is easily computed:
\begin{equation}
  S^v_{q,\mathrm{id}}(M,N,c) = \delta_{\substack{N=M\\ c=(1,\ldots,1)}}\cdot \mathcal{N}_q.\nonumber
\end{equation}
We record the following relevant properties for the moduli of the element $w_{\ast}$:
\begin{lemmy}\label{lm:relevance}
  Let $n\geq 4$, $N = (\ast, 1, \ldots, 1, \ast)$, $M = (\ast, 1,\ldots, 1, \ast) \in \mathbb{N}^{n-1}$, $v\in V$ and $q \in \mathbb{N}$. Then $S^v_{q, w_{\ast}}(M, N, c)$ vanishes unless
  \begin{equation*}
    c = (q^nr, q^n rs, q^nrs^2, \ldots,  q^nrs^{n-2}) \quad \text{or} \quad  (q^nrs^{n-2}, q^n rs^{n-1},  \ldots,  q^nrs, q^nr)
  \end{equation*}
  for some $r, s \in \mathbb{N}$.
\end{lemmy}
\begin{proof}
  This is part of \cite[Theorem~4.4]{AB}.
\end{proof}

As hinted at above, the factorization properties of Kloosterman sums for $\Gamma(q)^{\natural}$ with $q>1$ are complicated in general. We require the following result given in \cite[Lemma~4.1]{AB}. For $c = ( r, rs, rs^2, \ldots, rs^{n-2}) $ we have
\begin{equation}
  \vert S_{q,w_{\ast}}^v(N,N,c)\vert \leq \vert S_{1,w_{\ast}}^v(N,N',d)\vert\cdot \prod_{p\mid q}X_{q,w_{\ast}}^{(p)}(\overline{d}_p t_p), \label{factori}
\end{equation}
where
\begin{equation*}
  d= ( r',r's',\ldots,r'(s')^{n-2}), \overline{d}_p = (1,\overline{s}_p,\ldots,\overline{s}_p^{n-2}) \text{ and }t_p = (u_p,u_pv_p,\ldots,u_pv_p^{n-2})
\end{equation*}
with $r'=r/(r,q^{\infty})$, $s'=s/(s,q^{\infty})$, $u_p=(r,p^{\infty})$, $v_p=(s,p^{\infty})$ and a representative $\overline{s}_p$ of $s/v_p$ modulo $p^{v_p(q)}$. The new parameter $N'\in \mathbb{Z}^{n-1}$ can in principle computed explicitly, but we do not need this. It suffices to know that if the entries of $N$ are coprime to the entries of $d$, then also the entries of $N'$ are coprime to $d$. %Furthermore, $N'\in \mathbb{N}^{n-1}$ is some new essentially irrelevant parameter. We only note that up to the possible dependence on $m$ everything in $N$ and $N'$ is coprime to the entries of $c''$.

The trivial bound for Kloosterman sums always is $\vert S^v_{q,w_1}(M,N,c)\vert \leq \# X_{q,w}(c)$. However, in general, it is reasonable to expect that the oscillation of the characters $\theta_M(x)$ and $\theta_N(x)$ leads to some cancellation. Even for $q=1$, these non-trivial bounds are hard to come by in higher rank, but there are some results available. Let us record some important results:
\begin{itemize}
\item Let $n=3$, $c=(c_1,c_2)\in \mathbb{N}^2$ with $(c_1c_2,N_1N_2M_1M_2)=1$. We have the classical bound \cite[(5.9)]{Ste}
  % \begin{equation}
  $S_{1,w_l}^v(N,M,c) \ll (c_1c_2)^{1/2+\epsilon} (c_1, c_2).$ %\label{long_nontriv_3}
  % \end{equation}
  Stronger bounds are obtained in \cite{DF}, but we do not need them.  This implies quite easily
  \begin{equation}
    \sum_{\substack{c_1\leq C_1, c_2\leq C_2\\ rc_1\equiv sc_2 \, (\text{mod }t)\\ (c_1c_2, t) = 1}} \frac{\vert S^v_{1,w_l}(N,M,(c_1,c_2))\vert}{c_1c_2} \ll (tC_1C_2)^{\epsilon}\Big(\frac{(r,s,t)}{t} (C_1C_2)^{1/2} +\min(C_1,  C_2)^{1/2}\Big)\label{Weil_on_av}
  \end{equation}
  for $C_1, C_2 \geq 1$, $r, s, t \in \mathbb{Z}$, $t > 0$, $\epsilon > 0$.  Note that we can allow $N, M$ to vary with $c_1, c_2$ as long as they are coprime to $c_1c_2$. 
  
  % As a special case of \cite[Theorem~C]{Fr} we also have
  % \begin{equation}
  %   S_{1,w_1'}^v(N,M,c) \ll (c_1c_2)^{\frac{2}{3}+\epsilon}.\label{vor_nontriv_3}
  % \end{equation}
\item For general $n\geq 4$ we have
  \begin{equation}
    S_{1,w_{\ast}}^v(N,M,c) \ll  (c_1\cdots c_{n-1})^{1-\frac{1}{4n}+\epsilon}\label{wast_nontriv}
  \end{equation}
  as long as the entries of $N$ and $M$ are coprime to the entries of $c$. (Here we do not display the dependence on $q$ on the left hand side.) This is \cite[Corollary~2]{BM} (where the explicit saving can be found at the end of \cite[Section~5.4]{BM}).
\end{itemize}

\section{The geometric estimate}\label{sec_g}

In general, sums of the form
\begin{equation}
  \mathcal{S}_{w,N}^{(n)}(X) = \sum_{v\in V}\sum_{  c_1,\ldots,c_{n-1}\leq X}\frac{\vert S_{q,w}^v(N,N,c)\vert}{c_1\cdots c_{n-1}}\label{eq:def_S}
\end{equation}
with $w\in W$ make an appearance on the geometric side of the Kuznetsov formula.  In this section, we derive a crucial estimate for the special Weyl element $w_{\ast}$.  As a warm up, we start with the case $n=3$.

\begin{lemmy}\label{gl_3_geo}
  Let $q,m\in \mathbb{N}$ with $(q,m)=1$ and put $N=(m,1)$. We have
  \begin{equation}
    \mathcal{S}_{w_l,N}^{(3)}(X) \ll (mqX)^{\epsilon}\mathcal{N}_q\Big(1+\frac{X}{q^{6}}\Big). \nonumber
  \end{equation}
  Furthermore, for $w\in \{w_1,w_1'\}$ we have
  \begin{equation}
    \mathcal{S}_{w,N}^{(3)}(X) \ll  q^{\epsilon}\frac{\mathcal{N}_q}{q^2} \frac{X^{1/2}}{q^2}.\nonumber
  \end{equation}
\end{lemmy}
\begin{proof}
  We start by treating the case $w=w_1'$. The case $w=w_1$ is similar and we omit the details. We apply Lemma~\ref{lm:gl3_set_est_vor} and use trivial bounds along with the Chinese Remainder Theorem  to obtain
  \begin{equation*}
    \mathcal{S}_{w_l,N}^{(3)}(X) \ll q^{\epsilon} \sum_{\substack{q^2 \mid c_1 \leq X\\ q^4 \mid c_2 \leq X\\ c_2 = c_1^2/m}} \frac{\mathcal{N}_q}{q^2} \ll  \frac{\mathcal{N}_q}{q^{2-\epsilon}} \sum_{\substack{\gamma \leq \sqrt{mX}/q^2\\ m\mid \gamma^2}}1 \ll
    (m q)^{\epsilon}\frac{\mathcal{N}_q}{q^{2}}\frac{X^{1/2}}{q^2}. 
  \end{equation*}

  We turn towards $w=w_l=w_{\ast}$. Note that since here we are dealing with the long Weyl element, all moduli are relevant. If $d=(d_1,d_2)$ and $c=(c_1,c_2)$, then we let $\delta_d(c)$ be the function that is $1$ if $X_{q,w_l}^{(p)}((d_1c_1,d_2c_2))\neq\emptyset$ and $\delta_d(c)=0$ otherwise. Then, using Lemma~\ref{deg3_prime_long}, we can estimate
  \begin{equation}
    \frac{\vert S^v_{q,w_l}(N,N,(d_1c_1,d_2c_2))\vert}{d_1c_1d_2c_2} \ll \frac{\mathcal{N}_q}{q^{2-\epsilon}}\delta_d(c)\cdot \frac{\vert S^v_{1,w_l}(N,N',(c_1,c_2))\vert}{c_1c_2} \nonumber
  \end{equation}
  whenever $d_1,d_2\mid (mq)^{\infty}$ and $(c_1c_2,mq)=1$, for some suitable $N'$ depending possibly on $c_1, c_2, d_1, d_2$ with entries coprime to $c_1c_2$. Thus we can write our sum as
  \begin{equation}
    \mathcal{S}_{w_l,N}^{(3)}(X) \ll \max_{v \in V}\frac{\mathcal{N}_q}{q^{2-\epsilon}}\sum_{\substack{d_1,d_2\leq X\\ d_1, d_2\mid (mq)^{\infty}}} \sum_{\substack{c_1\leq X/d_1\\ c_2\leq X/d_2\\ (c_1c_2,mq)=1}} \delta_d(c) \frac{\vert S^v_{1,w_l}(N,N',(c_1,c_2))\vert}{c_1c_2}.\nonumber
  \end{equation}
  To unpack the condition $\delta_d(c)=1$, we write
  \begin{equation}
    d_{i} = d_{i,m}\cdot \prod_{p\mid q}d_{i,p} \nonumber
  \end{equation}
  for $i = 1, 2$, where $d_{i,m}\mid m^{\infty}$ and $d_{i,p}\mid p^{\infty}$. For a subset $\mathcal{S}\subseteq \{p\mid q\}$ we define $q_{\mathcal{S}} = \prod_{p\in \mathcal{S}}p^{v_p(q)}$ and $d_i^{\mathcal{S}} = \prod_{p\not\in \mathcal{S}}d_{i,p}$. We estimate
  \begin{equation}
    \mathcal{S}_{w_l,N}^{(3)}(X) \ll \max_{v \in V}\frac{\mathcal{N}_q}{q^{2-\epsilon}}\sum_{\mathcal{S}\subseteq \{p\mid q\}}\sum_{\substack{d_1,d_2\leq X\\ d_1,d_2\mid (mq)^{\infty}\\ 3v_p(q)\leq v_p(d_{1,p})<4v_p(q)\, \forall p\in \mathcal{S}\\ 4v_p(q)\leq v_p(d_{1,p})\forall p\not\in \mathcal{S}}} \sum_{\substack{c_1\leq X/d_1\\ c_2\leq X/d_2\\ (c_1c_2,mq)=1}} \delta_d(c) \frac{\vert S^v_{1,w_l}(N,N',(c_1,c_2))\vert}{c_1c_2}.\nonumber
  \end{equation}
  Lemma~\ref{deg3_prime_long} imposes the conditions
  \begin{equation*}
    d_{1, p} = d_{2, p}  \quad (p \in \mathcal{S}) \quad \text{and} \quad d_1c_2\equiv d_2c_2 \pmod{q_{\mathcal{S}}^4}.
  \end{equation*}
  % Invoking the support conditions from Lemma~\ref{deg3_prime_long} lets us write this sum as
  % \begin{multline}
  %   \mathcal{S}_{w_l,N}^{(3)}(X) \ll \frac{\mathcal{N}_q}{q^{2-\epsilon}}\sum_{\mathcal{S}\subseteq \{p\mid q\}}\sum_{\substack{d_1,d_2\leq X,\\ d_1,d_2\mid (mq)^{\infty},\\ 3v_p(q)\leq v_p(d_{1,p})<4v_p(q)\, \forall p\in \mathcal{S},\\ v_p(d_{1,q}),v_p(d_{2,p})\geq 4v_p(q)\forall p\not\in \mathcal{S}}} \prod_{p\in \mathcal{S}}\delta_{d_{1,p}=d_{2,p}} \\ \cdot \sum_{\substack{c_1\leq X/d_1,\\ c_2\leq X/d_2,\\ d_1c_2\equiv d_2c_2\mathrm{ mod }q_{\mathcal{S}}^4}} \frac{\vert S^v_{1,w_l}(N,N_d',(c_1,c_2))\vert}{c_1c_2}. \nonumber
  % \end{multline}
  Using the estimate \eqref{Weil_on_av} under these circumstances yields
  \begin{multline}
    \mathcal{S}_{w_l,N}^{(3)}(X) \ll \frac{\mathcal{N}_q}{q^{2-\epsilon}}X^{\epsilon}\sum_{\mathcal{S}\subseteq \{p\mid q\}}\sum_{\substack{d_1,d_2\leq X\\ d_1,d_2\mid (mq)^{\infty}\\ 3v_p(q)\leq v_p(d_{1,p})=v_p(d_{2,p})<4v_p(q)\, \forall p\in \mathcal{S}\\ v_p(d_{1,q}),v_p(d_{2,p})\geq 4v_p(q)\forall p\not\in \mathcal{S}}} \left[\frac{X}{q_{\mathcal{S}}^4(d_1^{\mathcal{S}}d_2^{\mathcal{S}})^{\frac{1}{2}}}+\frac{X^{\frac{1}{2}}}{\max(d_1,d_2)^{\frac{1}{2}}}\right]. \nonumber
  \end{multline}
 To estimate the second term, we simply use that  $q^3 \mid d_1, d_2$ to obtain
 $$\frac{\mathcal{N}_q}{q^{2-\epsilon}}X^{\epsilon}\sum_{\mathcal{S}\subseteq \{p\mid q\}}\sum_{\substack{d_1,d_2\leq X\\ q^3 \mid d_1,d_2\mid (mq)^{\infty} }}  \frac{X^{\frac{1}{2}}}{\max(d_1,d_2)^{\frac{1}{2}}} \ll \frac{\mathcal{N}_q}{q^{2 }} X^{1/2} q^{\epsilon-\frac{3}{2}}(mX)^{\epsilon}. $$
  %  We can trivially estimate (simply using that $q^3 \mid d_1, d_2$)
%  \begin{equation}
%    \sum_{\mathcal{S}\subseteq \{p\mid q\}}\sum_{\substack{d_1,d_2\leq X\\ d_1,d_2\mid (mq)^{\infty}\\ 3v_p(q)\leq v_p(d_{1,p})=v_p(d_{2,p})<4v_p(q)\, \forall p\in \mathcal{S}\\ v_p(d_{1,q}),v_p(d_{2,p})\geq 4v_p(q)\forall p\not\in \mathcal{S}}} \max(d_1,d_2)^{-\frac{1}{2}} \ll q^{\epsilon-\frac{3}{2}}(mX)^{\epsilon}.\nonumber
%  \end{equation}
For the contribution of the first term we use that $(q/q_\mathcal{S})^4 \mid d_1, d_2$ and 
$$\frac{\mathcal{N}_q}{q^{2-\epsilon}}X^{\epsilon}\sum_{\mathcal{S}\subseteq \{p\mid q\}}\sum_{\substack{d_1,d_2\leq X\\ (q/q_{\mathcal{S}})^4 \mid d_1,d_2\mid (mq)^{\infty} }} \frac{X}{q_{\mathcal{S}}^4(d_1^{\mathcal{S}}d_2^{\mathcal{S}})^{\frac{1}{2}}} \ll \frac{\mathcal{N}_q}{q^{2}}X q^{\epsilon-4}(mX)^{\epsilon}.$$
% The other contribution is also easily estimated by
 % \begin{equation}
  %  \sum_{\mathcal{S}\subseteq \{p\mid q\}}\sum_{\substack{d_1,d_2\leq X\\ d_1,d_2\mid (mq)^{\infty}\\ 3v_p(q)\leq v_p(d_{1,p})=v_p(d_{2,p})<4v_p(q)\, \forall p\in \mathcal{S}\\ v_p(d_{1,q}),v_p(d_{2,p})\geq 4v_p(q)\forall p\not\in \mathcal{S}}} \frac{1}{q_{\mathcal{S}}^4(d_1^{\mathcal{S}}d_2^{\mathcal{S}})^{\frac{1}{2}}} \ll q^{\epsilon-4}(mX)^{\epsilon}.\nonumber
%  \end{equation}
  Combining these bounds completes the proof.
  % Thus we have
  % \begin{equation}
  %   \mathcal{S}_{w_l,N}^{(3)}(X) \ll \frac{\mathcal{N}_q}{q^{2-\epsilon}}(mX)^{\epsilon} \left[\frac{X}{q^4}+\frac{X^{\frac{1}{2}}}{q^{\frac{3}{2}}}\right]. \nonumber
  % \end{equation}
\end{proof}

We turn towards the higher rank cases. Here we only treat the case $w=w_{\ast}$ and we estimate everything trivially. In contrast to the case $n=3$, we exploit that admissibility of $(w_{\ast},c)$ gives strong constraints on the shape of the moduli $c_1,\ldots,c_{n-1}$.

\begin{prop} \label{pro_geo}
  Let $n\geq 4$ and $q\in \mathbb{N}$ be arbitrary and $N=(m,1,\ldots,1)$ for $m\in \mathbb{N}$. Then we have
  \begin{equation}
    \sum_{v\in V}\sum_{c_1, \ldots, c_{n-1} \leq X} 
    \frac{\vert S_{q,w_{\ast}}^v(N,N,c)\vert }{c_1\cdots c_{n-1}} \ll  (qmX)^{\epsilon} \mathcal{N}_q\cdot \frac{X^{\frac{3}{4}+\frac{1}{4n}}}{q^{\frac{7n-3}{4}}}\nonumber 
  \end{equation}
  for all $X>0$ and $\epsilon > 0$. 
\end{prop}
\begin{proof}
  We start by applying Lemma~\ref{lm:relevance} to see that $S_{q,w_{\ast}}^{v}(N,N,c)$ vanishes unless
  \begin{equation}
    c= (r, rs,\ldots,rs^{n-2}) \text{ or }c=(rs^{n-2},\ldots,rs,r), \nonumber
  \end{equation}
  with $s,r\in \mathbb{N}$ and $q^n\mid r$. We will focus on moduli of the first shape. The argument for those of the second one is similar and we omit the details. Next we factor the Kloosterman sums using \eqref{factori}. This yields
  \begin{multline}
    \sum_{\substack{c\in \mathbb{N}^{n-1},\\ c_j\leq X}} \frac{\vert S_{q,w_{\ast}}^v(N,N,c)\vert }{c_1\cdots c_{n-1}}  \ll   \sum_{\substack{r_1s_1^{n-2}\leq X\\ r_1,s_1\mid (qm)^{\infty}}} \prod_{p\mid q} \frac{\sup_{\overline{s}\in \mathbb{Z}_p^{\times}} \# X_{q,w_{\ast}}^{(p)}(d_p(\overline{s},s_1,r_1))}{(r_1,p^{\infty})^{n-1}(s_1,p^{\infty})^{(n-1)(n-2)/2} }\\ \cdot \sum_{\substack{r_2s_2^{n-2}\leq X/(r_1s_1^{n-2})\\ (r_2s_2,qm)=1}}\frac{\vert S_{1,w_{\ast}}^v(N,N',(r_2,r_2s_2,\ldots,r_2s_2^{n-2}))\vert }{r_2^{n-1}s_2^{ (n-1)(n-2)/2}},\nonumber
  \end{multline}
  for $d_p(\overline{s},s_1,r_1)=(p^{v_p(r_1)},\overline{s}p^{v_p(r_1)+v_p(s_1)},\cdots,\overline{s}^{n-2}p^{v_p(r_1)+(n-2)v_p(s_1)})$ and some suitable $N'$ (depending on the various moduli). 

  The unramified sum can be estimated using \eqref{wast_nontriv}.  We obtain
  \begin{multline}
    \sum_{\substack{r_2s_2^{n-2}\leq  X/(r_1s_1^{n-2})\\ (r_2s_2,qm)=1}}\frac{\vert S_{q,w_{\ast}}^v(N,N',(r_2,r_2s_2,\ldots,r_2s_2^{n-2}))\vert }{r_2^{n-1}s_2^{(n-1)(n-2)/2}} \\ \ll  \sum_{\substack{r_2s_2^{n-2}\leq  X/(r_1s_1^{n-2})\\ (r_2s_2,qm)=1}} r_2^{-\frac{(n-1)}{4n}}s_2^{-\frac{(n-1)(n-2)}{8n}}\ll X^{\epsilon}\Big(\frac{X}{r_1s_1^{n-2}}\Big)^{\frac{3}{4}+\frac{1}{4n}}.\nonumber
  \end{multline}
  So far, we have seen that
  \begin{equation}
    \sum_{c_1, \ldots, c_{n-1} \leq X} \frac{\vert S_{q,w_{\ast}}^v(N,N,c)\vert }{c_1\cdots c_{n-1}}\ll X^{\frac{3}{4}+\frac{1}{4n}+\epsilon} \sum_{\substack{r_1s_1^{n-2}\leq X\\ r_1,s_1\mid (mq)^{\infty}}} \prod_{p\mid q} \frac{\sup_{\overline{s}\in \mathbb{Z}_p^{\times}} \# X_{q,w_{\ast}}^{(p)}(d_p(\overline{s},s_1,r_1))}{(r_1,p^{\infty})^{\alpha}(s_1,p^{\infty})^{\beta}},\nonumber
  \end{equation}
  for $\alpha=n-1+\frac{3}{4}+\frac{1}{4n}$ and $\beta=\frac{(n-1)(n-2)}{2}+(n-2)(\frac{3}{4}+\frac{1}{4n}).$ At this point, we insert the results from Lemma~\ref{lem:kszeta} to obtain
  \begin{multline}
    \sum_{c_1, \ldots, c_{n-1} \leq X} \frac{\vert S_{q,w_{\ast}}^v(N,N,c)\vert }{c_1\cdots c_{n-1}} \ll (qmX)^{\epsilon}X^{\frac{3}{4}+\frac{1}{4n}}\frac{\mathcal{N}_q}{q^{n-1}}\sum_{\substack{r_1\leq X\\ q^n\mid r_1\mid (mq)^{\infty}}} (r_1,q^{\infty})^{-\frac{3}{4}-\frac{1}{4n}}  \\
    \cdot \sum_{\substack{s_1^{n-2}\leq X\\ s_1\mid (mq)^{\infty}}} (s_1,q^{\infty})^{-(n-2)(\frac{3}{4}+\frac{1}{4n})}.\nonumber
  \end{multline}
  Note that the $s_1$-sum is absolutely convergent and can be estimated trivially.  On the other hand, we have
  \begin{equation}
    \sum_{\substack{r_1\leq X\\ q^n\mid r_1\mid (mq)^{\infty}}} (r_1,q^{\infty})^{-\frac{3}{4}-\frac{1}{4n}} \ll  (mqX)^{\epsilon}q^{-n(\frac{3}{4}+\frac{1}{4n})},\nonumber
  \end{equation}
  by a standard application of Rankin's trick. The result follows directly.
\end{proof}

\section{The Jacquet--Langlands transfer}\label{sec:cnfm30d6ii}

In this section, we will establish the main ingredient for the proof of Theorem~\ref{th_cocomp}. Recall that $B$ is a division algebra over $\Q$ of degree $n\geq 3$ which splits over $\R$. Furthermore, we have an embedding $\phi\colon B\to \mathrm{Mat}_{n}(\R)$ such that $\phi(B)$ is a $\Q$-form of $\mathrm{Mat}_{n}(\R)$. Let $\mathcal{O}\subset B$ be an order and let $q$ be coprime to $\mathrm{discr}(\mathcal{O})$. Then we are considering the congruence lattices
\begin{equation*}
  \Gamma_{\mathcal{O}}(q) = \phi(\mathcal{O}^1)\cap \Gamma(q).
\end{equation*}
In this setting, we will prove the following:

\begin{prop}\label{main_ing_coc}
  There is an integer $D_{\mathcal{O}}\mid \mathrm{discr}(\mathcal{O})^{\infty}$ that only depends on $\mathcal{O}\subseteq B$ such that 
  \begin{equation}
    N_v(\sigma,\mathcal{F}_{\Gamma_{\mathcal{O}}(q)}(M)) \leq N_v(\sigma,\mathcal{F}_{\Gamma(qD_{\mathcal{O}}),\mathrm{disc}}(M))\nonumber
  \end{equation}
  for any place $v$ where $\Gamma_{\mathcal{O}}(q)$ is unramified.
\end{prop}
\begin{proof}
  Proving this will amount to an application of the Jacquet--Langlands correspondence.  It will be convenient to work adelically using the language of automorphic representations.  Recall that a cuspidal automorphic representation $\pi$ of $\mathrm{GL}_n(\A)$ has a factorization $\pi=\otimes_v'\pi_v$. We write $\pi^{K(q)}$ for the space of $K(q)$-invariant vectors in $\pi$. At each place $v$ where $\pi_v$ is unramified, in particular at all the places away from $q$, we associate a Langlands parameter $\mu_{\pi_v}$. Following the discussion in \cite[Section~2.6.3 and Remark~6.1]{AB}, we write
  \begin{equation}
    N_v(\sigma,\mathcal{F}_{\Gamma(q'),\mathrm{cusp}}(M)) = \frac{1}{\varphi(q')}\sum_{\substack{\pi \text{ cusp. aut. rep of }\mathrm{GL}_n(\A)\\ \Vert \mu_{\pi_{\infty}}\Vert\leq M,  \sigma_{\pi_v}\geq \sigma}} \dim_{\mathbb{C}}\pi^{K(q')}, \nonumber
  \end{equation}
  for $q' \in \mathbb{N}$.  Note that we have used the fact that $\dim_{\mathbb{C}}\pi^{K(q')}\neq 0$ if and only if $\pi\mid L^2_{\mathrm{cusp}}(X(q'))$.

  Similarly, we can adelize $B(\A) = \otimes_v' B_v$, where $B_v=B\otimes \Q_v$.  We use $\phi\colon B\to \mathrm{Mat}_{n}(\R)$ to identify $B_{\infty}$ with $\mathrm{Mat}_{n}(\R)$. Set $K_{\mathcal{O},\infty}(q) = \phi^{-1}(\mathrm{SO}_n(\R))$ and $K_{\mathcal{O},p}(q) = \phi^{-1}(\Gamma_{\mathcal{O}}(q))\otimes \Z_p$. In this way, we obtain a compact subgroup $K_{\mathcal{O}}(q)= \otimes'_v K_{\mathcal{O},v}(q)\subseteq B(\A)^{\times}$. In this setting, the usual adelization procedure yields
  \begin{equation}
    N_v(\sigma,\mathcal{F}_{\Gamma_{\mathcal{O}}(q)}(M)) \ll_{\mathcal{O}} \frac{1}{\varphi(q)}\sum_{\substack{\sigma \text{ aut. rep of }B^{\times}(\A)\\ \Vert \mu_{\sigma_{\infty}}\Vert\leq M, \sigma_{\sigma_v}\geq \sigma}} \dim_{\mathbb{C}}\sigma^{K_{\mathcal{O}}(q)}. \nonumber
  \end{equation}

  According to \cite[Theorem~A.7]{Ba}, the only residual automorphic representations of $B^{\times}(\A)$ are the one dimensional ones (i.e.\ characters). On the other hand, the global Jacquet--Langlands correspondence gives us an injective map $\sigma \mapsto \mathrm{JL}(\sigma)$ from cuspidal automorphic representations of $B^{\times}(\A)$ to cuspidal automorphic representations of $\mathrm{GL}_n(\A)$. See for example \cite[Theorem~5.1]{Ba} or \cite[Theorem~2.5]{Fl} for a precise statement. In particular, by (strong) multiplicity one we directly obtain the inequalities
  \begin{align}
    N_v(\sigma,\mathcal{F}_{\Gamma_{\mathcal{O}}(q)}(M)) &\ll_{\mathcal{O}}1+ \frac{1}{\varphi(q)}\sum_{\substack{\sigma \text{ cusp. aut. rep of }B^{\times}(\A)\\ \Vert \mu_{\sigma_{\infty}}\Vert\leq M, \,  \sigma_{\sigma_v}\geq \sigma}} \dim_{\mathbb{C}}\sigma^{K_{\mathcal{O}}(q)} \nonumber\\
                                               &= 1+\frac{1}{\varphi(q)}\sum_{\substack{\pi \text{ cusp. aut. rep of }\mathrm{GL}_n(\A)\\ \Vert \mu_{\pi_{\infty}}\Vert\leq M, \,  \sigma_{\pi_v}\geq \sigma\\ \pi=\mathrm{JL}(\sigma)}} \dim_{\mathbb{C}}\sigma^{K_{\mathcal{O}}(q)} .\nonumber
  \end{align} 
  Note that, in order to relate the Langlands parameters at $v$ and $\infty$ we have used that the global Jacquet Langlands correspondence is compatible with the local one at $\infty$ and $v$. We will discuss this in more detail below. 

  To complete the proof it remains to choose $q'=qD_{\mathcal{O}}$ such that
  \begin{equation*}
    \dim_{\mathbb{C}}\sigma^{K_{\mathcal{O}}(q)} \ll_{\mathcal{O}} \dim_{\mathbb{C}}\mathrm{JL}(\sigma)^{K(q')} \label{dim_comp}
  \end{equation*}
  for all cuspidal automorphic representations $\sigma$ of $B^{\times}(\A)$. Indeed in this case we simply have
  \begin{align}
    N_v(\sigma,\mathcal{F}_{\Gamma_{\mathcal{O}}(q)}(M)) &\ll_{\mathcal{O}}1+\frac{1}{\varphi(q')}\sum_{\substack{\pi \text{ cusp. aut. rep of }\mathrm{GL}_n(\A)\\ \Vert \mu_{\pi_{\infty}}\Vert\leq M, \, \sigma_{\pi_v}\geq \sigma\\ \pi=\mathrm{JL}(\sigma)}} \dim_{\mathbb{C}}\pi^{K(q')} \nonumber\\
                                               &\leq 1 +\frac{1}{\varphi(q')}\sum_{\substack{\pi \text{ cusp. aut. rep of }\mathrm{GL}_n(\A)\\ \Vert \mu_{\pi_{\infty}}\Vert\leq M, \, \sigma_{\pi_v}\geq \sigma}} \dim_{\mathbb{C}}\pi^{K(q')} \nonumber\\
                                               &\leq 1+ N_v(\sigma,\mathcal{F}_{\Gamma(q'),\mathrm{cusp}}(M)) \leq N_v(\sigma,\mathcal{F}_{\Gamma(q'),\mathrm{disc}}(M)).\nonumber
  \end{align}

  Choosing $q'$ is a purely local problem. To see this, we recall that the global Jacquet--Langlands correspondence is compatible with local one. More precisely, if $\pi=\otimes_v' \pi_v$ and $\sigma=\otimes_v'\sigma_v$ and $\pi=\mathrm{JL}(\sigma)$, then $\pi_v\cong \sigma_v$ for all places $v\not\in S$ and $\pi_v=\mathrm{JL}_v(\sigma_v)$ for all places $v\in S$. Here, $S$ is a finite set of places, which we can take to divide $\mathrm{discr}(\mathcal{O})$, while $\mathrm{JL}_v$ denotes the local Jacquet--Langlands correspondence. It is clear that \eqref{dim_comp} would follow from the corresponding local estimates
  \begin{equation*}
    \dim_{\mathbb{C}}\sigma_v^{K_{\mathcal{O},v}(q)} \ll_{\mathcal{O}} \dim_{\mathbb{C}}\pi_v^{K_v(q')}
  \end{equation*}
  where $\pi=\otimes_v'\pi_v = \mathrm{JL}(\sigma)$. If $v\not\in S$, then $\pi_v\cong \sigma_v$ and $K_v(qD_{\mathcal{O}}) = K_v(q) \cong K_{\mathcal{O},v}(q)$. Thus we have
  \begin{equation}
    \dim_{\mathbb{C}}\sigma_v^{K_{\mathcal{O},v}(q)} = \dim_{\mathbb{C}}\pi_v^{K_v(q')} \text{ for all }v\not\in S.\nonumber
  \end{equation}
  Finally, we turn to the place $v\in S$ (i.e. $v=p$, $(p,q)=1$ and $p\mid\mathrm{discr}(\mathcal{O})$). Note that by Bernstein's uniform admissibility \cite{Be}, we have
  \begin{equation}
    \dim_{\mathbb{C}}\sigma_v^{K_{\mathcal{O},v}(q)} \ll_{\mathcal{O}} 1.\nonumber 
  \end{equation}
  On the other hand, up to unramified twist, there are only finitely many irreducible admissible unitary representations $\sigma_v$ of $B^{\times}_v$ such that $\sigma^{K_{\mathcal{O},v}(q)}\neq \{0\}$. This follows easily from \cite[Lemma~1]{BH}, for example.  Let $\{\sigma_{v,1}\ldots,\sigma_{v,L}\}$ be the set of all these representations (up to unramified twist) and note that $L=L(\mathcal{O})$ is independent of $q$. Therefore it suffices to find $d_{\mathcal{O},v}$ such that 
  \begin{equation}
    \mathrm{JL}_v(\sigma_{v,i})^{K_v(q')} \neq \{0\} \nonumber
  \end{equation}
  for all $i=1,\ldots,L$. Here we recall that $q'=qD_{\mathcal{O}}$, and we put $D_{\mathcal{O}} = \prod_{p\mid \mathrm{disc}(\mathcal{O})}p^{d_{\mathcal{O},p}}$. The exponents $d_{\mathcal{O},v}$ are easily constructed. To be concrete, we can argue as follows. Let $t=\max_i\rho(\sigma_{v,i})$ be the maximum depth of the representations $\sigma_{v,1},\ldots,\sigma_{v,L}$. By \cite[Corollary~2.8]{ABPS}, the local Langlands correspondence preserves depth, so that $t=\max_i\rho(\mathrm{JL}_v(\sigma_{v,i})).$  An application of \cite[Proposition~2.1]{MY} shows that $\mathrm{JL}_v(\sigma_{v,i})^{K_v(p^{t-1})}\neq \{0\}$. Thus we can put $d_{\mathcal{O},v}=t-1$. By our discussion above, this only depends on $\mathcal{O}$, but not on $q$.
\end{proof}

\section{The endgame}\label{sec:cnfm30d8bc}

We first prove Theorem~\ref{th_mt}. Arguing as in \cite[Section~5]{AB} we obtain
\begin{equation}
  \int_{\substack{\Gamma(q)^{\natural}\\ \| \mu_{\varpi} \| \leq M}}\vert A_{\varpi}(N)\vert^2Z^{2\sigma_{\varpi,v}}d\varpi \ll (qZ)^{\epsilon} \sum_{w\in W}\mathcal{S}_{w,N}^{(n)}(Z')\nonumber
\end{equation}
for some $Z' \ll Z$.  As in \cite{AB}, the notation on the left hand side is shorthand for a spectral decomposition of $L^2(\Gamma^{\natural}\backslash {\rm SL}_n(\mathbb{R})/{\rm SO}_n(\mathbb{R}))$ respecting the unramified Hecke algebra, restricted to $\| \mu_{\varpi} \| \leq M$. The $O$-constant on the right hand side is polynomial in $M$.  The geometric sums $\mathcal{S}_{w,N}^{(n)}(X)$ were defined in \eqref{eq:def_S}. At this point, we drop the continuous contribution from the spectral side and apply Proposition~\ref{pro_four_coeff}. This gives
\begin{equation}
  \sum_{\substack{\varpi  \in {\rm ONB}(L^2_{{\rm cusp}}(X(q)))\\ \| \mu_{\varpi} \| \leq M}}Z^{2\sigma_{\varpi,v}} \ll  (qZ)^{\epsilon} \frac{\mathcal{V}_q}{\mathcal{N}_q} \sum_{w\in W}\mathcal{S}_{w,N}^{(n)}(Z').\label{eq_start_gl3_here}
\end{equation}
Note that when applying Proposition~\ref{pro_four_coeff}, we potentially have to replace $q$ by $q_0q$ for some $q_0$ depending only on $n$. Since this shift can always be absorbed into the implicit constants, we omit it from the notation.

From now on, we suppose that $Z\ll q^{n+2-\epsilon}$. Then, by Lemma~\ref{lm:divisibility} we obtain
\begin{equation}
  \sum_{\substack{\varpi  \in {\rm ONB}(L^2_{{\rm cusp}}(X(q)))\\ \| \mu_{\varpi} \| \leq M}}Z^{2\sigma_{\varpi,v}} \ll (qZ)^{\epsilon} \frac{\mathcal{V}_q}{\mathcal{N}_q}\left(\mathcal{S}_{\mathrm{id},N}^{(n)}(Z')+\mathcal{S}_{w_{\ast},N}^{(n)}(Z')\right).\nonumber
\end{equation}
The contribution of $w=\mathrm{id}$ is simply $\mathcal{N}_q$. On the other hand, we have estimated $\mathcal{S}_{w_{\ast},N}^{(n)}(x)$ in Proposition~\ref{pro_geo}. Since $Z\ll q^{n+2-\epsilon}$ the bound obtained there reads
\begin{equation}
  \mathcal{S}_{w_{\ast},N}^{(n)}(Z) \ll q^{\epsilon} \mathcal{N}_q\left(1+\frac{q^{(n+2)(\frac{3}{4}+\frac{1}{4n})}}{q^{(7n-3)/4}}\right) \ll q^{\epsilon} \mathcal{N}_q. \nonumber
\end{equation}
This yields
\begin{equation}
  \sum_{\substack{\varpi  \in {\rm ONB}(L^2_{{\rm cusp}}(X(q)))\\ \| \mu_{\varpi} \| \leq M}}Z^{2\sigma_{\varpi,v}} \ll  (qZ)^{\epsilon} \mathcal{V}_q, \nonumber
\end{equation}
as long as $Z\ll q^{n+2-\epsilon}$. The statement of Theorem ~\ref{th_mt} follows from this via the observation
\begin{equation}
  N_v(\sigma,\mathcal{F}_{\mathrm{cusp}}) \leq Z^{-2\sigma}\sum_{\substack{\varpi  \in {\rm ONB}(L^2_{{\rm cusp}}(X(q)))\\ \| \mu_{\varpi} \| \leq M}}Z^{2\sigma_{\varpi,v}}\ll  (qZ)^{\epsilon} \mathcal{V}_qZ^{-2\sigma}.\nonumber
\end{equation}
The choice $Z\asymp q^{n+2-\epsilon}$ gives the desired result. (Compare \cite[Section~6.4]{AB}, for example.)\footnote{Note that the bottleneck lies not in treatment of $S_{w_{\ast},N}^{(n)}(Z)$, but rather in that, as soon as we allow $Z\geq q^{n+2}$, other Weyl elements will contribute, and we do not know at the moment how to handle their contribution.}

The proofs of Corollary~\ref{cor_disc} and Corollary~\ref{cor_full_spec} are obtained by following the arguments from \cite[Section~7]{AB} verbatim.\footnote{Note that in \cite[Section~7]{AB}, only trivial estimates for the dimensions of $\pi_p^{K_p(p^m)}$ are used. Sharp estimates are now available in \cite{Su}. These can in principle be used to give more precise bounds on the residual spectrum. We have chosen not pursue this here, since the trivial estimates are sufficient for the density hypothesis.} We omit the details.  Corollaries \ref{cor:count} and \ref{cor:lift} follow as in \cite[Section 8]{AB} and \cite{JK}. Note that the only point in \cite{JK} where the squarefreeness is needed is when the density theorem from \cite{AB} is applied.

Let us turn towards Theorem~\ref{th_gl3}. To prove this, we start from \eqref{eq_start_gl3_here}. Note that as usual, the contribution of $w=\mathrm{id}$ is clearly acceptable.  We are thus left with
\begin{equation}
  \sum_{\substack{\varpi  \in {\rm ONB}(L^2_{{\rm cusp}}(X(q)))\\ \| \mu_{\varpi} \| \leq M}}Z^{2\sigma_{\varpi,v}} \ll (qZ)^{\epsilon}\mathcal{V}_q +  \frac{\mathcal{V}_q}{\mathcal{N}_q}\mathcal{S}_{w_1',N}^{(n)}(Z')+\frac{\mathcal{V}_q}{\mathcal{N}_q}\mathcal{S}_{w_1,N}^{(n)}(Z')+\frac{\mathcal{V}_q}{\mathcal{N}_q}\mathcal{S}_{w_l,N}^{(n)}(Z').\nonumber
\end{equation}
Inserting the estimates from Lemma~\ref{gl_3_geo} yields
\begin{equation}
  \sum_{\substack{\varpi  \in {\rm ONB}(L^2_{{\rm cusp}}(X(q)))\\ \| \mu_{\varpi} \| \leq M}}Z^{2\sigma_{\varpi,v}} \ll (qZ)^{\epsilon}\mathcal{V}_q\left(1+ \frac{Z}{q^6} + \frac{Z^{1/2}}{q^4}\right).\nonumber
\end{equation}
This estimate holds without restriction on $Z$, but for $Z\ll q^{6+\epsilon}$, we have 
\begin{equation}
  \sum_{\substack{\varpi  \in {\rm ONB}(L^2_{{\rm cusp}}(X(q))\\ \| \mu_{\varpi} \| \leq M)}}Z^{2\sigma_{\varpi,v}} \ll (qZ)^{\epsilon}\mathcal{V}_q.\nonumber
\end{equation}
The statement of Theorem~\ref{th_gl3} directly follows.

Finally, we turn towards Theorem~\ref{th_cocomp}. Using Proposition~\ref{main_ing_coc} and Corollary~\ref{cor_disc}, we obtain
\begin{equation}
  N_v(\sigma,\mathcal{F}_{\Gamma_{\mathcal{O}}(q)}(M)) \leq N_v(\sigma,\mathcal{F}_{\Gamma(qD_{\mathcal{O}}),\mathrm{disc}}(M)) \ll   M^K \cdot [\mathrm{SL}_n(\mathbb{Z})\colon \Gamma(qD_{\mathcal{O}})]^{1-\frac{2\sigma}{n-1}+\epsilon}.\nonumber
\end{equation}
But this is all we need, since $[\mathrm{SL}_n(\mathbb{Z})\colon \Gamma(qD_{\mathcal{O}})] \asymp_{\mathcal{O}} [\Gamma_{\mathcal{O}}\colon\Gamma_{\mathcal{O}}(q)].$ \\

\textbf{Acknowledgement: }The first author would like to thank Jessica Fintzen and David Schwein for useful disccussions concerning the theory of supercuspidal representations. 

% ----------------------------------------------------------------------------------------
%	BIBLIOGRAPHY
% ----------------------------------------------------------------------------------------

\end{document}